\documentclass[11pt,a4paper]{article}
\usepackage[utf8]{inputenc}
\usepackage{amsmath}
\usepackage{amsfonts}
\usepackage{amssymb}
\usepackage{amsthm}
\usepackage{graphicx}
\usepackage{caption}
\usepackage{tabularx}
\usepackage{subcaption}
\usepackage{bm}
\usepackage[left=2.5cm,right=2.5cm,top=2cm,bottom=3cm]{geometry}
\usepackage{enumitem}
\usepackage{soul}
\usepackage[dvipsnames]{color}
\usepackage{hyperref}
\usepackage{comment}

\definecolor{nb}{RGB}{31,119,180}
\definecolor{no}{RGB}{255,127,14}
\definecolor{ng}{rgb}{0,0.5,0}

\newtheorem*{rem}{Remark}

\newtheorem{prop}{Proposition}[section]
\newtheorem{defi}{Definition}[section]

\usepackage[sort]{natbib}

\usepackage{numbered_constants}

\newtheorem{theo}{Theorem}

\newtheorem{lemma}[prop]{Lemma}

\newtheorem{cor}[prop]{Corollary}

\title{Eco-evolutionary cycles in a matching type predator-prey interaction}
\date{01/09/2026}

\newcommand{\prey}{N^K}
\newcommand{\pred}{H^K}
\newcommand{\eps}{\varepsilon}

\newcommand{\N}{\mathbb{N}}
\newcommand{\R}{\mathbb{R}}
\newcommand{\E}{\mathbb{E}}
\renewcommand{\P}{\mathbb{P}}

\newcommand{\zbf}{\mathbf{z}}

\usepackage{authblk}

\DeclareConstant{C}{C}
\begin{document}

\title{Eco-evolutionary cycles in a matching type predator-prey interaction}
\author[1]{Manon Costa}
\author[2]{Peter Czuppon}
\author[3]{Raphaël Forien}
\affil[1]{Univ Toulouse, INSA Toulouse, CNRS, IMT, Toulouse, France.}
\affil[2]{Aix Marseille Univ, CNRS, I2M, Marseille, France}
\affil[3]{INRAE, BioSP, 84000, Avignon, France}

\maketitle

\begin{abstract}
We study the population dynamics of a predator-prey system with two types in each species. Within a species, predator or prey, dynamics are described by a neutral competitive Lotka-Volterra model, i.e., birth, death and competition parameters are equal for both types. Additionally, we assume that the intra- and inter-type competition parameters are equal. The predator-prey interaction is defined by a matching-types model where predators of type $i$ exclusively interact with prey of type $i$. The individual-based model is described by a birth-death process with immigration, where immigration reflects mutations between the types of the same species. We completely describe the deterministic dynamics arising as a large population limit of this birth-death process. We find that depending on the parameters, potential equilibria are the coexistence of all four types, coexistence of a non-matching or matching pair of predators and prey, or the extinction of the predator or prey species resulting in a line of two-type equilibria. When mutations are sufficiently rare, then the predator-prey dynamics are described by successive jumps between the different deterministic equilibria on this mutational time scale. These jumps describe eco-evolutionary cycles of repeated prey or predator invasions and declines. 
When coexistence of all the types is possible, we show that these cycles accumulate on this time scale.
Lastly, to prove that after the accumulation point the system converges to the coexistence equilibrium, we consider a slightly modified model with unequal intra- and inter-type competition parameters. This modified setting allows us to conclude that after the accumulation point all four populations remain macroscopic and converge to the coexistence equilibrium.
\end{abstract}

\section{Introduction}
Predator-prey models were among the first theoretical descriptions of interactions between different species~\cite{lotka,volterra}. Accordingly, predator-prey interactions (as well as host-parasites interactions) are well documented empirically and abundantly studied theoretically (see \cite{yamamichi2020} for a recent review in the context of co-evolution). 
Most theoretical works study and classify the asymptotic behaviour of systems of ordinary differential equations with multiple, mostly two or three, interacting populations (e.g. \cite{Takeuchi,TA83,TAT78}), but stochastic models exist as well, considering either noisy versions of deterministic systems or individual-based models (e.g. \cite{costa2014,hening2018,golmohammadi2025}). 

In this article, we study the co-evolutionary dynamics of a predator-prey or host-parasite system that are often described as Red Queen dynamics, where species have to adapt continuously to persist \cite{marrow1992coevolution,dieckmann1995evolutionary}.
This phenomenon is usually associated with oscillations in genotype or phenotype due to selection that favours rare types.  These co-evolutionary oscillations are predicted theoretically, e.g~\cite{Schenk}, but are difficult to validate empirically, though some evidence in host-parasite and host-pathogen systems exists~\cite{ebertBio,papkou2018}. Throughout this article, we will refer to predator-prey interactions, but interpretations in terms of host-parasite systems hold as well.  

To describe the changes in genotypes induced by such coevolutionary behaviour, we focus on the case where the predator-prey interaction is described by a matching-allele model~\cite{dybdahl2014}. That is, the two predators are specialized and only prey on a single ``matching'' prey species. This specialized predation behaviour can be due to specific mutations, e.g. the prey develops a toxin that the predator needs to be resistant against~\cite{holding2016}.

Matching-allele models have been applied abundantly to describe cyclic dynamics arising from antagonistic co-evolutionary dynamics, e.g.~\cite{ashby2019,schenk2017,song2015}. These cycles arise from rare types being advantageous over abundant types. That is, the rare prey has an advantage over the abundant prey because its matching predator is rare. Rare predators indirectly benefit from their unmatching prey being controlled by the abundant predator, thus helping the matching prey population to increase. However, these dynamics are exclusively studied in models with fixed and finite population sizes, or in models with specialist predators, i.e., predators do not have alternative resources to feed from and are not in competition with each other; for recent reviews of different co-evolutionary models studied in the theoretical biology literature we refer to~\cite{buckingham2022} and \cite{Schenk}. 

Here, we model the dynamics as a multi-type birth and death process with logistic competition and predator-prey interaction. We consider two types of prey and two types of predators which are neutral except for the predator-prey interaction. We consider that individuals can mutate to the other type or switch their type at birth with a probability which decreases as a power of the population size. 
We study the large population size limit of this stochastic predator-prey model with two predators and two prey populations including mutations at a rapid scale between the different phenotypes of the populations. 
Such a scaling, which has been introduced by Durett and Mayberry \cite{durrett_traveling_2011} and popularized by \cite{bovier_crossing_2019, champagnat_stochastic_2021,coquille_stochastic_2021}, allows one to follow the dynamics of populations with different sizes: macroscopic populations or \textit{resident populations} will be compared to solutions of differential equations, while microscopic population invasions or decays will be studied using comparisons with branching processes. The main novelty of our work is to consider a biological system with different species interacting with each other, which requires complex couplings taking positive and negative interactions into account, and the fact that we also characterise the limiting behaviour of the system (or a slightly modified version) near and after the accumulation point of these invasion times.

Previous works in the context of Red-Queen dynamics typically find stable oscillations over time in the absence of intraspecific competition. In contrast, by including intraspecific competition we find that cyclic dynamics for phenotypes cannot be maintained indefinitely. Instead, the possible outcomes of our model are convergence to an equilibrium with all four species or convergence to equilibria with only the two prey or predator species.

Interestingly, the trajectories leading to these equilibria may exhibit transient cyclic dynamics. These cycles correspond to successive invasion attempts of predator or prey species into a resident population. An invasion attempt will change the proportions of the residents, which then triggers a new invasion event. This process repeats itself until either all species are at a macroscopic scale (of some order $K$) or until the proportion of the resident species exceeds a critical value so that no further invasion of a non-macroscopic species is possible. 
We will identify the asymptotic limits of the system and provide conditions under which these are attainable.

\section{Matching types predator-prey model}
We consider two populations with a predator-prey interaction, where each population is composed of two types of individuals, denoted by $0$ and $1$. We assume interactions according to a matching alleles model, that is, predators of type $ i \in\{0, 1\}$ thrive in the presence of prey of type $ i $, while prey of type $ j\in\{0, 1\} $ are more likely to survive in the absence of predators of type $ j $.

Each population evolves according to a sequence of birth and death events, whose respective rates depend on the state of the total population. Each individual produces offspring of its own type most of the time, but can on some rare occasion produce offspring of the other type.
The four-dimensional stochastic system will be indexed by a parameter $ K > 0 $ which corresponds to the order of the total prey population size. The predator population size will be proportional to $ K^{m} $ for some $ m >0 $. The parameter $m$ allows to consider population with different typical sizes. The case where predators are rarer than prey $m<1$ corresponds to what is expected at the first sight, but the opposite case is also relevant for example when considering prey as trees and predators as insects or parasites (see Robinson and al. \cite{Umea} for the study of Aspen canopy and its arthropod community or Ludwig and al. \cite{ludwig1978qualitative} for the interaction between spruce bud-worm and the forest). From a biological perspective, this scaling can be associated with the metabolic theory. The metabolic theory links the mass of individuals with their metabolic rates. Numerous experimental studies display relationships between the individual mass and the birth and death rates or the community carrying capacity (see Brown and al. \cite{brown2004toward}, Damuth \cite{Damuth81}). 

\subsection{Individual-based model}
	
For $ t \geq 0 $ and $ i \in \lbrace 0, 1 \rbrace $, let $ \prey_i(t) $ (resp. $ \pred_i(t) $) denote the number of prey (resp. predators) of type $ i $ alive at time $ t $.
We also set
\begin{align*}
	\prey(t) = \prey_0(t) + \prey_1(t), && \text{ and } && \pred(t) = \pred_0(t) + \pred_1(t).
\end{align*}

Each prey of type $ i \in \lbrace 0, 1 \rbrace $ produces new offspring at rate $ b > 0 $ and dies at rate 
\[ d + c_K \prey(t) + p_K \pred_i(t) \]
where $ d > 0 $ denotes the natural death rate, $ c_K > 0 $ the competition and $ p_K > 0 $ the predation rate.
Each new prey is either of the same type as its parent (with probability $ 1-v_K $), or of the other type (with probability $ v_K $), for some $ v_K \in (0,1) $.

Each predator of type $ j \in \lbrace 0, 1 \rbrace $ produces new offspring at rate \[ \beta + \rho_K \prey_j(t) \] that includes natural births at rate $\beta>0$ and the positive effect from predation at rate $\rho_K>0$. Each predator dies at rate $ \delta + \gamma_K \pred(t) $, composed of a natural death rate $ \delta> 0 $ and a logistic competition rate $ \gamma_K > 0 $.
Moreover, each new predator inherits the type of its parent with probability $ 1-\vartheta_K $, or mutates to the other type with probability $ \vartheta_K $, for some $ \vartheta_K \in (0,1) $.

Hence, we observe that all the prey (resp. predators) are in competition with one another, irrespective of their type, since the death rates increase with the size of the total prey (resp. predator) population. Moreover, in this model the type has no influence on the ecological parameters $b,d,c,\beta, \delta$ and $\gamma$. This specific choice allows to observe a variety of behaviours and will induce mathematical challenges due to the large amount of symmetry.
We also see that the death rate of prey of type $ i $ increases with the size of the type $ i $ predator population, while the birth rate of predators of type $ j $ increases with the size of the type $ j $ prey population.
\medskip

The expected behaviour of this process is as follows.
If the predator population is mostly of a single type, say 0, then type 1 prey have a lower death rate than type 0 prey, and as a consequence the type 1 prey population should increase in size at the expense of the type 0 prey population.
When the prey are mostly of type 1, however, type 1 predators have a higher birth rate than predators of type 0, hence type 1 predators should replace predators of type 0, thus shifting the advantage in the prey population in favour of type 0, leading to another invasion of type 0 prey, \textit{etc.}
The aim of this study is to investigate how the parameters of the model determine the interval between these ``switches''.
\medskip

We denote by $ (Z^K(t), t \geq 0) $ the pure jump Markov process taking values in $ \N^4 $ that describes the behaviour of the total population:
	\begin{align*}
		Z^K(t) = (\prey_0(t), \prey_1(t), \pred_0(t), \pred_1(t)),
	\end{align*}
where the birth and death rates of each population are summarized in Table~\ref{table:rates}. 

We assume that parameters depend on $K$ as follows:
	\begin{align*}
	c_K = \frac{c}{K}, && \gamma_K = \frac{\gamma}{K^{m}}, && p_K = \frac{p}{K^m}, && \rho_K = \frac{\rho}{K},
	\end{align*}
	where $ c $, $ \gamma $, $ p $, $ \rho $ are positive constants, $ m > 0 $ .
	For clarity and when possible, we use Latin and Greek letters for parameters referring to prey and predator populations, respectively. \\
These assumptions imply that the typical size of the prey (resp. predator) population is $K$ (resp $K^m$). We will therefore say that a prey (resp. predator) population is at a macroscopic level when its size is of the order of $K$ (resp. $K^m$), and is at a microscopic level otherwise. The parameter $m$ influences the scaling of the population dynamics.

We also assume that the mutation rates scale with $K$ as
\begin{align*}
v_K = \frac{1}{K^v}, && \vartheta_K = \frac{1}{K^{m\vartheta }},
\end{align*}
and $ v $ and $ \vartheta $ are in $ (0,1) $.
This scaling corresponds to relatively rapid mutations since when a population is macroscopic, the number of mutations stemming from this population will be of the order of $K^{1-v}$ or $K^{m(\vartheta-1)}$ which grows to $\infty$ as $K\to\infty$, but not as fast as $K$ (or $K^m$).
\begin{table}[h]
\centering
\renewcommand{\arraystretch}{2}
\begin{tabularx}{\textwidth}{| l | X | X |}
\hline
\textbf{Population} & \textbf{Total birth rate} & \textbf{Total death rate} \\
\hline
$ \prey_i $, $i \in \lbrace 0,  1 \rbrace$ & $ b (1-v_K) \prey_i(t) + b\, v_K \prey_{1-i}(t) $ & $ (d + c_K\prey(t)+ p_K \pred_i(t)) \prey_i(t) $ \\
\hline
$ \pred_i $, $i \in \lbrace 0,  1 \rbrace$ & $ (\beta + \rho_K \prey_i(t)) (1-\vartheta_K) \pred_i(t) + \vartheta_K (\beta + \rho_K \prey_{1-i}(t)) \pred_{1-i} $ & $ (\delta + \gamma_K \pred(t)) \pred_i(t) $ \\
\hline
\end{tabularx}
\caption{Birth and death rates of the different populations: prey, predators.
}
\label{table:rates}
\end{table}

\subsection{Large population limit} 
\label{subsec:dynamical_system_results}

In the following analysis, we will often assume that the initial population sizes $\prey_0(0),\prey_1(0)$ (resp.  $(\pred_0(0),\pred_1(0)$) are of order $K$ (resp. $K^m$).
As a consequence, we introduce a rescaled stochastic process 
\begin{equation} \label{def:Zsc}
	(\bm{Z}^K_{sc}(t),t\geq0)=\left(\frac{\prey_0(t)}{K},\frac{\prey_1(t)}{K}, \frac{\pred_0(t)}{K^m}, \frac{\pred_1(t)}{K^m}\right),
\end{equation}
which will be comparable to a solution of the dynamical system
\begin{equation}\label{eq:syst-4}
\left\{\begin{aligned}
&\frac{dn_0(t)}{dt}= n_0(t)(b-d-c(n_0(t)+n_1(t))-ph_0(t))\\
&\frac{dn_1(t)}{dt}= n_1(t)(b-d-c(n_0(t)+n_1(t))-ph_1(t))\\
&\frac{dh_0(t)}{dt}=h_0(t)(\beta-\delta-\gamma (h_0(t)+h_1(t)) +\rho n_0(t))\\
&\frac{dh_1(t)}{dt}=h_1(t)(\beta-\delta-\gamma (h_0(t)+h_1(t)) +\rho n_1(t))\\
\end{aligned}\right.
\end{equation}
Let us denote by 
$$(\bm{z}^{(\bm{z}^0)}(t),t\geq0)=(n_0(t), n_1(t), h_0(t), h_1(t))_{t\geq0},$$ 
the unique solution to system~\eqref{eq:syst-4} starting from 
$\bm{z}^{(\bm{z}^0)}(0)=\bm{z}^0 \in \R_+^4$. 
Existence and uniqueness of such a solution follow from the fact that the vector field is locally Lipschitz and that the solutions do not explode in finite time \cite{chicone2006ode}. 
We have the following classical approximation result from Theorem 2.1 p.456 in \cite{Ethier-Kurtz}. 

\begin{lemma}\label{lemapprox}
 Fix $T > 0$. Assume that the sequence $(\bm{Z}^K_{sc}(0),K \geq 1)$ converges in probability when $K \to \infty$ to a deterministic vector ${\bm{z}^0} \in \R_+^4$.
 Then 
\begin{equation}\label{EK2}
\underset{K \to \infty}{\lim}\  \sup_{s\in [0,T]}\ \|\bm{Z}^K_{sc}(s)-\bm{z}^{(\bm{z}^0)}(s)  \|=0 \quad \text{in probability},
\end{equation}
where $\| \cdot \|$ denotes the $L^\infty$ norm on $\R^4$.
\end{lemma}

However, in the scaling limit we will consider, not all populations will always be macroscopic (i.e. of the order of $K$ for prey and of the order of $K^m$ for predators).
In these situations, we will need to study the trajectories of \eqref{eq:syst-4} when some populations are absent.
We therefore briefly present the cases that will be of interest in the rest of the article.

\subsubsection{Two types systems}
\paragraph{Non-matching predator-prey system.}
Let us assume that $\bm{Z}^K_{sc}(0)$ converges as $K\to\infty$ to a vector with only a positive number of one type of prey and its non-matching predator, say $(n_0(0), 0,0, h_1(0))$.
Then the limiting system that describes the behaviour of $(n_0(t), h_1(t))$ consists of two independent logistic equations:
\begin{equation}\label{eq:2-non-matching}
\left\{\begin{aligned}
&\frac{dn_0(t)}{dt}= n_0(t)(b-d-cn_0(t))\ ,\\
&\frac{dh_1(t)}{dt}=h_1(t)(\beta-\delta-\gamma h_1(t) )\ .
\end{aligned}\right.
\end{equation}
We will denote by $\bar{n}$ and $\bar{h}$ the non-zero equilibria of these equations, namely
$$(\bar{n},\bar{h}):=\left( \frac{b-d}{c},\frac{\beta-\delta}{\gamma}\right)\, .$$

\begin{prop} \label{prop:non-matching}
Assume that  $b - d > 0$.
Then, if $n_0(0) > 0$ and $h_1(0) > 0$, the solution of \eqref{eq:2-non-matching} converges as $t\to\infty$ to
\begin{itemize}
\item $(\bar{n},\bar{h})$ if $\beta-\delta>0$,
\item $(\bar{n}, 0)$ otherwise.
\end{itemize}
\end{prop}
\noindent The proof of the above result is elementary. In all the analyses that follow, we will always assume that $b-d>0$ so that $\bar{n}>0$.

We note that in most of the theoretical literature on predator-prey models the difference $\beta-\delta$ is negative, i.e., the predator cannot exist without the prey. A value $\beta-\delta$ larger than zero corresponds to a situation where the predator can feed on alternative resources and is referred to as a generalist~\cite{hanski1991}. We mainly study here the case where $\beta-\delta>0$ which presents the most interesting dynamics.
 
\paragraph{Matching predator-prey system.}
Let us assume that in the limit, only one prey population and its matching predator are present, for example $\bm{Z}^K_{sc}(0)$ converges to $(n_0(0), 0, h_0(0),0)$ for some $n_0(0), h_0(0)$ positive.
Then the limiting system describing the behaviour of $(n_0(t), h_0(t))$ is a predator-prey Lotka-Volterra system with competition 
\begin{equation}\label{eq:2-matching}
\left\{\begin{aligned}
&\frac{dn(t)}{dt}= n(t)(b-d-cn(t)-ph(t))\ ,\\
&\frac{dh(t)}{dt}=h(t)(\beta-\delta-\gamma h(t) +\rho n(t))\ .
\end{aligned}\right.
\end{equation}
This system admits four equilibria in $\R^2$:
$$(0,0),\quad \left(0,\frac{\beta-\delta}{\gamma}\right), \quad \left(\frac{b-d}{c},0\right),$$
and 
\begin{equation}
\label{eq:eq_matching}
(\hat{n},\hat{h}) :=\left( \frac{(b-d)\gamma -p(\beta-\delta)}{p\rho +c\gamma},\frac{(b-d)\rho+c(\beta-\delta)}{p\rho +c\gamma}\right) .
\end{equation}
This coexistence equilibrium can be rewritten as
\begin{align} \label{eq:n_hat_h_hat}
    \hat{n} = \frac{\bar{n} - \frac{p}{c} \bar{h}}{1 + \frac{p \rho}{c \gamma}}, && \hat{h} = \frac{\bar{h} + \frac{\rho}{\gamma} \bar{n}}{1 + \frac{p \rho}{c \gamma}}.
\end{align}

The coexistence equilibrium is feasible if the following two conditions hold
\begin{equation}\label{eq:cond_equilibria_matching}
\hat{n}>0\iff \bar{h} < \frac{c}{p} \bar{n}\,,\quad\quad  \hat{h}>0 \iff \bar{h} > -\frac{\rho}{\gamma} \bar{n}\, .
\end{equation}

\begin{prop}
\label{prop:cv_sys_matching}
Let us assume $b-d>0$.
\begin{enumerate}[label=\roman*)]
\item If $\bar{h} \leq - \frac{\rho}{\gamma} \bar{n} $, for any positive initial condition the solution of system \eqref{eq:2-matching} converges to $(\bar{n},0)$ as $ t \to \infty $.
\item If $ - \frac{\rho}{\gamma} \bar{n} < \bar{h} < \frac{c}{p} \bar{n} $ (i.e. \eqref{eq:cond_equilibria_matching} is satisfied), then for any positive initial condition, the solution of system \eqref{eq:2-matching} converges to 
$(\hat{n},\hat{h}) $ as $ t \to \infty $.
\item Finally, if $ \bar{h} \geq \frac{c}{p} \bar{n} $, for any positive initial condition, the solution of system \eqref{eq:2-matching} converges to $(0,\bar{h})$ as $ t \to \infty $.
\end{enumerate}
\end{prop}
\noindent
The proof is given in Section~\ref{proof:2SpecMatching}.

\subsubsection{Four types system}
We return to the case where all four types are present in the limit that is described by system~\eqref{eq:syst-4}. 
If a positive equilibrium $(n_0^*,n_1^*,h_0^*,h_1^*) \in (\R_+^*)^4$ exists then necessarily $n_0^*=n_1^*=n^*$ and $h_0^*=h_1^*=h^*$ and 
$$(n^*,h^*) :=\left( \frac{(b-d)2\gamma -p(\beta-\delta)}{p\rho +4c\gamma},\frac{(b-d)\rho+2c(\beta-\delta)}{p\rho +4c\gamma}\right) .
$$
The other equilibria of this system are
\begin{equation*}
    (0,0,0,0), \quad (a\bar{n}, (1-a)\bar{n}, 0,0), a\in[0,1], \quad (0,0, \alpha\bar{h}, (1-\alpha)\bar{h}), \alpha\in[0,1], 
\end{equation*}
as well as the previously described equilibria $(\bar{n}, 0,0,  \bar{h})$, $(0, \bar{n},  \bar{h},0)$, $(\hat{n},0, \hat{h},0)$ and $(0,\hat{n},0, \hat{h})$.
Note that we can also write
\begin{align} \label{eq:n_star_h_star}
    n^* = \frac{2\bar{n} - \frac{p}{c} \bar{h}}{4 + \frac{p \rho}{c \gamma}}, && h^* = \frac{2\bar{h} + \frac{\rho}{\gamma} \bar{n}}{4 + \frac{p \rho}{c \gamma}}.
\end{align}

The feasibility conditions for the coexistence equilibrium $(n^*, n^*, h^*, h^*)$, where all coordinates are positive, read
\begin{equation}
\label{eq:cond_4_types}
n^*>0\iff \bar{h} < \frac{2 c}{p} \bar{n}\,,\quad\quad
h^*>0\iff \bar{h} > - \frac{\rho}{2 \gamma} \bar{n}\, .
\end{equation}
Note that condition \eqref{eq:cond_4_types} is less restrictive than condition  \eqref{eq:cond_equilibria_matching} for the coexistence of matching prey and predators, independent of the sign of $\beta-\delta$.
The different regions (in terms of the value taken by $\bar{h}$ relative to $\bar{n}$) where the different equilibria exist are shown in Figure~\ref{fig:cases}.

\begin{prop}
\label{prop:cv_syst_4}
Assume that the inequalities in \eqref{eq:cond_4_types} hold, then every solution of \eqref{eq:syst-4} starting from a positive initial condition converges to the positive equilibrium $(n^*, n^*, h^*, h^*).$
\end{prop}
\noindent
The proof is given in Section~\ref{proof:4SpeciesCoexist}.

\subsubsection{Three types systems}
The final cases of interest correspond to situations where three populations temporarily coexist in the large population limit: either the two prey and a predator, or the two predators and a prey.

\paragraph{Two prey, one predator.}
We first study the dynamics of system~\eqref{eq:syst-4} when initialized with two prey and one predator population. Then if $\bar{h}>0$ the system converge to the non-matching equilibrium $(\bar{n},0,0,\bar{h})$ as $t\to\infty$.

\begin{prop}
    \label{prop:cv_3types}Let $n_0(0), n_1(0), h_0(0)>0$. Assume $\beta-\delta>0$. Then, the solution of \eqref{eq:syst-4} started from $(n_0(0), n_1(0), h_0(0), 0)$ converges to $(0,\bar{n}, \bar{h},0)$ as $t\to\infty$.
\end{prop}
\noindent
The proof is given in Section~\ref{proof:3Species}. 

\paragraph{One prey, two predators.}
Lastly, we gather results on the one prey-two predator dynamical system:
\begin{equation}\label{eq:syst-3}
\left\{\begin{aligned}
&\frac{dn_0(t)}{dt}= n_0(t)(b-d-c n_0(t))-ph_0(t))\\
&\frac{dh_0(t)}{dt}=h_0(t)(\beta-\delta-\gamma (h_0(t)+h_1(t)) +\rho n_0(t))\\
&\frac{dh_1(t)}{dt}=h_1(t)(\beta-\delta-\gamma (h_0(t)+h_1(t)) ).\\
\end{aligned}\right.
\end{equation}

We will first show that if the matching equilibrium exists, i.e., if condition \eqref{eq:cond_equilibria_matching} holds, then the solution of this system converges to the matching types equilibrium as $t \to \infty$.

\begin{prop}
\label{prop:syst-3-matching}
Assume that \eqref{eq:cond_equilibria_matching} holds. Then the solution of \eqref{eq:syst-3} starting from $(n_0(0), h_0(0), h_1(0))$ with $n_0(0)$, $h_0(0)$, $h_1(0)>0$ converges to $(\hat{n},\hat{h},0)$ as $t\to\infty$.
\end{prop}
\noindent
The proof is given in Section~\ref{proof:3types-matching}.
\smallskip

Lastly, we will show that if $\bar{h} > 0 $, i.e., $\beta-\delta>0$, and if condition~\eqref{eq:cond_equilibria_matching} is violated, i.e., the matching equilibrium $\hat{n}$ is negative, this system admits a line of equilibria of the form $(0, \alpha\bar{h}, (1-\alpha)\bar{h})$ for $\alpha\in[0,1]$ and no equilibrium with the three populations. 
Additionally, we aim to characterize the dynamics of the proportion of type $0$ predators 
\begin{equation}
\label{def:alpha_t}
\alpha(t):=\frac{h_0(t)}{h_0(t)+h_1(t)}.
\end{equation}
The dynamics of this proportion are
\begin{equation}
\label{eq:alpha_t}
\frac{d\alpha(t)}{dt} = \rho n_0(t) \alpha(t)(1-\alpha(t))\, .
\end{equation}
\begin{prop}
\label{prop:syst-3}
Assume $\beta-\delta>0$.
\begin{enumerate}[label=\roman*)]
\item The Jacobian matrix at the equilibrium $(0, \alpha \bar{h}, (1-\alpha)\bar{h})$ admits two negative eigenvalues and a null eigenvalue if and only if $\alpha> \alpha_c$, where $\alpha_c$ is the critical proportion of type 0 predators below which the type 0 prey can invade: 
\begin{equation}
\label{eq:alpha_c_first}
	\alpha_c := \sup\{\alpha: b-d-p\alpha\bar{h} > 0\}= \frac{c \bar{n}}{p\bar{h}} \, .
\end{equation}
\item Assume $\hat{n}<0$. Let us consider the solution of \eqref{eq:syst-3} with initial condition $$(n_{0,0}, \alpha_0 \bar{h}, (1-\alpha_0)\bar{h})$$ such that $n_{0,0}>0$ and $\alpha_0<\alpha_c$. Then the solution converges as $t\to\infty$ to $(0, \alpha_\infty\bar{h}, (1-\alpha_\infty)\bar{h})$, with  $\alpha_\infty>\alpha_c$.\\
Moreover the limiting proportion $\alpha_\infty$ can be written as a function $\kappa(n_{0,0},\alpha_0)$ of the initial conditions which satisfies that
\begin{equation}\label{def:kappa_0}
\kappa_0(\alpha_0):=\lim_{n_{0,0}\to0} \kappa(n_{0,0}, \alpha_0),
\end{equation} 
exists and $ \kappa_0(\alpha_0)>\alpha_c$ for $\alpha_0 \neq \alpha_c$ and $\kappa_0(\alpha_c) = \alpha_c$.
\item If furthermore $n^*\le0$, or equivalently $\alpha_c<1/2$, we define the function $f:[0,1]\to\R_+$ by 
\begin{equation}
    \label{def:f}
    f(\alpha) := -\alpha_c \log(\alpha) -(1-\alpha_c) \log (1-\alpha)
,
\end{equation} and the positive constant $M = (1-2\alpha_c)\log((1-\alpha_c)/\alpha_c)$. For any $\alpha_0<\alpha_c$, the limiting proportion $\kappa_0(\alpha_0)$ satisfies 
$$f(\kappa_0(\alpha_0))\le f(1-\alpha_0)-M $$ and in particular $ \kappa_0(\alpha_0)<1-\alpha_0$.
\end{enumerate}
\end{prop}
\noindent
The proof and detailed construction of the function $\kappa$ is given in Section~\ref{proof:PredProportion}.\smallskip

As expected, the null eigenvalue in point $i)$ is associated with the eigenvector $(0,1,-1)$ corresponding to the line of equilibria. In point $ii)$, the matching equilibrium does not exist and any invasion attempt by a prey will fail as the proportion of matching predators will eventually overcome the critical proportion $\alpha_c$.
We will see below that this regime will lead to successive invasion attempts by each prey.
When $\alpha_c > 1/2$, which corresponds to the case when the four-species coexistence equilibrium exists and is locally stable, these invasions will take place more and more rapidly and eventually accumulate.
When $\alpha_c < 1/2$, however, we will see that, after a finite number of failed prey invasions, the proportion of each predator type will be in $[\alpha_c, 1-\alpha_c]$, at which point both prey populations go extinct.

\section{Limit at the mutation scale}

\subsection{Notations associated to the deterministic system}
Consider the dynamical system \eqref{eq:syst-4}, and let $ \mathcal{E} \subset \R_+^4 $ be the set of  non negative fixed points of this system,
\begin{equation*}
	\mathcal{E}: = \left\lbrace z = (n_0, n_1, h_0, h_1) \in \R_+^4 : \quad \begin{aligned}
		&n_i (b-d - c(n_0 + n_1) - p h_i) = 0, \\
		&h_i (\beta - \delta - \gamma(h_0 + h_1) + \rho n_i) = 0,
	\end{aligned} \quad i \in \lbrace 0, 1 \rbrace \right\rbrace.
\end{equation*}
We then define, for $ z \in \mathcal{E} $ and $ i \in \lbrace 1, 2, 3, 4 \rbrace $,
\begin{equation*}
	\Psi_i(z) := \lim_{\varepsilon \downarrow 0} \lim_{t \to \infty} \bm{z}^{(z + \varepsilon e_i)}(t),
\end{equation*}
where $ e_i $ is the $ i $-th vector of the canonical basis of $ \R^4 $.
That is, $ \Psi_i(z) $ is the equilibrium of the system after an arbitrarily small quantity of the $ i $-th population is introduced at the beginning, starting from state $ z $.
By definition, $ \Psi_i(z) \in \mathcal{E} $ for any $ z \in \mathcal{E} $ and $ i \in \lbrace 1, 2, 3, 4 \rbrace $.
Note also that, if $ z $ is locally stable, then $ \Psi_i(z) = z $ for any $ i $.

We can then reformulate some results from the previous section in terms of the function $\Psi_i$. 
Let us note that the set of all possible positive equilibria depends on the parameters (see Figure~\ref{fig:cases}). 
We can nonetheless identify a maximal set:
\begin{multline*}
    \mathcal{E} \subset \{(0,0,0,0) , (\bar{n},0,0,\bar{h}), (0,\bar{n},\bar{h},0), (\hat{n},0,\hat{h},0), (0,\hat{n},0,\hat{h}), (n^\ast,n^\ast,h^\ast,h^\ast),\\
    (a \bar{n}, (1-a)\bar{n},0,0), a \in [0,1], (0,0,\alpha \bar{h},(1-\alpha)\bar{h}), \alpha\in[0,1]\}\, .
\end{multline*}
When $\hat{n}>0$ and $\bar{n}>0$ (case A below), we will mostly study perturbations of the matching and non-matching equilibria. For the non-matching equilibrium we have
\begin{equation*}
    \Psi_i(\bar{n},0,0,\bar{h}) = \begin{cases}
    (\hat{n},0,\hat{h},0) & \text{ if } i = 3, \\
    (\bar{n},0,0,\bar{h}) & \text{ otherwise,}
    \end{cases}
\end{equation*}
which translates the fact that in a non-matching predator-prey equilibrium, the matching predator can replace the non-matching one.
Similarly, for the matching types predator-prey equilibrium, 
\begin{equation*}
    \Psi_{i}(\hat{n},0,\hat{h},0) = \begin{cases}
    (0,\bar{n},\bar{h},0) & \text{ if } i = 2, \\
    (\hat{n},0,\hat{h},0) & \text{ otherwise,}
    \end{cases}
\end{equation*}
which corresponds to the invasion of the non-matching prey.\\
Finally, when $\hat{n}<0$, the results of Proposition \ref{prop:syst-3} can be written as
\begin{equation*}
    \Psi_1(0,0,\alpha_0\bar{h}, (1-\alpha_0)\bar{h})=(0,0,\alpha_{\infty}\bar{h}, (1-\alpha_\infty)\bar{h}),
\end{equation*}
where 
\begin{equation*}
    \alpha_\infty = \begin{cases}
    \alpha_0 & \text{ if  } \alpha_0 > \alpha_c,\\
    \kappa_0(\alpha_0) & \text{ otherwise,}
    \end{cases}
\end{equation*}
where the mapping $\kappa_0$ is defined in Proposition~\ref{prop:syst-3}.

\subsection{Definition of the limiting process}

Our goal is to study the successive invasions of prey and predators that might happen due to mutations. To do so, we extend the reasoning developed in \cite{bovier_crossing_2019,champagnat_stochastic_2021,coquille_stochastic_2021} and derive the limiting behaviour (as $K \to \infty$) of the population processes considered on time scales of the order of $\log(K)$.
Let us recall an important result on the dynamics on the $\log(K)$ scale of branching processes with immigration. We borrow the notations of \cite{champagnat_stochastic_2021}. 
\begin{defi}
    \label{def:BPI}
We will say that $(A^K_t)_{t\ge0}$ is a linear branching processes with immigration  $BPI_K(b,d,a,c,f,\beta)$ if $A^K_0=\lfloor K^\beta \rfloor$, the birth rate is $b\ge0$, the death rate is $d\ge0$ and the immigration rate at time $s\ge0$ is $fK^c e^{as}$.
\end{defi} 
We recall below the main result that will be useful for our proofs, which specifies the dynamics of this process on the $\log(K)$ time scale. This result is based on a precise study of the exponential martingale associated with branching processes.
\begin{theo}[Theorem B.1 in \cite{champagnat_stochastic_2021}]
\label{theo:convBPI}
Let $(A^K_t)_{t\ge0}$ be a $BPI_K(b,d,a,c,f,\beta)$ and denote by $r=b-d$ its growth rate.
Assume $c\le\beta$ and $\beta>0$. Then for all $T>0$ such that 
$$\inf_{t\in[0,T]} (\beta+rt)\vee (c+at) >0,$$
then the process $\left(  \frac{\log(1+A^K_{s\log(K)})}{\log(K)},0\le s\le T\right)$ converges to $\left( (\beta+rs)\vee (c+as), 0\le s \le T  \right)$ when $K \to \infty$ in probability in $L^\infty ([0,T])$. 
\end{theo}

Let us introduce the logarithmic exponents 
\begin{equation}
\label{def:XY}
X^K_i(t) := \frac{\log(1 + \prey_i(t))}{\log(K)}\,, \qquad Y_i^K(t) := \frac{\log(1+\pred_i(t))}{m \log(K)}, \quad  i\in\{0,1\}.
\end{equation}
This means that $X^K_i(t) \geq 0$ and $Y^K_i(t) \geq 0$ are such that
\begin{align*}
    \prey_i(t) = K^{X^K_i(t)} - 1, && \pred_i(t) ={ K^{mY^K_i(t)}} - 1.
\end{align*}
Also recall the definition of the vector of scaled population sizes $ \bm{Z}^K_{sc}(t) $ in Eq.~\eqref{def:Zsc}.
\medskip

We now prepare to state the convergence of the population exponents
\begin{equation*}
	(X^K_0(t \log(K)), X^K_1(t \log(K)), Y^K_0(t \log(K)), Y^K_1(t \log(K)))_{t\ge0}
\end{equation*}
to a deterministic limit $(x_0(t), x_1(t), y_0(t), y_1(t)) $, which we define recursively on successive time intervals corresponding to the asymptotic intervals between invasions of different populations (prey or predators).
These successive invasions will lead to different intermediate equilibria depending on the parameters of the birth and death events, and these equilibria will determine the subsequent evolution of the limiting exponents $ (x_0, x_1, y_0, y_1) $.

Before stating our convergence result, let us define the limiting exponents and the intermediate equilibria as follows.

\begin{defi} \label{def:limit}
	Given $ (x_0(0), x_1(0), y_0(0), y_1(0)) \in [0,1]^4 $ and $ \widetilde{\bm{z}}_0 = (\widetilde{n}_{0}^{(0)}, \widetilde{n}_{1}^{(0)}, \widetilde{h}_{0}^{(0)}, \widetilde{h}_{1}^{(0)}) \in \mathcal{E} $ such that
	\begin{equation} \label{condition_initial_exponents}
		x_i(0) < 1 \iff \widetilde{n}_{0,i} = 0, \quad \text{ and } \quad y_i(0) < 1 \iff \widetilde{h}_{0,i} = 0.
	\end{equation}
	We define $ (x_0(t), x_1(t), y_0(t), y_1(t), t \in [0, \tau_*)) $ and two sequences $ (s_k, k \geq 0) $, $ (\widetilde{\bm{z}}_k, k \geq 0) $ as follows, where $ s_k \in \R_+ \cup \lbrace + \infty \rbrace $ is an increasing sequence of times and $ \widetilde{\bm{z}}_k \in \mathcal{E} $ for all $ k \geq 0 $ such that $ s_k < \infty $.
	Suppose that the two sequences have been defined up to $ s_k < \infty $ and that $ \widetilde{\bm{z}}_k = (\widetilde{n}_{0}, \widetilde{n}_{1}, \widetilde{h}_{0}, \widetilde{h}_{1}) $.
	Then, for $ t \in [s_k, s_{k+1} \wedge \tau_*) $, we set
	\begin{align*}
		x_i(t) &= \sup_{s \in [s_k, t] : x_i(s) > 0} \lbrace x_i(s) + r_{i}(\widetilde{\bm{z}}_k) (t-s) \rbrace \vee (x_{1-i}(t) - v) \vee 0, \\
		y_i(t) &= \sup_{s \in [s_k, t] : y_i(s) > 0} \lbrace y_i(s) + \varrho_i(\widetilde{\bm{z}}_k) (t-s) \rbrace \vee (y_{1-i}(t) - \vartheta) \vee 0,
	\end{align*}
	where the slopes are defined by	
\begin{equation}
\label{eq:slopes}
r_i(\widetilde{\bm{z}}) := b-d - c(\widetilde{n}_0 + \widetilde{n}_1) - p \widetilde{h}_i \quad \text{ and } \quad
\varrho_i(\widetilde{\bm{z}}) := \frac{1}{m} (\beta-\delta - \gamma(\widetilde{h}_0 + \widetilde{h}_1) + \rho \widetilde{n}_i).
\end{equation}
	We then define $ s_{k+1} $ as the first time following $ s_k $ at which one exponent reaches 1 with a positive slope (which corresponds to an invasion of the type in question), and $ \widetilde{\bm{z}}_{k+1} $ is the new equilibrium following this invasion,
	\begin{align*}
s_{k+1} &= \inf \left\lbrace t > s_k : \exists i \in \lbrace 0, 1 \rbrace : x_i(t) \geq 1 \text{ and } r_i(\widetilde{\bm{z}}_k) > 0 \text{ or } y_i(t) \geq 1 \text{ and } \varrho_i(\widetilde{\bm{z}}_k) > 0 \right\rbrace, \\
\widetilde{\bm{z}}_{k+1} &= \Psi_j(\widetilde{\bm{z}}_k),
	\end{align*}
	where $ j \in \lbrace 1, 2, 3, 4 \rbrace $ corresponds to the type invading at time $ s_{k+1} $ (the one whose exponent reaches 1 with a positive slope).
	If $ s_k = \infty $, we then set $ s_j = \infty $ and $ \widetilde{\bm{z}}_{j} = \widetilde{\bm{z}}_k $ for all $ j \geq k $.\medskip

\noindent We stop the construction at $ t = \tau_* $ if one of the following happens:
	\begin{itemize}
		\item two exponents reach 1 with a positive slope at the same time,
		\item one exponent reaches 1 with a positive slope at the same time as another reaches 0.
	\end{itemize}
	If none of the above ever takes place, we set $ \tau_* =  \lim_{k \to \infty} s_k \in (0, \infty] $. 

	We can then also define $ (\widetilde{\bm{z}}(t), t \in [0, \tau_*)) $ as
	\begin{equation*}
		\widetilde{\bm{z}}(t) = 
			\widetilde{\bm{z}}_k  \text{ if } t \in [s_k, s_{k+1}), 
	\end{equation*}
\end{defi}

We note that all the types for which $ \widetilde{\bm{z}}^{(i)} > 0 $ necessarily have a slope ($ r_i(\widetilde{\bm{z}}) $ or $ \varrho_i(\widetilde{\bm{z}}) $) equal to zero, and, just after a new invasion, all the types with an exponent equal to 1 have a slope either equal to zero (if they are present in the new intermediate equilibrium), or negative (if they are not).
Hence all the exponents are necessarily bounded from above by 1.

\subsection{Main result}
We are now ready to state our main results on the convergence of the logarithmic exponents at the $\log(K)$ scale.
Let us specify our initial conditions. 

\paragraph{Assumptions on the initial conditions.} Assume that $ (\bm{Z}^K_{sc}(0))_{K\ge1} $ converges in probability to some $ \widetilde{\bm{z}}_0 \in \mathcal{E} $. We ask furthermore that $\widetilde{\bm{z}}_0$ corresponds either a matching type or a non-matching type equilibrium depending on the choice of the parameters. We additionally assume that $ (X^K_i(0), Y^K_i(0), i \in \lbrace 0, 1 \rbrace)_{K\ge1} $ converges in probability to $ (x_i(0), y_i(0), i \in \lbrace 0, 1 \rbrace) $, for $ x_i(0) \in [0,1] $, $ y_i(0) \in [0,1] $, satisfying \eqref{condition_initial_exponents}.
\begin{theo} \label{thm:main_result}
	Assume the above assumptions on initial conditions.
	Let $ (x_i(t), y_i(t), i \in \lbrace 0, 1 \rbrace, t \in [0,\tau_*)) $ and $ (\widetilde{\bm{z}}(t), t \in [0,\tau_*)) $ be given by Definition~\ref{def:limit}.
	Then, for all $ T \in (0, \tau_*) $,
	\begin{equation*}
		\Bigl( X^K_0\big(t \log(K)\big), X^K_1\big(t \log(K)\big), Y^K_0\big(t \log(K)\big), Y^K_1\big(t \log(K)\big), t \in [0,T] \Bigr)
	\end{equation*}
	converges as $ K \to \infty $ to $ (x_0(t), x_1(t), y_0(t), y_1(t), t \in [0, T]) $ in probability, uniformly on $ [0,T] $.
	Moreover, for any $ t \in [0,T] \setminus (\lbrace s_k, k \geq 0 \rbrace ) $, $ \bm{Z}^K_{sc}(t) \to \widetilde{\bm{z}}(t) $ in probability as $ K \to \infty $.
\end{theo}

\noindent Note that the convergence of $ (\bm{Z}^K_{sc}(t), t \in [0,T]) $ is actually locally uniform on each $ (s_k, s_{k+1}) $, see for example Lemma~\ref{lemma:invasion_phase} below in case~A.
\\

\begin{figure}[htb]
\centering
\includegraphics[width = \linewidth]{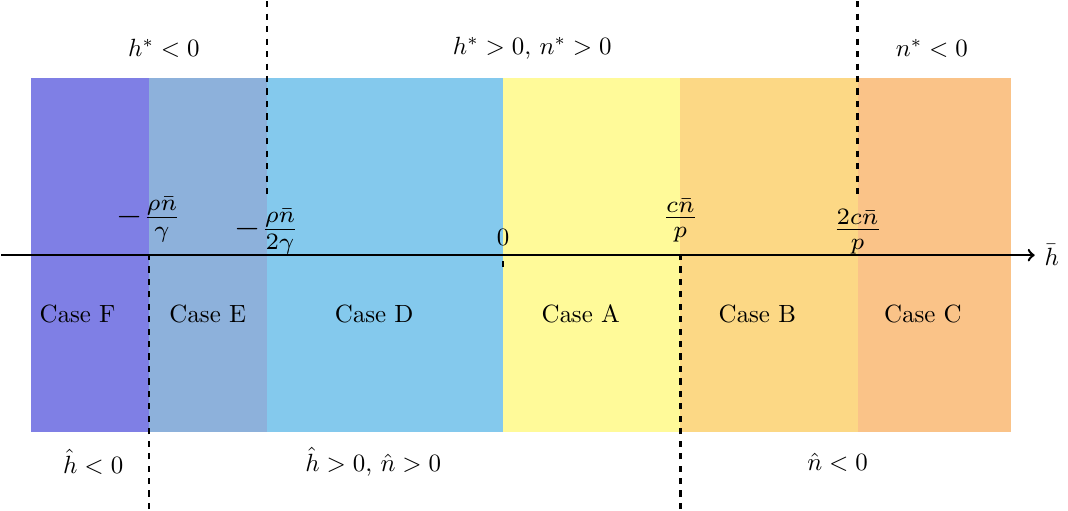}
\caption{Graphical representation of the six cases studied in our analysis, as a function of the value of $\bar{h}$ relative to $\bar{n}$. The non-matching types equilibria $(\bar{n}, 0, 0, \bar{h})$ and $(0, \bar{n}, \bar{h}, 0)$ described in Proposition~\ref{prop:non-matching} exist whenever $\bar{h} > 0$. The matching types equilibria $(\hat{n}, 0, \hat{h}, 0)$ and $(0, \hat{n}, 0, \hat{h})$ described in Proposition~\ref{prop:cv_sys_matching} exist when $\bar{h} \in (-\frac{\rho}{\gamma} \bar{n}, \frac{c}{p} \bar{n} )$. Finally, the coexistence equilibrium $(n^*, n^*, h^*, h^*)$ described in Proposition~\ref{prop:cv_syst_4} exists when $\bar{h} \in (-\frac{\rho}{2 \gamma} \bar{n}, \frac{2 c}{\rho} \bar{n})$. The resulting six cases are shown in the figure. }
\label{fig:cases}
\end{figure}

Before proving Theorem~\ref{thm:main_result}, let us describe the behaviour of the limiting exponents 
$$\bigl(x_0(t), x_1(t), y_0(t), y_1(t), t \in [0,\tau_*)\bigr)$$ in more detail.
In view of the results of Section~\ref{subsec:dynamical_system_results}, we need to consider six different situations, corresponding to the existence of different sets of equilibria for the dynamical system \eqref{eq:syst-4} (Figure~\ref{fig:cases}). Recall that we assume throughout that $b-d>0$, i.e. $\bar{n}>0$.

\begin{itemize}
\item[\textbf{Case A:}] $\bar{h} \in (0, \frac{c}{p} \bar{n})$ --  This corresponds to the situation where the four-types coexistence equilibrium, the matching types equilibrium and the non-matching types equilibrium are all positive, i.e.
\begin{align*}
    \bar{h} > 0, && \hat{h} > 0, && \hat{n} > 0, && n^* > 0, && h^* > 0.
\end{align*}
In this case, we will prove in Section \ref{sec:A} that there is always a prey population and a predator population of macroscopic order, that is of order $K$ for prey and $K^m$ for predators. When the types of the macroscopic prey and predator are matching, then the other prey will invade the community and replace the former prey. Otherwise, when the types of the macroscopic prey and predator do not match, the predator population whose type is associated with the prey will invade and replace the former predator.\ This situation is shown in Figure \ref{fig:caseA}.\\
We will prove that these successive invasions accumulate in finite time and that at this time $S_*$ all exponents are equal to $1$. Since the equilibrium with four types is stable we expect that after the accumulation time, the four populations converge to this coexistence equilibrium. Due to the strong symmetry of the system we only succeed to prove this in a slightly modified setting described in Section \ref{subsec:accumulation_point}.

\item[\textbf{Case B:}] $\bar{h} \in (\frac{c}{p} \bar{n}, \frac{2c}{p} \bar{n})$  -- This corresponds to the situation where the four-types coexistence equilibrium and the non-matching types equilibrium exist, but the matching types equilibrium does not, i.e.
\begin{align*}
    \bar{h} > 0, && n^* > 0, && h^* > 0, && \hat{n} < 0.
\end{align*}
In this case, the matching types equilibrium $(\hat{n}, \hat{h})$ is not positive and a prey and its matching predator cannot coexist on their own. This will lead to a very different situation than in case A, since, at the beginning, only predators will reach the macroscopic level. The assumption $n^*>0$ (which is equivalent to $\alpha_c > 1/2$, in the notation of Proposition~\ref{prop:syst-3}) means that both prey are able to invade the predator equilibrium $(0,0, \alpha \bar{h}, (1-\alpha) \bar{h})$ if the proportions of predator types are sufficiently close to $1/2$ (see Proposition \ref{prop:syst-3}). 
In fact, we prove that the dynamics will be governed by a succession of prey invasions, where the invading prey type corresponds to the less abundant predator (see Figure \ref{fig:caseB}). These prey invasions will accumulate, as in case A, leading to a situation where all populations reach a macroscopic level, and thus to the eventual coexistence of the four types.

\item[\textbf{Case C:}] $\bar{h} > \frac{2c}{p} \bar{n}$ -- This corresponds to the situation where neither the four-types equilibrium nor the matching types equilibrium are positive, i.e.
\begin{align*}
    \bar{h} > 0, && n^* < 0, && \hat{n} < 0.
\end{align*}
This situation is very similar to the previous one since the dynamics will be characterized by successive prey invasions into a population where the two predator types are at the macroscopic level. However, in this case when predator proportions are close to $1/2$, then no prey population is able to invade. 
We prove that after each prey invasion, the proportion of predators ends closer to $1/2$, and that, after a finite number of prey invasions, both prey populations become extinct (see Figure \ref{fig:caseC}).

\item[\textbf{Case D:}] $\bar{h} \in (-\frac{\rho}{2\gamma} \bar{n}, 0)$ -- This corresponds to the situation where both the matching types and four-types equilibria are positive, but the non-matching types equilibrium is not, i.e.
\begin{align*}
    \bar{h} < 0, && \hat{n} > 0, && \hat{h} > 0, && n^* > 0, && h^* > 0.
\end{align*}
It is similar to case B, but the roles of prey and predators are exchanged.
More precisely, because $\bar{h}<0$, the predator population in the non-matching equilibrium is negative, so that only the two prey populations will be at the macroscopic level at the beginning. 
Assumption $h^*>0$ then ensures that both predators are able to invade the prey equilibrium $(a \bar{n}, (1-a) \bar{n}, 0, 0)$ if the proportion of prey $a$ is sufficiently close to $1/2$. 
We prove that, similarly as in case B, the dynamics will be governed by successive predator invasions, where the invading predator corresponds to the most abundant prey. These invasions accumulate and lead to a situation with all species present at a macroscopic level.

\item[\textbf{Case E:}] $\bar{h} \in (-\frac{\rho}{\gamma} \bar{n}, -\frac{\rho}{2 \gamma} \bar{n})$ -- This corresponds to the situation where the matching types equilibrium exists, but the four-types and the non-matching types equilibria are negative, i.e.
\begin{align*}
    \bar{h} < 0, && \hat{n} > 0, && \hat{h} > 0, && h^* < 0.
\end{align*}
It is similar to case C. In this case, as in case D, successive invasion attempts by predators will drive the prey proportions into the vicinity of $1/2$. Because $h^*< 0$, once these proportions are sufficiently close to $1/2$ none of the predator populations will be able to invade, and the predator populations will become extinct.

\item[\textbf{Case F:}] $ \bar{h} < - \frac{\rho}{\gamma} \bar{n} $ -- In this case, neither the four-types nor the matching types equilibria are positive, i.e.
\begin{align*}
    \bar{h} < 0, && \hat{h} < 0, && h^* < 0.
\end{align*}
This means that both predator populations will go extinct immediately and the two prey species may coexist. As a consequence, we do not consider this case in this article.
\end{itemize}

\section{Detailed study of case A}
\label{sec:A}

\subsection{Description of the limiting exponents}

Let us denote by $\bar{z}_{0,1}$ and $\bar{z}_{1,0}$ the two equilibria featuring a prey population and a non-matching predator population, more precisely,
\begin{align*}
    \bar{z}_{0,1} = (\bar{n}, 0, 0, \bar{h}), && \bar{z}_{1,0} = (0, \bar{n}, \bar{h}, 0).
\end{align*}
Similarly, let $\hat{z}_0$ and $\hat{z}_1$ denote the two equilibria consisting of a prey population and the matching predator population,
\begin{align*}
    \hat{z}_0 = (\hat{n}, 0, \hat{h}, 0), && \hat{z}_1 = (0, \hat{n}, 0, \hat{h}).
\end{align*}
We observe that in this case, the successive invasions will lead to cyclic dynamics of the equilibrium $(\widetilde{z}_{k},k\ge0)$ with the successive states
$$
\begin{array}{ccc}
\hat{z}_0& \longrightarrow & \bar{z}_{1,0}\\
\uparrow&&\downarrow\\
\bar{z}_{0,1}&\longleftarrow&\hat{z}_1
\end{array}
$$
The situation is illustrated in Figure \ref{fig:caseA}. The associated slopes will depend on whether the resident equilibrium is composed of matching or non-matching predator-prey types.

\begin{figure}[h!]
\begin{center}
\includegraphics[width=\linewidth]{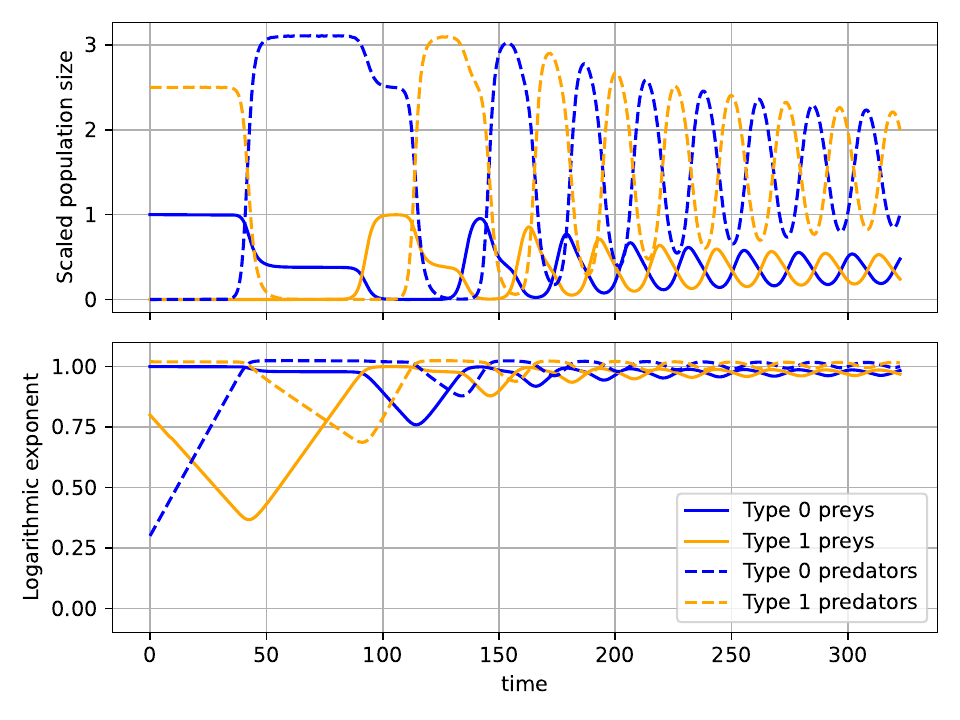}
\end{center}
\caption{Case A. Numerical simulation of the process $Z^K$ for $K=10^{20}$. We observe in the lower panel the convergence of the limiting exponent toward the limiting process. The upper panel corresponds to the associated rescaled population sizes. For small times (on $[0,150]$) we observe convergence to the process of equilibria, while afterwards the convergence speed of the deterministic system is not rapid enough and we see that the population sizes go closer to the four type coexistence equilibrium.}
\label{fig:caseA}
\end{figure}
\begin{figure}[h!]
\begin{center}
\includegraphics[width=\linewidth]{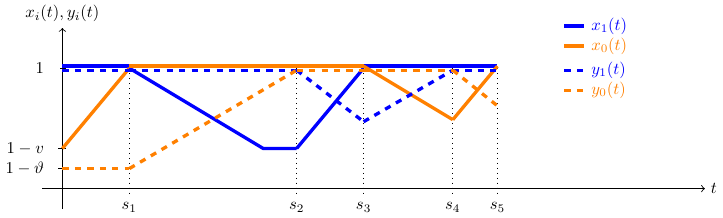}
\end{center}
\caption{Case A. Illustration of the limiting process.}
\label{fig:caseA2}
\end{figure}

Let us now describe the slopes associated with the limiting process $(x_0,x_1,y_0,y_1)$.
There are two different situations:
\begin{itemize}
    \item In a non-matching types equilibrium $\bar{z}_{0,1}$, we have \begin{align*}
    &r_1(\bar{z}_{0,1}) = b - d - c\bar{n} - p \bar{h} = - p \bar{h} < 0,\\ &\varrho_0(\bar{z}_{0,1}) = \frac{1}{m} (\beta - \delta - \gamma \bar{h} + \rho \bar{n}) = \frac{\rho \bar{n}}{m} > 0.
\end{align*}
In other words, the type 1 prey population is decaying and the type 0 predator population is invading.
The situation is symmetrical in the other non-matching equilibrium. We therefore set
\begin{equation}
    \label{eq:slope_A_bar}
    \bar{r}=-p\bar{h}, \qquad \bar{\varrho} =\frac{\rho \bar{n}}{m}\, .
\end{equation}
\item 
In a matching types equilibrium $\hat{z}_0$ or $\hat{z}_1$ the opposite prey will be favoured while the opposite predator decreases.
We indeed have that, using Eq.~\eqref{eq:n_hat_h_hat} in the last step of each line,
\begin{align} \label{slopes_matching_1}
    &\hat{r} :=r_1(\hat{z}_0)=r_0(\hat{z}_1)= b-d-c\hat{n} = c(\bar{n}-\hat{n}) = p \hat{h} > 0,\\
    &\hat{\varrho}:=\varrho_1(\hat{z}_0) = \varrho_0(\hat{z}_1)=\frac{\gamma}{m}(\bar{h}-\hat{h}) = - \frac{\rho}{m} \hat{n} <0. \label{slopes_matching_2}
\end{align}
\end{itemize}

\begin{prop}\label{prop:invasionTime_contraction}
    We assume case A where $\hat{n}>0$ and $\bar{h}>0$. We then have
    \begin{equation}
        s_{2(k+2)}-s_{2(k+1)} \leq \left(\frac{\bar{r}\hat{\varrho}}{\hat{r}\bar{\varrho}} \right)(s_{2(k+1)}-s_{2k})
    \end{equation}
    and $\bar{r}\hat{\varrho}/\hat{r}\bar{\varrho}<1$.\\
    As a consequence $\lim_{k\to\infty} s_k =S_*<+\infty$
\end{prop}

\begin{cor}\label{cor:exp_conv}
Assuming case A where $\hat{n}>0$ and $\bar{h}>0$, we have furthermore that $$\lim_{s\uparrow S_*}x_i(s)=1 \quad \text{and}\quad \lim_{s\uparrow S_*}y_i(s)=1\,, $$
for $i\in\{0,1\}$.
\end{cor}

\begin{proof}[Proof of Proposition \ref{prop:invasionTime_contraction}]
Let us write the basic identity
$$s_{2(k+1)}-s_{2k}=s_{2(k+1)}-s_{2k+1}+s_{2k+1}-s_{2k}\,.$$
In the sequel we will prove the recursion on two sequences $(s_{2(k+1)}-s_{2k+1})$ and $(s_{2k+1}-s_{2k})$.
Without loss of generality we start at time $s_k$ in a matching types equilibrium $\hat{z}_0$, but the same result follows if we start in a non-matching equilibrium as will be explained below.\smallskip

\noindent\textbf{Step 1: Invasion of a prey.}
Since $s_k$ is the time of the $(k-1)$-th invasion that led to the matching types equilibrium $\hat{z}_0$, the $k$-th invasion event corresponds to the type 1 prey invasion. We therefore have that $\widetilde{z}_k=\hat{z}_0$ and $\widetilde{z}_{k+1}=\bar{z}_{1,0}$. 
The length of this prey invasion is given by 
\begin{equation}
    \label{eq:first_step}
    s_{k+1}-s_k = \frac{1-x_1(s_k)}{\hat{r}}\, ,
\end{equation}
where $x_1(s_k)$ is the population density of the invading prey.
During this invasion the predator population of type 1 will decrease from $y_1(s_k)=1$ to 
\begin{align*}
y_1(s_{k+1}) &= 1 + \hat{\varrho}\left(\frac{1-x_1(s_k)}{\hat{r}}\right)\vee \vartheta \\&\ge 1 + \hat{\varrho}\left(\frac{1-x_1(s_k)}{\hat{r}}\right)\, .\end{align*}

\noindent\textbf{Step 2: Invasion of a predator.}
The new starting situation is $\widetilde{z}_{k+1}=\bar{z}_{1,0}$ and this will lead to the invasion of predators of type 1, i.e. $\widetilde{z}_{k+2}=\hat{z}_1$. The duration of this transition is
\[ s_{k+2}-s_{k+1} = \frac{1-y_1(s_{k+1})}{\bar{\varrho}}\, . \]
During this invasion the prey population of type 0 will decrease to 
\[ x_0(s_{k+2}) \ge  1 + \bar{r} \left( \frac{1-y_1(s_{k+1})}{\bar{\varrho}} \right)\, . \]

\noindent\textbf{Step 3: Invasion of a prey.}
Starting in $\widetilde{z}_{k+2}=\hat{z}_1$, next prey of type 0 will invade. We then have 
\[ \begin{aligned}
s_{k+3}-s_{k+2} &= \frac{1-x_0(s_{k+2})}{\hat{r}} \le  - \frac{\bar{r}(1-y_1(s_{k+1}))}{\hat{r}\bar{\varrho}} \le \frac{\bar{r}\hat{\varrho}}{\bar{\varrho}\hat{r}} (s_{k+1}-s_k)\, .
\end{aligned}
\]\medskip

A similar reasoning by considering first a predator invasion, then a prey invasion, allows to obtain
\[ \begin{aligned}
s_{k+4}-s_{k+3} &\le   \frac{\bar{r}\hat{\varrho}}{\bar{\varrho}\hat{r}} (s_{k+2}-s_{k+1})\, ,
\end{aligned}
\]
which leads to the stated upper bound.

\noindent\textbf{Step 4: Conclusion.}
To prove that the difference in invasion times is a contraction, we need to show that $\bar{r}\hat{\varrho}/\bar{\varrho}\hat{r}<1$.
In view of \eqref{eq:slope_A_bar}, \eqref{slopes_matching_1} and \eqref{slopes_matching_2}, we have
\begin{equation*}
    \frac{\bar{r} \hat{\varrho}}{\hat{r} \bar{\varrho}} = \frac{p \bar{h} \frac{\rho}{m} \hat{n}}{p \hat{h} \frac{\rho}{m} \bar{n}} = \frac{\bar{h} \hat{n}}{\hat{h} \bar{n}}.
\end{equation*}
Substituting \eqref{eq:n_hat_h_hat} into this equation yields
\begin{equation*}
    \frac{\bar{r} \hat{\varrho}}{\hat{r} \bar{\varrho}} = \frac{\bar{h}\bar{n} - \frac{p}{c} \bar{h}^2}{\bar{h}\bar{n} + \frac{\rho}{\gamma} \bar{n}^2} < 1.
\end{equation*}
Defining $s_{N}$ as a telescopic sum, we see that it admits a finite limit as $N\to\infty$:
\[ s_{2(N+1)} = \sum_{k=0}^N s_{2(k+1)}-s_{2k} \leq (s_2-s_0) \sum_{k=0}^N \left( \frac{\bar{r}\hat{\varrho}}{\hat{r}\bar{\varrho}} \right)^k \xrightarrow{N\to\infty} S_\ast < \infty. \]
\end{proof}

\begin{proof}[Proof of Corollary \ref{cor:exp_conv}] 
Let us consider similarly as above a $k$ such that $\widetilde{z}_k=\hat{z}_0$, we prove below that $\lim_{s\uparrow S_*} x_1(s)=1$. Similar reasoning can be made for the three other quantities. 

Note that we have observed in \eqref{eq:first_step} that 
$$1-x_1(s_k)=\hat{r}(s_{k+1}-s_{k})$$ and furthermore that from the cyclic behaviour of invasion
\begin{align*}
    &x_1(s_{k+1})=x_1(s_{k+2})=x_1(s_{k+3})=1,\\
    &\widetilde{z}_{k+4}=\hat{z}_0.
\end{align*}
We therefore deduce that for all $n\ge0$
$$1-x_1(s_{4n}) = \hat{r}(s_{4n+1}-s_{4n})$$
and thus as $(s_{4n+1}-s_{4n})\to_{n\to\infty}0$ we conclude that $\lim_{n\to\infty} x_1(s_n)=1$. This allows to conclude the proof since $(x_1(s), s\ge0)$ is the linear interpolation of the sequence $(x_1(s_k))_{k\ge1}$. 
\end{proof}

In the next subsection, we adapt the proofs of the works \cite{champagnat_stochastic_2021} and \cite{coquille_stochastic_2021} to our setting to obtain the convergence of the exponents $X^K_i(\cdot \log(K))$ and $Y^K_i(\cdot \log(K))$ on each interval $[s_k, s_{k+1}]$.
To study the limiting behaviour of our system around and after the accumulation point $S_*$, additional work is needed. Section \ref{subsec:accumulation_point} is devoted to this study.

\subsection{Convergence of the exponents} \label{subsec:cvg_exponents}

Here we prove the convergence of $X^K_i(t \log(K))$ and $Y^K_i(t\log(K))$ to a deterministic limit on intervals of the form $[0, s_k]$  for $k \in \N$ as stated in Theorem \ref{thm:main_result}.
The proof follows similar steps as the proof of Theorem 2.2 and Proposition 2.3 in  \cite{coquille_stochastic_2021} Section 4. \medskip

We assume that at time 0 the resident populations are close to the matching types equilibrium, and the other populations are much smaller. For any $\nu>0$ assuming $K$ large enough we have
\begin{align}
    \left| \frac{1}{K} \prey_0(0) - \hat{n} \right| \leq \nu, && \left| \frac{1}{K^m} \pred_0(0) - \hat{h} \right| \leq \nu,
\end{align}
and
\begin{align}
    \left| X^K_1(0) - x_1^0 \right| \leq C_0 \, \nu, && \left| Y^K_1(0) - y_1^0 \right| \leq C_0 \, \nu,
\end{align}
for some $x_1^0 \in (1-v,1)$ and $y_1^0 \in (1-\vartheta,1]$ and $C_0 > 0$.
As a result, the first invasion will be by the type 1 prey population.
(We assume that $x_1^0 > 1-v$ and $y_1^0 > 1 - \vartheta$ to avoid a jump at time 0, and we assume that $x_1^0 < 1$ so that it takes some non-negligible time for the type 1 prey population to invade.)
If $y_1^0 = 1$, then we also assume that $\prey_1(0) \leq \varepsilon \eta K$. This ensures that at least one population has not reached its macroscopic size yet. The opposite case would correspond to what we call the accumulation point, which will be treated below in Section \ref{subsec:accumulation_point}.

Let us now define the following stopping time.
For $ \varepsilon > 0 $, $C > 1$ and $ \eta > 0 $, we define
\begin{multline*}
	\theta^K_1 := \inf \Bigg\lbrace t \geq 0 : \prey_1(t) > \varepsilon \eta K \text{ or } \pred_1(t) > \varepsilon \eta K^m \\ \left. \text{ or } \left| \frac{1}{K}\prey_0(t) - \hat{n} \right| >  \varepsilon \text{ or } \left| \frac{1}{K^m} \pred_0(t) - \hat{h} \right| >  \varepsilon \right\rbrace.
\end{multline*}

\begin{prop} \label{prop:stability_eq}
	For any $ \varepsilon \in (0, \hat{n} \wedge \hat{h}) $ and $\nu\le \sqrt{2}\varepsilon$ there exist $ \eta > 0 $ and $ \zeta > 0 $ such that, if initial conditions satisfy the previous assumptions then, for any sequence $ (t_K, K > 0) $ satisfying $ \log(t_K) \leq \zeta K^{m\wedge 1} $,
	\begin{align*}
		\lim_{K \to \infty} \P\left( \left| \frac{1}{K} \prey_0(t_K \wedge \theta^K_1) - \hat{n} \right| >  \varepsilon \text{ or } \left| \frac{1}{K^m} \pred_0(t_K \wedge \theta^K_1) - \hat{h} \right| >  \varepsilon \right) = 0.
	\end{align*}

\end{prop}
This result is usually derived from a comparison of the resident population with logistic birth and death processes. However, since here the prey and predator populations are coupled through a non-monotonic interaction, finding an appropriate coupling is difficult and would require to study four interacting populations. We therefore give an alternative proof using a control on the Lyapunov function in Appendix \ref{app:proof_lyap}.

On the interval $ [0,\theta^K_1] $, we can ``flank'' the processes $ \prey_1 $ and $ \pred_1 $ by pairs of branching processes with immigration, whose growth rates will be only $ \mathcal{O}(\varepsilon) $ apart.

Using these couplings, Theorem \ref{theo:convBPI} and Proposition~\ref{prop:stability_eq}, we obtain the following Lemma.
\begin{lemma} \label{lemma:invasion_phase}
	There exists a constant $ C > 0 $, which does not depend on $ \varepsilon $, such that,
	\begin{align} \label{limit_exponent_H}
	\lim_{K \to \infty} \P\left({ \sup_{t \in [0, \theta^K_1/\log(K)]} \left| Y^K_1(t\log(K)) - \left( y_1^0  + \hat{\varrho} t \right) \vee (1-\vartheta) \right| > C \varepsilon}\right) = 0,
	\end{align}
	and
	\begin{align} \label{limit_exponent_N}
	\lim_{K \to \infty} \P\left(\sup_{t \in [0, \theta^K_1/\log(K)]} \left| X^K_1(t \log(K)) - \left( x_1^0  + \hat{r} t \right) \right| > C \varepsilon \right) = 0.
	\end{align}
	Moreover, there exists a constant $ C' > 0 $ such that
	\begin{align} \label{limit_theta}
		\lim_{K \to \infty} \P\left({\left| \frac{\theta^K_1}{\log(K)} - s_1 \right| > C' \varepsilon}\right) = 0,
	\end{align}
	where $s_1$ is defined in Definition~\ref{def:limit}, and in this case is given by
	\begin{equation*}
	    s_1 = \frac{1 - x_1^0}{\hat{r}}.
	\end{equation*}
\end{lemma}

\begin{proof}
\noindent
Recall that $\hat{r} > 0$ and $\hat{\varrho} < 0$.

\noindent \underline{Step 1: Control of the predator population.} We start by considering the small predator population $ \pred_1 $.
We note that for $ t \in [0,\theta^K_1] $ and $ K $ large enough so that the mutation rate $ \vartheta_K=K^{-\vartheta} \leq \varepsilon $,
\begin{multline*}
(1-\varepsilon) \beta \pred_1(t) + (\beta + \rho(\hat{n}-\varepsilon))(\hat{h}-\varepsilon) K^{(1-\vartheta)m} \leq \text{ birth rate of $ \pred_1 $ } \\
\leq (\beta + \rho \eta \varepsilon) \pred_1(t) + (\beta + \rho (\hat{n}+\varepsilon)) (\hat{h} + \varepsilon) K^{(1-\vartheta)m},
\end{multline*}
and 
\begin{align*}
(\delta + \gamma({\hat{h}}-\varepsilon)) \pred_1(t) \leq \text{ death rate of $ \pred_1 $ } \leq (\delta + \gamma ({\hat{h}} + \varepsilon +\eta\varepsilon)) \pred_1(t).
\end{align*}
As a result, there exist two branching processes with immigration $ (\tilde{H}^{K,-}_1(t), t \geq 0) $ and $ (\tilde{H}^{K,+}_1(t), t \geq 0) $ with birth rates and death rates given in Table~\ref{table:rates_H_N_tilde} such that, for all $ t \in [0,\theta^K_1] $,
\begin{align*}
\tilde{H}^{K,-}_1(t) \leq \pred_1(t) \leq \tilde{H}^{K,+}_1(t),
\end{align*}
almost surely.
Here, $ \beta - (\delta + \gamma {\hat{h}}) = -\rho {\hat{n}} < 0 $, so $ \beta - (\delta+ \gamma {\hat{h}}) - (\beta + (1+\eta)\gamma) \varepsilon < 0 $ and $ \varepsilon $ can be chosen small enough that $ \beta - (\delta + \gamma {\hat{h}}) + (\rho \eta + \gamma) \varepsilon < 0 $, so that both $ (\tilde{H}^{K,-}_1(t), t \geq 0) $ and $ (\tilde{H}^{K,+}_1(t), t \geq 0) $ are subcritical.
By Theorem \ref{theo:convBPI}, we then have
\begin{align*}
\lim_{K\to\infty} \frac{\log(1+\tilde{H}^{K,-}_1(t \log(K)))}{m \log(K)} = \left( y_1^0 - C_0\varepsilon - \frac{1}{m}(\rho \hat{n} +(\beta + (1+\eta)\gamma)\varepsilon) t \right) \vee (1-\vartheta),
\end{align*}
and
\begin{align*}
\lim_{K\to\infty} \frac{\log(1+\tilde{H}^{K,+}_1(t \log(K)))}{m \log(K)} = \left( y_1^0+C_0\varepsilon - \frac{1}{m}(\rho \hat{n} - (\rho \eta + \gamma)\varepsilon) t \right) \vee (1-\vartheta),
\end{align*}
in $ L^\infty([0,T]) $, in probability.
This implies Eq.~\eqref{limit_exponent_H}.\medskip

\begin{table}[h]
    \centering
    \renewcommand{\arraystretch}{2}
    \begin{tabularx}{\textwidth}{| l | X | X | X | X |}
    \hline
        \textbf{Process} & \textbf{Initial condition} & \textbf{Immigration rate} & \textbf{Per capita birth rate} & \textbf{Per capita death rate} \\
    \hline
        $ \tilde{H}^{K,-}_1 $ & $\lfloor K^{m(y^0_1 - C_0 \nu)} - 1 \rfloor$ & $a_- K^{(1-\vartheta) m}$ & $ (1-\varepsilon) \beta $ & $ \delta + \gamma ({\hat{h}} + (1+\eta) \varepsilon) $ \\
    \hline
        $ \tilde{H}^{K,+}_1 $ & $\lfloor K^{m(y^0_1 + C_0 \nu)} - 1 \rfloor $ & $a_+ K^{(1-\vartheta) m}$ & $ \beta + \rho \eta \varepsilon $ & $ \delta + \gamma ({\hat{h}}-\varepsilon) $ \\
    \hline
        $\tilde{N}^{K,-}_1$ & $\lfloor K^{(x_1^0-C_0\nu)}-1\rfloor$ & $f_- K^{1-v}$ & $(1-\varepsilon) b  $ & $d + c \hat{n} + ((1+\eta)c + p\eta) \varepsilon$ \\
    \hline
        $\tilde{N}^{K,+}_1$ & $\lfloor K^{(x_1^0+C_0\nu)}-1\rfloor$ & $f_+ K^{1-v}$ & $b  $ & $d + c \hat{n} - c \varepsilon $ \\
    \hline
    \end{tabularx}
    \caption{Parameters of the branching processes with immigration used to control the size of the two mesoscopic populations $H^K_1$ and $N^K_1$. The constants $a_-$, $a_+$, $f_-$ and $f_+$ are given by $a_- = (\beta + \rho ({\hat{n}}-\varepsilon))({\hat{h}}-\varepsilon)$, $a_+ = (\beta + \rho ({\hat{n}} + \varepsilon))({\hat{h}}+\varepsilon)$, $f_- = b(\hat{n}-\varepsilon)$ and $f_+ = b(\hat{n} + \varepsilon)$.}
    \label{table:rates_H_N_tilde}
\end{table}

\noindent\underline{Step 2: Control of the prey population.} We now turn to the invading prey population $ \prey_1 $.
For $ t \in [0,\theta^K_1] $ and $ K $ large enough so that $ v_K \leq \varepsilon $,
\begin{align*}
(1-\varepsilon) b \prey_1(t) + b (\hat{n}-\varepsilon) K^{1-v} \leq \text{ birth rate of $ \prey_1 $ } \leq b \prey_1(t) + b (\hat{n} + \varepsilon) K^{1-v},
\end{align*}
and
\begin{align*}
(d + c(\hat{n}-\varepsilon)) \prey_1(t) \leq \text{ death rate of $ \prey_1 $ } \leq (d + c(\hat{n} + (1+\eta)\varepsilon) + p \eta \varepsilon) \prey_1(t).
\end{align*}
As a result, there exist two branching processes with immigration $ (\tilde{N}^{K,-}_1(t), t \geq 0) $ and $ (\tilde{N}^{K,+}_1(t), t \geq 0) $ with birth rates and death rates given in Table~\ref{table:rates_H_N_tilde} such that, for all $ t \in [0,\theta^K_1] $,
\begin{align*}
\tilde{N}^{K,-}_1(t) \leq N^{K}_1(t) \leq \tilde{N}^{K,+}_1(t),
\end{align*}
almost surely.
We note that $ b-d-c\hat{n} = p \hat{h} > 0 $ so $ b-d-c\hat{n} + c \varepsilon > 0 $ and $ \varepsilon $ can be chosen small enough so that $ b-d-c\hat{n} - (b + (1+\eta)c + p\eta) \varepsilon > 0 $.
Then, by Theorem~B.1 in \cite{champagnat_stochastic_2021} (see also their Lemma~B.4),
\begin{align*}
\lim_{K \to \infty} \frac{\log( 1 + \tilde{N}^{K,-}_1(t\log(K)) )}{\log(K)} = (x_1^0-C_0\varepsilon)  \vee (1-v) + (p \hat{h} - (b + (1+\eta)c + p\eta)\varepsilon) t,
\end{align*}
and
\begin{align} \label{limit_exponent_N_tilde}
\lim_{K \to \infty}\frac{\log(1+ \tilde{N}^{K,+}_1(t\log(K)))}{\log(K)} = (x_1^0+C_0\varepsilon) \vee (1-v) + (p\hat{h} + c\varepsilon) t,
\end{align}
in $ L^\infty([0,T]) $, in probability.
This implies Eq.~\eqref{limit_exponent_N}.\medskip
		
\noindent\underline{Step 3: Proof of Eq.~\eqref{limit_theta}.}
Let us consider $C'>0$ to be chosen later.
Let us set $ s_1 = \frac{1 - x_1^0 \vee (1-v)}{p \hat{h}} $ and $ t_K = (s_1 - C'\varepsilon) \log(K) $  and note that
\begin{align*}
	\left\lbrace \frac{\theta^K_1}{\log(K)} < s_1 - C' \varepsilon \right\rbrace \subset \bigcup_{i=1}^4 E_i,
\end{align*}
with
\begin{align*}
&E_1 = \left\lbrace | \prey_0(t_K \wedge \theta_1^K) - \hat{n} K | > \varepsilon K \right\rbrace, \\
&E_2 = \left\lbrace | \pred_0(t_K \wedge \theta_1^K) - \hat{h} K^m | > \varepsilon K^m \right\rbrace, \\
&E_3 = \left\lbrace \pred_1(t_K \wedge \theta_1^K) > \varepsilon \eta K^m \right\rbrace, \\
&E_4 = \left\lbrace \prey_1(t_K \wedge \theta_1^K) > \varepsilon \eta K \right\rbrace.
\end{align*}

By Proposition~\ref{prop:stability_eq}, $ \P(E_i) \to 0 $ as $ K \to \infty $ for $ i \in \lbrace 1, 2 \rbrace $.
Now note that, by the fact that $ N_1^K(t\wedge \theta_1^K) \leq \tilde{N}^{K,+}_1(t\wedge \theta_1^K) $,
\begin{align*}
\left(	E_4 \cap \bigcap_{i=1}^3 E_i^c \right) &\subset \left\lbrace \tilde{N}^{K,+}_1(t_K \wedge \theta^K_1) > \varepsilon \eta K \right\rbrace.
\end{align*}
But we see that
\begin{align*}
	\tilde{N}^{K,+}_1(t_K \wedge \theta^K_1) > \varepsilon \eta K \quad \Rightarrow \quad \frac{\log(1+\tilde{N}^{K,+}_1(t_K \wedge \theta^K_1))}{\log(K)} > \frac{\log(1+\varepsilon \eta K)}{\log(K)} = 1 + o(1).
\end{align*}
But by \eqref{limit_exponent_N_tilde},
\begin{align*}
	\limsup_{K \to \infty} \sup_{t \in [0, s_1 - C' \varepsilon]} \frac{\log(1+\tilde{N}^{K,+}_1(t \log(K)))}{\log(K)} \longrightarrow_{K\to\infty} (x_1^0+C_0\varepsilon) \vee (1-v) + (p\hat{h} + c\varepsilon) (s_1 - C'\varepsilon),
\end{align*}
in probability as $K\to\infty$.
Substituting the expression for $ s_1 $, we note that
\begin{align*}
	(x_1^0+C_0\varepsilon) \vee (1-v) + (p\hat{h} + c\varepsilon) (s_1 - C'\varepsilon) &\le 1 - \varepsilon \left( C'(p\hat{h} + c \varepsilon) - c s_1 -C_0\right).
\end{align*}
Hence choosing $ C' > \frac{c s_1+C_0}{p\hat{h}} $ ensures that $ \P(E_4 \cap \bigcap_{i=1}^3 E_i^c) \to 0 $ as $ K \to \infty $ for $ \varepsilon > 0$ small enough.
We prove that $ \P(E_3 \cap_{i \neq 3} E_i^c ) \to 0 $ in the same way, and this concludes the proof of the lemma.
\end{proof}

In view of Proposition~\ref{prop:stability_eq} and Lemma~\ref{lemma:invasion_phase}, at time $\theta^K_1$, with high probability,
\begin{align*}
    \left| \frac{1}{K} \prey_0(\theta^K_1) - \hat{n} \right| \leq C \varepsilon, && \left| \frac{1}{K^m} \pred_0(\theta^K_1) - \hat{h} \right| \leq C \varepsilon, && \prey_1(\theta^K_1) \geq \varepsilon \eta K,
\end{align*}
and
\begin{equation*}
    Y^K_1(\theta^K_1) \leq y_1(s_1) + C \varepsilon,
\end{equation*}
where the right hand side is strictly smaller than 1 for $\varepsilon$ sufficiently small.
Thus, we obtain
\begin{equation*}
    \sup_{t \in [0,T]} \left\| Z^K_{sc}(\theta^K_1 + t) - z(t) \right\| \to 0
\end{equation*}
in probability as $K \to \infty$ for any $T$, where $z$ solves \eqref{eq:syst-4} with initial condition given by
\begin{align*}
    n_0(0) = \hat{n}, && h_0(0) = \hat{h}, && n_1(0) = \eta\varepsilon, && h_1(0) = 0.
\end{align*}

By Proposition~\ref{prop:cv_3types}, $z(t)$ converges to the non-matching types equilibrium $\bar{z}_{1,0}$ as $t \to \infty$,
\begin{equation*}
    \lim_{t \to \infty} z(t) = (0, \bar{n}, \bar{h}, 0) = \bar{z}_{1,0}.
\end{equation*}
Let us now set
\begin{equation*}
    \sigma^K_1 = \inf \left\lbrace t \geq \theta^K_1 : \left\| Z^K_{sc}(t) - \bar{z}_{1,0} \right\| \leq \varepsilon \text{ and } \prey_0(t) \leq \eta \varepsilon K \right\rbrace.
\end{equation*}
Then there exists $T > 0$ (depending on $\varepsilon$ and on all the other constants but not on $K$) such that, with probability converging to $1$ as $K\to\infty$, $\sigma^K_1 \leq \theta^K_1 + T$.
Moreover, on this event, $\lim_{K \to \infty} X^K_0(\sigma^K_1) = 1$ and there exists $C_1 > 0$ which does not depend on $\varepsilon$ such that, with probability converging to $1$ as $K\to\infty$
\begin{equation*}
    \left| Y^K_1(\sigma^K_1) - y_1(s_1) \right| \leq C_1 \varepsilon.
\end{equation*}

Then, after $\sigma^K_1$, we can proceed by induction as in \cite{coquille_stochastic_2021}.
In this case, the resident populations are in the non-matching types equilibrium, and the type 1 predator invades (with a positive growth rate given by $\bar{\varrho}$) while the type 0 prey population is decaying (with a negative growth rate given by $\bar{r}$).
In this case, the analogue of Proposition~\ref{prop:stability_eq} is a direct application of Lemma~A.5 in \cite{coquille_stochastic_2021}, as the two resident populations behave approximately like independent logistic birth-death processes.

At the end of the next invasion phase, $Z^K_{sc}$ will again be in an $\varepsilon$-neighbourhood of the matching types equilibrium (this time with type 1 populations), and the induction can continue.
\medskip

In view of the proof of {Proposition ~\ref{prop:invasionTime_contraction}}, at each new invasion time $s_k$, all the exponents are getting closer to 1.
It thus seems natural to expect that, after $S_*$, all these exponents are converging to 1, and that the four populations are close to their coexistence equilibrium. However, we were unable to prove that solutions of the four dimensional dynamical system starting from initial conditions of the form $(K\hat{n}, K^{1-\varepsilon},K^m\hat{h}, K^{(1-\varepsilon)m})$ reach a neighborhood of the coexisting equilibrium $\zbf^*$ in a time at most $O(\varepsilon\log(K))$. The main difficulty here comes from the symmetry of the system which makes the Lyapunov function degenerate in the sense that its derivative admits two null eigenvalues. We managed to prove the convergence after the accumulation time in a slightly modified setting presented in Section \ref{subsec:accumulation_point}.

\section{Description and proofs for cases B and C}

In this section we detail the elements of the proof that need to be modified for cases B and C (cases D and E being treated analogously in Appendix C). The only difference compared with the previous section is that the macroscopic populations will be either two predators (in cases B and C) or two prey (in cases D and E) with varying proportions. As we will see, the equilibria of the corresponding deterministic system (with two prey or two predators) are not hyperbolic, since any combination of the two types with a total size equal to the carrying capacity is an equilibrium. Handling these non-hyperbolic equilibria leads to difficulties (see \cite{coron2021emergence}). Below we only detail cases B and C since the two other cases are very similar. They correspond to a parameter range where we observe a coexistence at macroscopic scales of the two prey and successive invasions of predators.

\subsection{Case B }
\label{sec:B}

In case B, predators are autonomous ($\beta-\delta>0$), but the matching types equilibrium does not exist ($\hat{n}<0$), which implies that prey have negative invasion fitness if their matching predator is at their intraspecific carrying capacity $\bar{h}$. However, a prey species can invade if its matching predator population is not too abundant. More precisely, we will observe a coexistence of the two predator types at the macroscopic level. The associated equilibria for the predators will be of the form
$\{ (\alpha \bar{h}, (1-\alpha)\bar{h}) , \alpha\in[0,1]\}$.
As described in Proposition \ref{prop:syst-3}, the invasion of a prey will be possible if the proportion of the matching predator is below a critical proportion $\alpha_c$ defined in \eqref{eq:alpha_c_first}.
A prey invasion attempt will induce a change in predator proportions as described in Proposition~\ref{prop:syst-3}~$ii)$ leading to the invasion of the other prey population.
We note that in case B we assumed $n^* > 0$ which is equivalent to $\bar{h} \leq 2 \frac{c \bar{n}}{p}$ and to $\alpha_c>1/2$. 
The stability of the four-species equilibrium therefore implies that this critical proportion of predators $\alpha_c$ is larger than 1/2 (compare to Eq.~\eqref{eq:cond_4_types}).
\paragraph{Description of the limiting process}
We now describe more precisely the limiting process introduced in Definition \ref{def:limit}. A typical trajectory of Case B is shown in Figure \ref{fig:caseB}.

Recall that we assume that the initial state is a non-matching equilibrium (the matching one is not positive in case B). To fix ideas let us consider that at initial time $\widetilde{\bm{z}}_0=(\bar{n},0,0,\bar{h})$, $x_0(0)= y_1(0)=1$ and  $ x_1(0), y_0(0)\in(0,1)$. 
In the first time interval, $t\in [0,s_1]$, the predator $0$ invades with slope 
$$\varrho_0(\widetilde{\bm{z}}_0) = \rho \bar{n}>0 ,$$
while prey individuals of type $1$ decay with a slope 
$$r_1(\widetilde{\bm{z}}_0)= -p\bar{h}<0.$$
Since the type $0$ population is of order $K$ it produces $O(K^{1-v})$ prey of type $1$ by mutation. Therefore on $[0,s_1]$
\[
x_1(t) = \left(x_1(0)+r_1(\widetilde{\bm{z}}_0)t\right) \,\vee \, \left(1-v\right)
\]
and 
\begin{align*}
     y_0(t) &=y_0(0)+\varrho_0(\widetilde{\bm{z}}_0)t.
\end{align*}
The time $s_1$ is then defined as 
$$s_1=\inf\{t\ge0, y_0(0)+\varrho_0(\widetilde{\bm{z}}_0)t=1\}.$$
At the time $s_1$, the prey of type $0$ and both predator populations have macroscopic size. 
We deduce from Proposition \ref{prop:syst-3} that the next equilibrium is of the form $(0,0,\kappa_0(0)\bar{h}, (1-\kappa_0(0))\bar{h})$. Note that by definition $\kappa_0(0)>\alpha_c>1/2$ in case B, such that the predator $0$ is most abundant.

From this step on, we will observe successive prey invasions corresponding to the type of the less abundant predator. We will denote by $\alpha_n$ the proportion of predator $0$ at time $s_n$. Let us now detail the first step.
With this notation, $\alpha_1=\kappa_0(0)>\alpha_c>1/2$ and $\widetilde{\bm{z}}_1=(0,0,\alpha_1\bar{h},(1-\alpha_1)\bar{h}$. 
Note that as a consequence $1-\alpha_1<1-\alpha_c < \alpha_c$. 
Therefore, on the interval $[s_1,s_2]$ prey of type $1$ invade with a slope $$r_1(\widetilde{\bm{z}}_1)= b-d-p(1-\alpha_1)\bar{h}>0,$$
while prey individuals of type $0$ decay with a slope 
$$r_0(\widetilde{\bm{z}}_1)= b-d-p\alpha_1\bar{h}<0.$$
Note that here the mutation effect changes since the type $i$ population is of order $K$ it produces $K^{x_i(t)-v}$ prey of type $1-i$ by mutations. Therefore for $t \in [s_1,s_2]$
\begin{align*}
    x_1(t) &= \left(x_1(s_1)+r_1(\widetilde{\bm{z}}_1)(t-s_1)\right) \,\vee \, \sup_{s \in [s_1, t] : x_0(s) > v} \lbrace x_0(s) - v + r_1(\widetilde{\bm{z}}_1)(t-s) \rbrace, \\
    x_0(t) &= \left(1+r_0(\widetilde{\bm{z}}_1)(t-s_1)\right) \,\vee \, \sup_{s \in [s_1, t] : x_1(s) > v} \lbrace x_1(s) - v + r_0(\widetilde{\bm{z}}_1)(t-s) \rbrace, 
\end{align*}
which can be simplified to 
\begin{align*}
x_1(t) &= \left(x_1(s_1)+r_1(\widetilde{\bm{z}}_1)(t-s_1)\right) \,\vee \, \left(x_0(t)-v\right),\\
x_0(t) &= \left(1+r_0(\widetilde{\bm{z}}_1)(t-s_1)\right) \,\vee \, \left(x_1(t)-v\right).
\end{align*}

Since, $r_0(\widetilde{\bm{z}_1})<0$ the time $s_2$ is then defined as $$s_2=\inf\{t\ge s_1\,, x_1(s_1)+r_1(\widetilde{\bm{z}}_1)(t-s_1)=1\}.$$
At that time the next resident equilibrium is given by considering the long time behaviour of a dynamical system with the type $1$ prey and the two predators, which lead to a change in proportion of the predators, namely the new proportion of predators of type $1$ is given by $\kappa_0(1-\alpha_1)>\alpha_c$ and we have
$$\widetilde{\bm{z}_2}=(0,0,(1-\kappa_0(1-\alpha_1))\bar{h}, \kappa_0(1-\alpha_1)\bar{h})\, ,$$
following from Proposition~\ref{prop:syst-3}~$ii)$.
Consequently, the new predator $0$ proportion is $\alpha_2=1-\kappa_0(1-\alpha_1)<1-\alpha_c<1/2$ in case B.

\begin{figure}[h!]
\begin{center}
\includegraphics[scale=1.4]{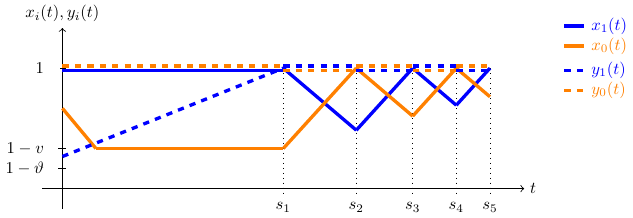}\\
\includegraphics[scale=0.3]{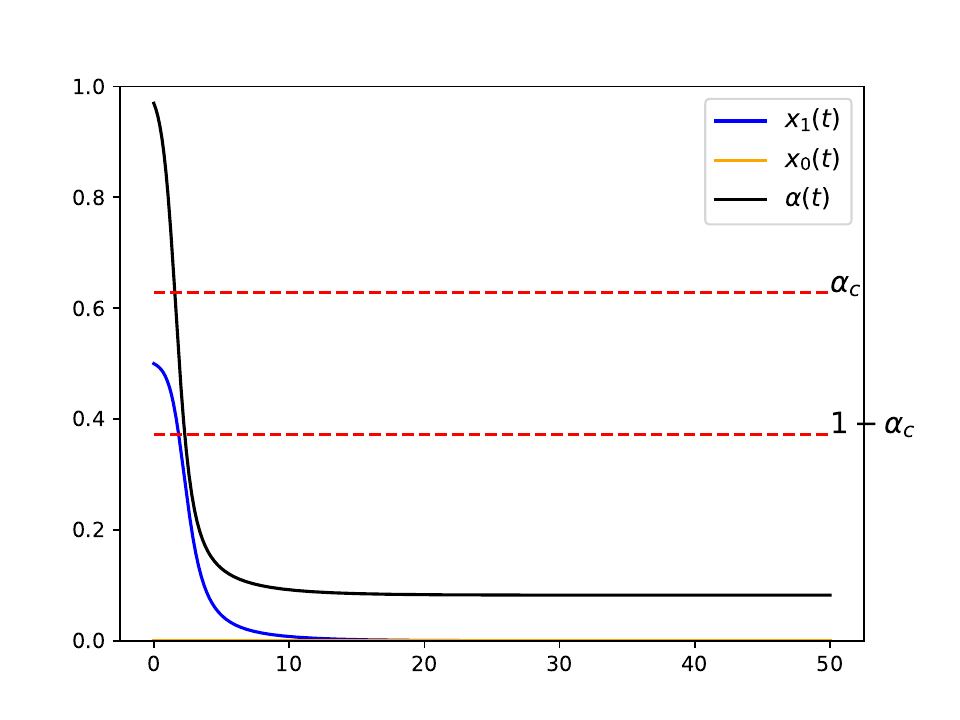}
\includegraphics[scale=0.3]{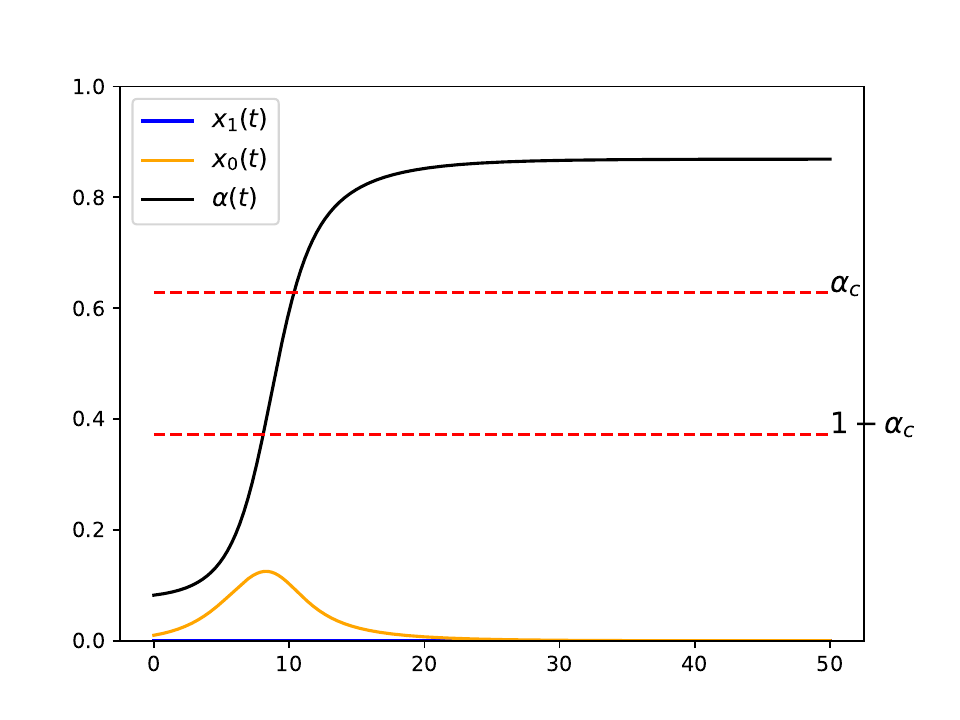}
\includegraphics[scale=0.3]{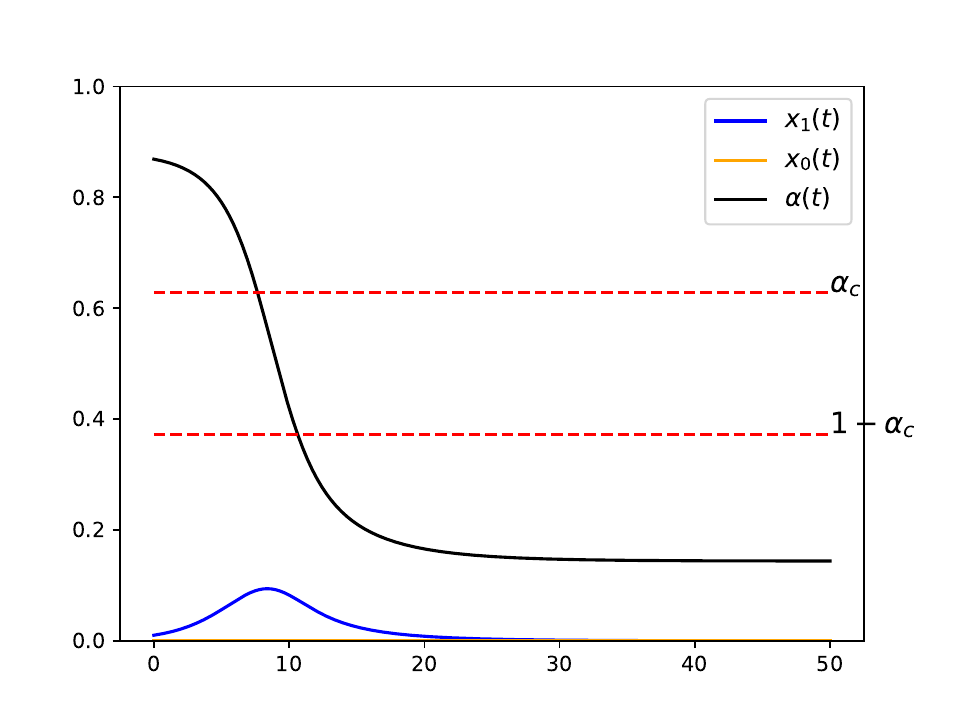}
\end{center}
\caption{Case B. The above Figure represents a typical trajectory of the limiting process $(x_0,x_1,y_0,y_1)$ in case B. The three graphs in the second line show the solutions of the dynamical system associated with the macroscopic populations at time $s_1$, $s_2$, $s_3$, where $\alpha(t)$ stands for the proportion of type $0$ predators $\alpha(t)=h_0(t)/(h_0(t)+h_1(t))$.} 
\label{fig:caseB}
\end{figure}

Following a similar argument, we can write the sequence of equilibria, 
$$\widetilde{\bm{z}}_n=(0,0,\alpha_n \bar{h}, (1-\alpha_n)\bar{h})$$
with 
$$\alpha_{2n}<1-\alpha_c<\alpha_c, \qquad \alpha_{2n+1}>\alpha_c.$$
As a consequence, the slopes driving successive prey invasions satisfy $
r_0(\widetilde{\bm{z}}_{2n})<0$, $r_0(\widetilde{\bm{z}}_{2n+1})>0$ and
$r_1(\widetilde{\bm{z}}_{2n})>0$, $r_1(\widetilde{\bm{z}}_{2n+1})<0$, and we have for all 
$t\in[s_{2n}, s_{2n+1}]$
\begin{align*}
      x_0(t) &= \left(x_0(s_{2n})+r_0(\widetilde{\bm{z}}_{2n})(t-s_{2n})\right)\,,\\
      x_1(t) &= \left(x_1(s_{2n})+r_1(\widetilde{\bm{z}}_{2n})(t-s_{2n})\right) \,\vee \, \left(x_0(t)-v\right)\, ,
\end{align*}
and for all $t\in[s_{2n+1}, s_{2n+2}]$ 
\begin{align*}
      x_0(t) &= \left(x_0(s_{2n+1})+r_0(\widetilde{\bm{z}}_{2n+1})(t-s_{2n+1})\right) \,\vee \, \left(x_1(t)-v\right)\,,\\
     x_1(t) &= \left(x_1(s_{2n+1})+r_1(\widetilde{\bm{z}}_{2n+1})(t-s_{2n+1})\right) \,.
\end{align*}

\paragraph{Accumulation of invasion times}
As in case A, we will prove that the sequence $s_n$ accumulates to a finite time and that $\lim_n x_i(s_n)=1$ for $i\in\{0,1\}$.\\
From the above construction we know that $(x_0(s_n),x_1(s_n))\in \{1\}\times[1-v,1) \cup [1-v,1)\times\{1\}$. It is then sufficient to prove that on any interval $[s_n,s_{n+1}]$ the positive slope is always larger than the negative one. 
For any $n\ge1$, on $[s_n, s_{n+1}]$ the larger predator proportion is $\alpha>\alpha_c$, such that $1-\alpha_c<1-\alpha$. Let us write the positive slope $\lambda_+$ and the negative one $-\lambda_-$. From the definition of $\alpha_c$ we obtain
\[ \lambda_-=p\bar{h}(\alpha-\alpha_c), \qquad \lambda_+ = p\bar{h}(\alpha_c+\alpha-1). \]
We deduce that 
$$\frac{\lambda_+}{\lambda_-} = 1+\frac{2\alpha_c-1}{\alpha-\alpha_c} \ge 1+\frac{2\alpha_c-1}{1-\alpha_c} >1.$$
This yields the following Proposition. 
\begin{prop}
\label{prop:accumulation_B}
We assume case B where $\beta-\delta>0$ and $n^*>0$. Then for both types $i\in\{0,1\}$, for any $n\ge0$ and for all $t\in[s_n,s_{n+1}]$, $x_i(t)>0$. Furthermore we have that $\lim_{n\to\infty} s_n<\infty$ and $\lim_{n\to\infty} x_i(s_{n})=1$ as $n\to\infty$.
\end{prop}

\paragraph{Convergence of the exponents}

The proof of convergence of the exponents follows similar steps as in case A: we first prove that the resident population remains close to equilibrium in order to compare microscopic populations with branching processes. Here, from the second invasion stage onwards, an additional difficulty arises from the fact that the dynamical system governing the densities of the two predators allows for an infinite number of equilibria. In particular, we need to prove that during a prey-invasion, the proportions of the two predator types will not vary significantly. 
In this section, we detail how the results from \cite{coron2021emergence} can be adapted to our case. Once these results are obtained, the comparison of microscopic populations with branching processes with immigration are similar to case A.

We provide below an analogue of Proposition \ref{prop:stability_eq}. Let us focus on a case where the initial time corresponds to the time $s_1$, the first time where the two predators have macroscopic size with a fixed proportion. More precisely, we assume that there exists $\alpha\in(0,\alpha_c)$ such that
\begin{align}
    \left| \frac{1}{K^m} \pred_0(0) - \alpha\bar{h} \right| \leq \nu, && \left| \frac{1}{K^m} \pred_1(0) -(1-\alpha) \bar{h} \right| \leq \nu,
    \label{condini_B1}
\end{align}
and
\begin{align}
    \left| X^K_0(0) - x_0^0 \right| \leq C_0 \, \nu, && \left| X^K_1(0) - x_1^0 \right| \leq C_0 \, \nu.
    \label{condini_B2}
\end{align}
To state these results rigorously, let us define different stopping times. The first one accounts for the growth of the prey population:
$$T_\varepsilon=\inf\{ t\ge0, \exists i\in\{0,1\}, N^K_i(t)=\lfloor \varepsilon K\rfloor\},$$
for $i\in\{0,1\}$.
The second one gives the first time when the proportions of type $0$ predators deviate considerably from their starting values: for any $\varepsilon>0$, 
\begin{equation}\label{defUeps}
 U_\varepsilon= \inf \left\{t\geq 0, \left|\frac{H^K_{0}(t)}{H^K(t)}-\frac{H^K_0(0)}{H^K(0)}\right|>\varepsilon \right\}.
\end{equation}
 The last one concerns the total predator population size: for any $\varepsilon>0$,
\begin{equation}\label{defReps}
 R_\varepsilon= \inf \left\{t\geq 0, \left|\frac{H^K(t)}{K^m}-\bar{h}\right|>\varepsilon  \right\}.
\end{equation} 
Note that all these stopping times depend on $K$ but we omit the dependency here.
Adapting the proof of Lemma 3.3 in \cite{coron2021emergence} we obtain the following proposition.

\begin{prop}
\label{prop:proportions}
Suppose that the assumptions of Proposition~\ref{prop:accumulation_B} hold and that initial conditions satisfy \eqref{condini_B1} and \eqref{condini_B2}.  Then, for any $\varepsilon>0$ there exists $\mathcal{A}_0>0$ and a positive constant $C$ independent of $\varepsilon$ such that, for all $\mathcal{A}\le\mathcal{A}_0$,
\[\liminf_{K \to \infty} \P \left( T_{\varepsilon }\wedge T_{0 } \le   R_{\mathcal{A}\varepsilon} \wedge U_{\varepsilon^{1/6}} \right) \ge 1- C\varepsilon^{1/12}.\]
\end{prop}
Our proof is very similar to Lemma 3.3 in \cite{coron2021emergence} and is detailed in Appendix \ref{app:proof_proportion}.

The scheme of the proof is very similar to case A, with additional difficulties in the coupling of the two prey populations with branching processes.
Let us set
\begin{equation*}
    \tau_{\varepsilon, K}:= T_{\varepsilon }\wedge T_0\wedge R_{\mathcal{A}\varepsilon} \wedge U_{\varepsilon^{1/6}}.
\end{equation*}
Then, by \eqref{condini_B1}, \eqref{defUeps} and \eqref{defReps}, we find that there exists $g = g(\nu, \eps)$ such that, for all $t \in [0, \tau_{\varepsilon, K})$,
\begin{align*}
    \left| \frac{1}{K^m} H^K_0(t) - \alpha \bar{h} \right| \leq g(\nu, \eps), && \left| \frac{1}{K^m} H^K_1(t) - (1-\alpha) \bar{h} \right| \leq g(\nu, \eps),
\end{align*}
and $g(\nu, \eps) \to 0$ as $\nu \to 0$ and $\eps \to 0$. In the sequel we shall choose $\varepsilon\ge \nu$ such that $\tau_{\varepsilon,K}>0$.

We will now construct four branching processes with immigration $\tilde{N}^{K,-}_0$, $\tilde{N}^{K,+}_0$, $\tilde{N}^{K,-}_1$ and $\tilde{N}^{K,+}_1$, such that, for all $t \in [0,\tau_{K,\eps}]$,
\begin{equation*}
    \tilde{N}^{K,-}_i(t) \leq N^K_i(t) \leq \tilde{N}^{K,+}_i(t), \quad i \in \lbrace 0, 1 \rbrace,
\end{equation*}
with probability growing to $1$ as $K\to\infty$. The difficulty lies in the fact that immigration rates have to be controlled in a precise way. We use the recursive technique developed in \cite{coquille_stochastic_2021} (Section 4.2) which amounts to first control the population without the mutations, and then to use these controls in the immigration rates.  Let us consider $ K $ large enough so that the mutation rate $ \vartheta_K=K^{-v} \leq \varepsilon $.\smallskip\\
In the following we explain briefly the construction of the lower bound process $\tilde N^{K,-}_i$ for $i\in\{0,1\}$, the upper bound being constructed similarly.\\
\label{p:BD_construction}We first remark that in the population $N^K_i$, the births without mutation occur at rate $$(1-\vartheta_K) b\ge (1-\varepsilon)b $$  and the death rate is upper bounded on $[0,\tau_{\varepsilon, K}]$ by 
$d+p\alpha\bar h +2c\varepsilon+pg(\varepsilon, \nu)$ for population with type $0$ and by 
$d+p(1-\alpha)\bar h +2c\varepsilon+pg(\varepsilon, \nu)$ for population with type $1$.
Therefore 
$$N^K_i(t) \ge \hat N^{K,-}_i(t)$$ where $\hat N^{K,-}_i(t)$ is a branching process with the previous birth and death rates and initial condition $\lfloor K^{(x_i^0-C_0\nu)}-1\rfloor$. \\
As a consequence from Theorem \ref{theo:convBPI}, with high probability 
$$X^K_0(t\log(K)) \ge X^K_0(0) +t \left[b(1-\varepsilon)- d-p\alpha\bar h -2c\varepsilon+pg(\varepsilon, \nu)\right]:= f_0^{-}(t),$$
and $$X^K_1(t\log(K)) \ge X^K_1(0) +t \left[b(1-\varepsilon)- d-p(1-\alpha)\bar h -2c\varepsilon+pg(\varepsilon, \nu)\right]:= f_1^{-}(t).$$
We now use this lower bound in the mutation rate and define four branching processes with immigration, denoted by $\tilde{N}^{K,-}_0$, $\tilde{N}^{K,+}_0$, $\tilde{N}^{K,-}_1$ and $\tilde{N}^{K,+}_1$, with immigration, birth and death rates given by Table~\ref{table:rates_N_tilde_case_B}, such that, for all $t \in [0,\tau_{\varepsilon, K}]$,
\begin{equation*}
    \tilde{N}^{K,-}_i(t) \leq N^K_i(t) \leq \tilde{N}^{K,+}_i(t), \quad i \in \lbrace 0, 1 \rbrace,
\end{equation*}
with high probability.\\
From this argument, we obtain the convergence of the slopes on the first invasion interval, and the conclusion of the proof is obtained recursively, as in case A.

\begin{table}[h]
    \centering
    \renewcommand{\arraystretch}{2}
    \begin{tabularx}{\textwidth}{| l | X | X | X | X |}
    \hline
        \textbf{Process} & \textbf{Initial condition} & \textbf{Immigration rate} & \textbf{Per capita birth rate} & \textbf{Per capita death rate} \\
    \hline
        $ \tilde{N}^{K,-}_0 $ & $\lfloor K^{(x_0^0 - C_0 \nu)} - 1 \rfloor$ & $b K^{f_1^-(t)}$ & $ (1-\varepsilon) b $ & $ d + p \alpha \bar{h} + 2c \eps + p g(\nu,\eps) $ \\
    \hline
        $ \tilde{N}^{K,+}_0 $ & $\lfloor K^{(x^0_0 + C_0 \nu)} - 1 \rfloor $ & $b K^{f_1^+(t)}$ & $ b $ & $ d + p \alpha \bar{h} - p g(\nu, \eps)  $ \\
    \hline
        $\tilde{N}^{K,-}_1$ & $\lfloor K^{(x_1^0-C_0\nu)}-1\rfloor$ & $ b K^{f_0^-(t)}$ & $(1-\varepsilon) b  $ & $d + p (1-\alpha) \bar{h} + 2 c \eps + p g(\nu, \eps)$ \\
    \hline
        $\tilde{N}^{K,+}_1$ & $\lfloor K^{(x_1^0+C_0\nu)}-1\rfloor$ & $ b K^{f_0^+(t)}$ & $b $ & $d + p(1-\alpha) \bar{h} - p g(\nu, \eps) $ \\
    \hline
    \end{tabularx}
    \caption{Parameters of the branching processes with immigration used to control the size of the two populations $N^K_0$ and $N^K_1$ on $[0,\tau_{\varepsilon, K}]$. In this table $f_i^{-}(t)=x^0_i +t \left[b(1-\varepsilon)- d-p\alpha\bar h -2c\varepsilon+pg(\varepsilon, v)\right]$ and $f_i^{+}(t)=x^0_i +t \left[b- d-p\alpha\bar h +2c\varepsilon+pg(\varepsilon, v)\right]$}
    \label{table:rates_N_tilde_case_B}
\end{table}

\subsection{Case C}
\label{sec:C}

In case C, predators are autonomous ($\beta-\delta>0$), but neither the two matching types nor the four types equilibrium exist. In contrast to the previous case B, this implies that the critical predator frequency for the matching prey invasion, $\alpha_c$, is smaller than 1/2. This can be derived from the first inequality in Eq.~\eqref{eq:cond_4_types}. As in case B, after the first successful invasion both predator species are of order 1. Prey types will then successively invade the predator population when the matching predator frequency is below $\alpha_c$. However, because $\alpha_c<1/2$, there is a neighborhood around 1/2 in the space of predator proportions, where neither prey species is able to invade. 
Proposition~\ref{prop:cv_3types} $iii)$ deals with case C and states that after a prey invasion, a predator with proportion $\alpha_0<\alpha_c$ will have a proportion $\alpha_c<\kappa_0(\alpha_0)<1-\alpha_0$. This implies, in particular, that $$\bigg| \alpha_0-\frac12\bigg| > \bigg|\kappa_0(\alpha_0)-\frac12\bigg|,$$ 
which means that the predator proportions are approaching 1/2, entering eventually the interval $(\alpha_c,1-\alpha_c)$, which then results in prey extinction.

\begin{figure}[h!]
\begin{center}
\includegraphics[scale=1]{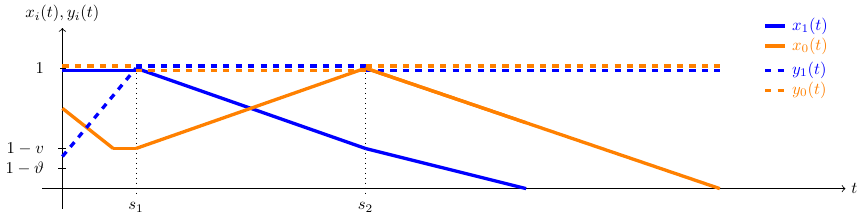}\\
\includegraphics[scale=0.45]{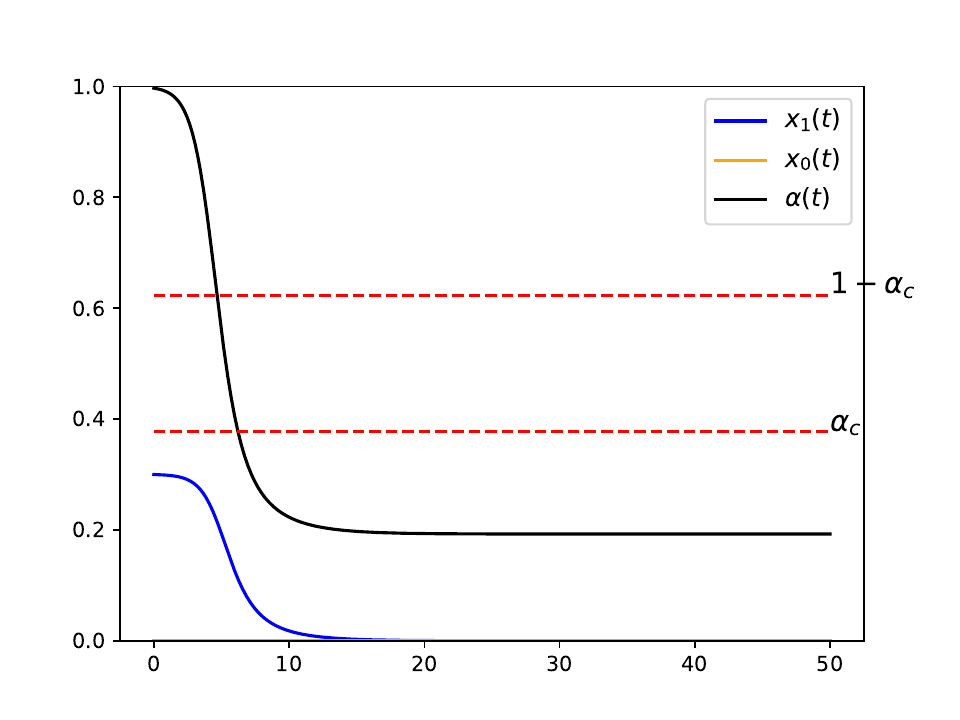}
\includegraphics[scale=0.45]{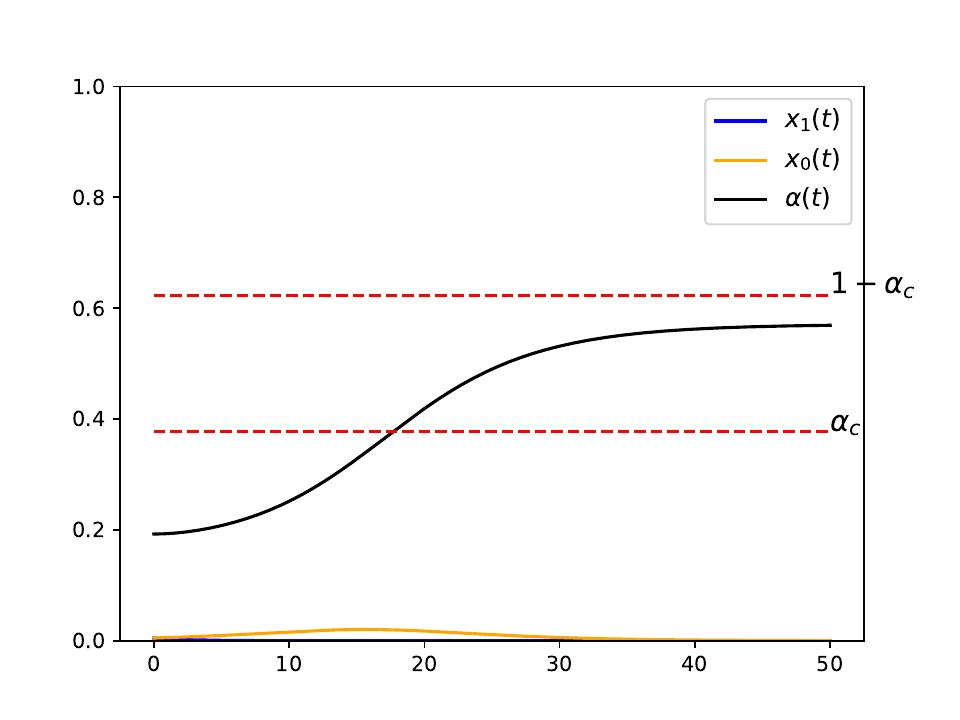}
\end{center}
\caption{Case C. The above Figure represents a typical trajectory of the limiting process $(x_0,x_1,y_0,y_1)$ in case C. The two graphs in the second line are the solutions of the dynamical system associated with the macroscopic populations at time $s_1$ and $s_2$, , where $\alpha(t)$ stands for the proportion of type $0$ predators $\alpha(t)=h_0(t)/(h_0(t)+h_1(t))$} 
\label{fig:caseC}
\end{figure}

\paragraph{Description of the limiting process}
As in case B, we assume that the initial state is a non-matching equilibrium with $\widetilde{\bm{z}}_0=(\bar{n},0,0,\bar{h})$. $x_0(0)=y_1(0)=1$ and $x_1(0), y_0(0) \in(0,1)$. Similarly, in the first time interval the predator of type $0$ invades with slope $\varrho_0(\widetilde{\bm{z}}_0)>0$ and the prey with type $1$ decays with slope $r_1(\widetilde{\bm{z}}_0)<0$. We can then define 
$$s_1 = \inf\{t\ge0, y_0(0)+  t\varrho_0(\widetilde{\bm{z}}_0)=1\}.$$
At time $s_1$, the prey of type $0$ and both predators have macroscopic size, and from Proposition \ref{prop:syst-3}, the next equilibrium has the form $(0,0,\kappa_0(0)\bar h, (1-\kappa_0(0))\bar h)$ with $\kappa_0(0)>1/2>\alpha_c$.
From this point, there are two different cases:
\begin{enumerate}
    \item $\kappa_0(0)<1-\alpha_c$ which entails that both prey population will have negative growth rates 
    $$r_0 = b-d-p\kappa_0(0) \bar{h} <0, \qquad  r_1 = b-d-p(1-\kappa_0(0)) \bar{h} <0,$$ and will go extinct. We can then define $$s_{ext,i} = \inf\{t\ge s_1, x_i(s_1) + (t-s_1)r_i=0\}, \quad  \forall i\in\{0,1\},$$ and set $s_2=\min(s_{ext,0},s_{ext,1})$,  $s_3=\max(s_{ext,0},s_{ext,1})$ and after this time the two prey populations are extinct.
    \item $\kappa_0(0)>1-\alpha_c$ and thus $1-\kappa_0(0)<\alpha_c$. This implies that the prey population with type $1$ will increase, and that similarly as in case B we will observe a succession of prey invasions. We denote by $\alpha_n$ the proportion of predator $0$ at time $s_n$, for $n\ge1$. By construction, we have $\alpha_1=\kappa_0(0)$, and since $\alpha_1>1-\alpha_c$, we obtain from Proposition \ref{prop:syst-3} that the invasion of prey with type $1$ leads to a change in the type 1 predator population such that $1-\alpha_2 = \kappa_0(1-\alpha_1)>1/2$. Using $iii)$, we obtain that moreover, $$f(1-\alpha_2) < f(1-\alpha_1) -M,$$ where $M>0$ and $f$ is the convex function given by \eqref{def:f}. This proves that after a finite number of invasions $k\ge1$, the predator proportions $\alpha_k$ and $1-\alpha_k$ will both belong to $(\alpha_c,1-\alpha_c)$, and we can conclude to the prey extinction as in the first item (see Figure \ref{fig:fcurve} in the appendix).
\end{enumerate}

\section{Behaviour at the accumulation point}
\label{subsec:accumulation_point}
The goal of this section is to study what happens at the accumulation time $S_*$. 
We have seen that in cases $A$ and $B$, $S_*$ is finite and $\lim_{k\to\infty} x_i(s_k)=\lim_{k\to\infty} y_i(s_k)=1$ for $i\in\{0,1\}$.  We expect that in these cases where the coexistence equilibrium $\zbf^*$ is positive, all four populations will coexist at a macroscopic level.
We introduce here a modified setting in which intraspecific competition is type dependent, for which a similar accumulation occurs. The main interest of such a modified setting is that we manage to prove the convergence of the population sizes after accumulation. We expect that a similar behaviour hold for the original setting but our proof strategy fails due to the symmetry of the system.

\subsection{A slightly modified setting} 

We consider here a slightly modified model in which the competition parameters depend on the type. More precisely we assume that a prey of type $i$ competes at rate $c$ with individuals of type $i$ and at rate $\tilde{c}$ with individuals of type $1-i$. Similarly, a predator of type $i$ competes at rate $\gamma$ with individuals of type $i$ and at rate $\tilde{\gamma}$ with individuals of type $1-i$. 
We will assume that 
\begin{equation}
    \label{ass:competition}
    c>\tilde{c}\text{ and }\gamma>\tilde\gamma.
\end{equation} 
With these assumptions the behaviour of the global system will be changed but if $c-\tilde{c}$ and $\gamma-\tilde\gamma$ are small, it should be qualitatively similar to that of the model with uniform competition studied above.

\paragraph{Behaviour of the associated dynamical system}
The associated large population limit follows a modified deterministic system, which reads
\begin{equation}
\label{eq:modified_system}
\left\{\begin{aligned}
&\frac{dn_i(t)}{dt}= n_i(t)(b-d-cn_i(t)-\tilde{c}n_{1-i}(t))-ph_i(t))\\
&\frac{dh_i(t)}{dt}=h_i(t)(\beta-\delta-\gamma h_i(t)-\tilde{\gamma}h_{1-i}(t)) +\rho n_i(t))\\
\end{aligned}\right.\quad \forall i\in\{0,1\}.
\end{equation}
This system admits a four-types coexistence equilibrium $(\tilde{n}, \tilde{n}, \tilde{h}, \tilde{h})$ where
\[ \tilde{n}= \frac{(b-d)(\gamma+\tilde\gamma)-p(\beta-\delta)}{\rho p +(c+\tilde{c})(\gamma+\tilde\gamma)}\,,
\qquad \tilde{h} = \frac{1}{p}(b-d-(c+\tilde{c})) \tilde{n}.\]
Note that similarly to \eqref{eq:n_star_h_star}, the equilibrium can be written as
\begin{equation} \label{ntilde_htilde}
    \tilde{n} = \frac{\bar{n}-\frac{p\gamma}{c(\gamma+\tilde{\gamma})}\bar{h}}{\frac{\rho p}{ c(\gamma+\tilde{\gamma})}+\frac{c+\tilde{c}}{c}},\qquad \tilde{h} = \frac{\bar{h}+\bar{n}\frac{\rho c}{(\gamma+\tilde\gamma)(c+\tilde c)}}{\frac{\rho p}{ \gamma(c+\tilde{c})}+\frac{\gamma+\tilde{\gamma}}{\gamma}}\,.
\end{equation}
We note that as $\tilde{c}\to c$ and $\tilde{\gamma}\to \gamma$, the equilibrium $(\tilde{n},\tilde{h})$ converges to $(n^*,h^*)$.
\begin{rem}
Note that this modified system \eqref{eq:modified_system} also admits the following equilibria of interest: non-matching equilibria $(0,\bar{n},\bar{h},0)$, $(\bar{n},0,0,\bar{h})$; matching equilibria $(\hat{n},0,\hat{h},0)$ $(0,\hat{n},0,\hat{h})$, equilibria with two prey $( \frac{b-d}{c+\tilde c},\frac{b-d}{c+\tilde c}, 0,0)$ and  equilibria with two predators $( 0,0,\frac{ \beta-\delta}{\gamma+\tilde \gamma},\frac{\beta-\delta}{\gamma+\tilde \gamma})$. In particular there will be no analog of cases $B$ or $C$.
\end{rem}

To mimic the conditions of case A, we need to ensure that
\begin{itemize}
    \item the matching types equilibrium $(\hat{n},\hat{h})$, the non-matching types equilibrium $(\bar{n},\bar{h})$ and the coexistence equilibrium $(\tilde{n},\tilde{h})$ are positive;
    \item the matching equilibrium is stable in a system with two predators and one prey;
    \item the non-matching equilibrium is stable in a system with two prey and a predator. 
\end{itemize} 
In view of \eqref{ntilde_htilde} and \eqref{eq:n_hat_h_hat}, the coexistence equilibrium is positive as soon as
\begin{equation*}
    0 < \bar{h} < \frac{c \bar{n}}{p} \le \frac{c(\gamma+\tilde \gamma)\bar{n}}{p\gamma}.
\end{equation*}
Therefore, we find that under the assumptions of case $A$ and with $c-\tilde{c}$ and $\gamma-\tilde\gamma$ small enough and positive, the dynamics of the stochastic processes are similar, that is, the convergence of $(X^K_i(t\log(K)),Y^K_i(t\log(K)))$ to a deterministic linear by parts process $(x_i(t),y_i(t))$ holds.

The slopes in the non-matching equilibrium $\bar{\zbf}_{0,1}$ or $\bar{\zbf}_{1,0}$ remain the same as in the symmetric case: $\bar{r}=-p\bar{h}$ for prey and $\bar{\varrho}=\rho\bar{n}{m}$ for predators. The only change comes from the slopes in the matching equilibria which become for prey 
$$ \tilde{r} = b-d-\tilde c \hat{n} = \tilde{c} (\bar{n}-\hat{n}) + (c-\tilde c)\bar{n}>0$$
and for predators
$$\tilde{\varrho}  m = \gamma\bar{h} -\tilde\gamma \hat{n} =  \tilde{\gamma}(\bar{h}-\hat{h}) + (\gamma-\tilde\gamma)\bar{h}.$$
The first term in $\tilde\varrho$ is negative such that if $(\gamma-\tilde\gamma)$ is small enough $\tilde\varrho$ remains negative.

Note that assuming that $(c-\tilde{c})$ and $(\gamma-\tilde\gamma)$ small enough ensures that $\frac{\bar r \tilde \varrho}{\bar \varrho\tilde r}<1$ and thus that the invasion times accumulate as in Proposition \ref{prop:invasionTime_contraction}.

\subsection{Convergence to coexistence equilibrium} 

Our goal is to prove that at the accumulation time $S_*$, the four populations will coexist around the equilibrium $\tilde{z}=(\tilde{n}, \tilde n, \tilde h, \tilde h)$.
More precisely, for all $\eta$,  we want to prove the existence of a stopping time $T^K_\eta$ such that
\begin{equation} \label{Z_reaches_equilibrium}
    \lim_{K\to\infty} \P(||\textbf{Z}^K_{sc}(T^K_\eta)-\tilde z||>\eta )=0
\end{equation}
and
\begin{equation} \label{time_to_reach_equilibrium}
    \lim_{K\to\infty} \P\left(\left|\frac{T^K_\eta}{\log(K)}-S_*\right|>\eta\right)=0
\end{equation}

The proof is split into three steps.
\begin{itemize}
    \item First, we will construct a stopping time close to $S_*$ such that, at this time, the stochastic population sizes are in a set of the form $[K^{1-\eps_0}, MK]^2 \times [K^{m(1-\eps_0)}, M K^m]^2$ with large probability.
    \item Second, we will prove that the deterministic large population approximation reaches a neighborhood of $\tilde{z}$ in a time of order $\eps_0\log(K)$ starting from such an initial condition.
    \item Third, we will prove that the stochastic process remains close to the deterministic approximation on this interval.
\end{itemize}

Let $\eta>0$ be fixed.
\paragraph{Step 1: Initial condition}
Recall the iterative construction of the limiting process $(x_0,x_1,y_0,y_1)$ and the associated times $(s_k)_{k\ge0}$. From Proposition \ref{prop:invasionTime_contraction} we deduce that we can choose $k_\eta$ large enough such that $|S_*-s_{k_\eta}|\le \eta/2$ and 
\begin{align*}
&x_0(s_{k_\eta}) \ge 1-\eta , \qquad x_1(s_{k_\eta}) \ge 1-\eta, \\
&y_0(s_{k_\eta}) \ge 1-\eta , \qquad y_1(s_{k_\eta}) \ge 1-\eta. 
\end{align*}
For $\eps>0$, using the proof of convergence of the exponents in Lemma \ref{lemma:invasion_phase}, there exists a stopping time that we will denote by $\theta^K_{k_\eta}$ such that, on an event $\Omega_{K,\eta}$, whose probability goes to 1 as $K \to \infty$,
\begin{align*}
    \left| \frac{1}{K} \prey_0(\theta^K_{k_\eta}) - \hat{n} \right| \leq  \varepsilon, && \left| \frac{1}{K^m} \pred_0(\theta^K_{k_\eta}) - \hat{h} \right| \leq  \varepsilon, && 2\varepsilon K\leq \prey_1(\theta^K_{k_\eta}) \leq 3\varepsilon K,
\end{align*}
and
\begin{equation*}
    Y^K_1(\theta^K_{k_\eta}) \geq y_1(s_1) - C \varepsilon \geq 1-\eta-C\varepsilon.
\end{equation*}
Moreover, $\frac{\theta^K_{k_\eta}}{\log(K)} \to s_{k_\eta}$ in probability as $K \to \infty$.
Let us now consider $$\eps_0 = \eta + C \varepsilon>0,$$ and $\kappa>0$ large enough. Then, for $K$ large enough, on the event $\Omega_{K,\eta}$, $$\textbf{Z}^K_{sc}(\theta^K_{k_\eta})\in[K^{-\eps_0},\kappa]^4.$$
Using the strong Markov property at time $\theta^K_{k_\eta}$, we can now consider the stochastic process starting from such an initial condition. 
Using the triangular inequality, we have $\forall t\ge0$
\[||\textbf{Z}^K_{sc}(t)-\tilde z||\le ||\textbf{Z}^K_{sc}(t)-\zbf^{(\textbf{Z}^K_{sc}(0))}(t)||+||\zbf^{(\textbf{Z}^K_{sc}(0))}(t)-\tilde z||.
\]
Let us first look at the second term on the right hand side.

\paragraph{Step 2: Convergence of the deterministic process}
Let us now introduce the Lyapunov function associated to the solution $$\zbf(t)=(n_0(t),n_1(t),h_0(t),h_1(t))$$ of the dynamical system \eqref{eq:modified_system}:
\[V(\zbf(t)) = \sum_{i=0}^1 \left[ n_i(t)-\tilde{n}\log(n_i(t))+\frac{p}{\rho}(h_i(t)+\tilde{h}\log(h_i(t))) \right].\]
Differentiation of the Lyapunov function along the trajectory of the dynamical system leads to 
\[ 
\frac{d}{dt}V(\zbf(t)) = -c(n_0-\tilde{n})^2 - 2\tilde{c}(n_0-\tilde{n})(n_1-\tilde{n})-c(n_1-\tilde{n})^2 -\gamma(h_0-\tilde{h})^2 - 2\tilde{\gamma}(h_0-\tilde{h})(h_1-\tilde{h})-\gamma(h_1-\tilde{h})^2 .
\]
This quadratic form is negative definite for $c>\tilde{c}$ and $\gamma>\tilde\gamma$, which implies the global stability of the coexistence equilibrium \eqref{ntilde_htilde}.
Moreover since $V$ is convex with a unique minimum at $(\tilde{n},\tilde{n},\tilde{h},\tilde{h})$, we have that $\forall \delta>0$ there exists $M_\delta > 0$ such that 
\[z\in\mathcal{B}(\tilde{z},\delta)\Rightarrow V(z)<M_{\delta}\] and conversely, $\forall M>0$ there exists $\delta_M>0$ such that 
\[ V(z)<M\Rightarrow z\in\mathcal{B}(\tilde{z},\delta_M).\]
Moreover, the map $M\to\delta_{M}$ is decreasing.

\begin{lemma}
\label{lem:lyap_argument}
Let us consider $\kappa>\eps>0$ and an initial condition $\zbf_0 \in [K^{-\eps}, \kappa]^{4}$. Then for all $\delta>0$, there exists a time $T^{K,\eps}=C_1+C_2\eps \log(K)$ where $C_1$ and $C_2$ only depend on $\delta$ and $\kappa$ such that for all $t\ge T^{K,\eps}$
\[z(t) \in\mathcal{B}(\tilde{z},\delta). \]
\end{lemma}

\begin{proof}
By the local stability of $\tilde{z}$, for any $\delta > 0$ small enough, there exists $\delta' > 0$ such that
\begin{equation} \label{delta_prime}
    \forall z \in B(\tilde{z}, \delta'), \quad \zbf^{(z)}(t) \in B(\tilde{z}, \delta), \forall t \geq 0.
\end{equation}
It is thus enough to show that there exists $T \leq T^{K,\varepsilon}$ such that $z(T) \in B(\tilde{z}, \delta')$.

Let us consider $M$ large enough such that $V(z)\le M$ implies $z\in\mathcal{B}(\tilde{z},\delta')$.
We define $T=T(\eps, \zbf_0,\delta')$ as the first time such that $\zbf^{(\zbf_0)}(t)$ enters the ball $\mathcal{B}(\tilde{z},\delta')$.\\
We can write
\[ V(z(t))=V(z(0))+\int_0^t \frac{d}{dt}V(z(t))|_{t=s} ds.\]
Then, using the fact that $\frac{d}{dt}V(z(t))$ is a quadratic negative definite form, we deduce that there exists a constant $C_{\delta'}>0$ such that, for all $t \leq T$, $\frac{d}{dt}V(z(t))\le -C_{\delta'}$, which implies
\[ V(z(t))\le V(z(0))-C_{\delta'} t, \quad \forall t \leq T.\]
We also note that, since $\zbf_0 \in [K^{-\eps}, \kappa]^{4}$,
\begin{align*}
    V(\zbf_0) &= \sum_{i\in\lbrace 0, 1\rbrace} \left[ n_i-\tilde{n}\log(n_i)+\frac{p}{\rho}(h_i-\tilde{h}\log(h_i)) \right] \\
    &\le 2 \left( 1 + \frac{p}{\rho} \right) \kappa + 2 \eps \log(K) \left( \tilde{n}+\frac{p}{\rho}\tilde{h} \right).
\end{align*}
Therefore, setting $C_1 = 2 \left( 1 + \frac{p}{\rho} \right) \kappa$ and $C_2 = 2 \left( \tilde{n}+\frac{p}{\rho}\tilde{h} \right) $, we have that $\forall t\le T$,
\[ V(z(t))\le C_1 + C_2 \eps \log(K) -C_{\delta'} t.\]
But the right hand side is smaller than $M$ as soon as 
\[t\ge \frac{1}{C_{\delta'}}(C_1-M+ C_2\eps\log(K) ) .\]
Thus, if we define
\begin{equation*}
    T^{K,\varepsilon}=\frac{1}{C_{\delta'}}(C_1-M+C_2\eps\log(K)),
\end{equation*}
we conclude that $T \leq T^{K,\varepsilon}$ and $z(T) \in B(\tilde{z}, \delta')$.
By \eqref{delta_prime}, this proves the statement.
\end{proof}

\paragraph{Step 3: Deterministic approximation of the stochastic process}
Recall that we assume that the initial condition satisfies $\textbf{Z}^K_{sc}(0)\in[K^{-\eps_0},\kappa]^4,$ where $\eps_0=\eta+C\eps>0$. In order to conclude the proof, it is sufficient to prove that with high probability as $K \to \infty$, $||\textbf{Z}^K_{sc}(T^{K,\eps_0})-\zbf^{(\textbf{Z}^K(0))}(T^{K,\eps_0})|| \le \delta$.

For sake of convenience, we shall use the following notations. When considering the non-rescaled process in $\N^4$, we write $\zbf = (N_0,N_1,H_0,H_1)$. We will note the $K$ dependency when considering the scaled version 
\begin{equation}
    \label{eq:notation}
\zbf^{K} = (n_0^K,n_1^K,h_0^K,h_1^K)  = \left(\frac{N_0}{K},\frac{N_1}{K},\frac{H_0}{K^m}, \frac{H_1}{K^m} \right).
\end{equation}
Let us introduce the infinitesimal generator $\mathcal{L}_{sc}^K$ associated with $\mathbf{Z}^K_{sc}$ defined in \eqref{def:Zsc}, for all measurable functions $f:\R^4\to \R$
\begin{align}
\mathcal{L}^K_{sc}f(\zbf^K)&=
\sum_{i=0}^1 bn_i^K(1-v_K) K(f(\zbf^K +\frac{1}{K} e_{n_i})-f(\zbf^K))  +  bn_i^K v_K K (f(\zbf +\frac{1}{K} e_{n_{1-i}})-f(\zbf^K))\nonumber\\
& + n_i^K K(d+c (n_i^K+n_{1-i}^K) + p h_i^K) (f(\zbf^K-\frac{1}{K}e_{n_i})-f(\zbf^K))\nonumber\\
&+   h_i^K(\beta+\rho n_i^K)(1-\vartheta_K) K^m(f(\zbf^K +\frac{1}{K^m} e_{h_i})-f(\zbf^K))\nonumber\\
&+  h_i^K (\beta+\rho n_i^K)\vartheta_K K^m (f(\zbf^K +\frac{1}{K^m} e_{h_{1-i}})-f(\zbf^K))\nonumber\\
& + h_i^K K^m(\delta+\gamma (h_i^K+h_{1-i}^K)) (f(\zbf^K-\frac{1}{K^m}e_{h_i})-f(\zbf^K)),
\label{eq:gen_sc}
\end{align}
where $(e_{n_0}, e_{n_1}, e_{h_0}, e_{h_1})$ stands for the canonical basis of $\R^4$.
We will use at several places the following decomposition for $f$ a measurable and $\mathcal{C}^1$ function
\begin{equation}
    \label{eq:generator_approx_1}
\mathcal{L}^K_{sc}f(\zbf^K) = F(\zbf^K) \cdot \nabla f(\zbf^K) + \mathcal{R}^K_1f(\zbf^K)+\mathcal{R}^K_2(\zbf^K),
\end{equation}
where $F : \R_+^4 \to \R^4$ is the direction of the flow of the associated deterministic system defined for all $(n_0,n_1,h_0,h_1)\in(\R_+)^4$ by \eqref{eq:syst-4}:
\begin{align*}
F((n_0, n_1, h_0, h_1)) &= \begin{pmatrix}
n_0 (b-d - c(n_0+n_1) - ph_0) \\
n_1 (b-d - c(n_0+n_1) - ph_1) \\
h_0 (\beta-\delta - \gamma (h_0+h_1) + \rho n_0) \\
h_1 (\beta-\delta - \gamma (h_0+h_1) + \rho n_1)
\end{pmatrix},
\end{align*}
and $\mathcal{R}^K_1f(\zbf^K)$ and $\mathcal{R}^K_2(\zbf^K)$ are rest terms involved in neglecting small effects of mutations and the Taylor expansion of $f$. The detailed decomposition will be stated in the proofs (see Appendix \ref{app:sto}).

Using \eqref{eq:generator_approx_1}
, we can write 
\begin{align*}
    \textbf{Z}^K_{sc}(t) &= \textbf{Z}^K_{sc}(0) + \int_0^t \mathcal{L}^K_{sc}I (\textbf{Z}^K_{sc}(s))ds + M^K_t\\
    &= \textbf{Z}^K_{sc}(0) + \int_0^t F(\textbf{Z}^K_{sc}(s)) ds+ \int_0^t \mathcal{R}_1^K (\textbf{Z}^K_{sc}(s))ds  + M^K_t,
\end{align*}
where $I(\zbf)=\zbf$, $M^K_t$ is the martingale term
and 
\begin{align*}
\mathcal{R}_1^K (\zbf^K)&=
\sum_{i=0}^1 bn_i^K v_K ( e_{n_{1-i}}-e_{n_i})+   h_i^K(\beta+\rho n_i^K)\vartheta_K ( e_{h_{1-i}}-e_{h_i}).
\end{align*}
In this case, there is no rest term coming from the Taylor expansion of the identity function.
Moreover, the deterministic solution $\zbf^K(t) = \zbf^{(\textbf{Z}^K_{sc}(0))}(t)$ satisfies 
\[ \zbf^K(t)=  \textbf{Z}^K_{sc}(0) + \int_0^t F(\zbf^K(s)) ds.\]
As a consequence 
\begin{align*}
||\textbf{Z}^K_{sc}(t) -\zbf^K(t)||
&\le  \int_0^t|| F(\textbf{Z}^K_{sc}(s))-F(\zbf^K(s))||ds+ \left\|\int_0^t\mathcal{R}_1^K (\textbf{Z}^K_{sc}(s))ds\right\| + ||M^K_t||.
\end{align*}
Let us consider a compact set $\mathcal{K} \subset \R_+^4$ such that $[K^{-\varepsilon_0}, \kappa]^4 \subset \mathcal{K}$ and define $\tau_\kappa^K$ as the exit time of $\textbf{Z}^K_{sc}$ from $\mathcal{K}$. Later on, we will prove that it is possible to choose  $\mathcal{K}$ such that $\P(T^{K,\eps_0}<\tau_\kappa^K)\to 1$ as $K \to \infty$.

Due to the regularity of $F$, there exists a constant $C_{Lip}$ such that $\forall z, z'\in \mathcal{K}$, 
\[||F(z)-F(z')||\le C_{Lip}||z-z'||.\]
Moreover, for all $t\le \tau_\kappa^K$, there exists a constant $C_R$ such that 
\[\mathcal{R}_1^K(\mathbf{Z}^K(t))\le C_R K^{- (v\wedge \vartheta)}.\]
As a consequence, we obtain that for all $t\le \tau_\kappa^K$
\begin{align*}
||\textbf{Z}^K_{sc}(t) -\zbf^K(t)||
&\le C_{Lip}  \int_0^t  ||\textbf{Z}^K_{sc}(s) - \zbf^K(s)||ds+  t C_R K^{- (v\wedge \vartheta)}+ \sup_{s\in[0,t]}||M^K_s||.
\end{align*}
Using Gronwall's lemma we find that 
\begin{align*}
||\textbf{Z}^K_{sc}(t\wedge\tau_\kappa^K) -\zbf^K(t\wedge\tau_\kappa^K)||
&\le \exp(C_{Lip}(t\wedge\tau_\kappa^K))\left( (t\wedge\tau_\kappa^K)C_R K^{- v\wedge \vartheta}+ \sup_{s\in [0,(t\wedge\tau_\kappa^K)]}||M^K_s||\right).
\end{align*}
Finally, since there exists a constant $C_M$ such that for all $t \geq 0$ 
\[\E[||\langle M^K\rangle_{t\wedge\tau_\kappa^K}||]\le t C_M K^{-(1\wedge m)}, \]
we obtain with Doob's maximal inequality that 
\begin{equation*}
    \P\left( \sup_{t \in [0, T^{K,\varepsilon_0}\wedge\tau_\kappa^K]} \| M^K(t) \| > a_K \right) \leq \frac{C_M T^{K,\varepsilon_0} K^{-  (1\wedge m)}}{a_K^2}.
\end{equation*}
Taking for example $a_K = K^{-(1\wedge m)/ 3}$, the right hand side tends to zero as $K \to \infty$.

On the event
\begin{equation*}
    \left\lbrace \sup_{t \in [0, T^{K,\varepsilon_0}\wedge \tau_\kappa^K]} \| M^K(t) \| \leq a_K \right\rbrace,
\end{equation*}
we have that there exists $\delta>0$ such that
\begin{align*}
||\textbf{Z}^K_{sc}(T^{K,\varepsilon_0}\wedge\tau_\kappa^K) -\zbf^K(T^{K,\varepsilon_0}\wedge\tau_\kappa^K)||
&\le \exp(C_{Lip}\ T^{K,\varepsilon_0}) T^{K,\varepsilon_0}CK^{- \delta},
\end{align*}
for a constant $C>0$ changing from line to line.
Recalling from Lemma~\ref{lem:lyap_argument} that $T^{K,\varepsilon_0} = C_1 + C_2 \varepsilon_0 \log(K)$, 
we conclude that 
\begin{equation*}
    ||\textbf{Z}^K_{sc}(T^{K,\varepsilon_0}\wedge\tau_\kappa^K) -\zbf^K(T^{K,\varepsilon_0}\wedge\tau_\kappa^K)|| \leq K^{\varepsilon_0 C_2 C_{lip}-\delta} C \log(K).
\end{equation*}
Choosing $\varepsilon_0$ small enough allows the r.h.s. to go to $0$ as $K\to\infty$. 
This yields
\begin{equation} \label{lim_P_sup_Z-z}
    \lim_{K \to \infty} \P\left( \sup_{t \in [0, T^{K,\varepsilon_0} \wedge \tau_\kappa^K]} \| \textbf{Z}^{K}_{sc}(t) - \zbf^K(t) \| > \delta \right) = 0,
\end{equation}
for any $\delta > 0$.\smallskip\\

Let us finally prove that we can  choose $\mathcal{K}$ such that $\P(T^{K,\eps_0}<\tau_\kappa^K)\to 1$ as $K \to \infty$. Recall that from Lemma \ref{lem:lyap_argument}, $T^{K,\eps_0}$ is defined from the solutions of the deterministic flow.
Let us remark that, on the event
\begin{equation*}
    \left\lbrace \sup_{t \in [0, T^{K,\varepsilon_0} \wedge \tau_K]} \| \textbf{Z}^{K}_{sc}(t) - \zbf^K(t) \| \leq \delta \right\rbrace,
\end{equation*}
the trajectory of $t \mapsto \textbf{Z}^K_{sc}(t \wedge T^{K,\varepsilon_0} \wedge \tau_K)$ lies in the set
\begin{equation*}
   \mathcal E_\delta:= \lbrace \zbf \in \R_+^4 : \exists t \in [0,T^{K,\varepsilon_0}] : \| \zbf - \zbf^K(t) \| \leq \delta \rbrace, 
\end{equation*}
which depends only on the initial condition $\textbf{Z}^K_{sc}(0)$ of both the stochastic and deterministic processes and $\delta$.
We can thus choose $\mathcal{K}\subset \mathcal E_\delta$ 
and then, by \eqref{lim_P_sup_Z-z}, $\P(\tau_K \leq T^{K,\varepsilon_0}) \to 0$ as $K \to \infty$.

Putting together the conclusions from steps 1, 2 and 3, this concludes the proof of \eqref{Z_reaches_equilibrium}, for $T^K_\eta = \theta^K_{k_\eta} + T^{K,\varepsilon_0}$, for a small enough $\varepsilon_0$.

Let us now prove \eqref{time_to_reach_equilibrium}.
We observe, using the value of $T^{K,\varepsilon_0}$ given in Lemma \ref{lem:lyap_argument}, that
\begin{equation*}
    \frac{\theta^K_{k_\eta}}{\log(K)} \leq \frac{T^K_\eta}{\log(K)} \leq \frac{\theta^K_{k_\eta}}{\log(K)} + C_2 \varepsilon + \frac{C_1}{\log(K)}.
\end{equation*}
Since $\frac{\theta^K_{k_\eta}}{\log(K)} \to s_{k_\eta}$ as $K \to \infty$ and $|s_{k_\eta} - S_*| \leq \eta/2$, choosing $\varepsilon$ small enough, we obtain \eqref{time_to_reach_equilibrium}.

\section*{Acknowledgements}
We would like to thank Camille Coron and Diala Abu Awad for organizing the Research Programm ``Ecosystems dynamics: Stakes, data and models'' (Institut Pascal - Université Paris Saclay) in 2019 where our discussions on this project started.
\\
M.C. is partially funded by the Chair ”Modélisation Mathématique et Biodiversité” of Veolia Environnement-École Polytechnique-Muséum national d’Histoire naturelle-Fondation X and by ANR project HAPPY (ANR-23-CE40-0007) and DEEV (ANR-20-CE40-0011-01).

\bibliographystyle{plain}
\bibliography{biblio}

\newpage
\appendix
\begin{center}
{\Large 
    Appendix}
\end{center}
\section{Proofs associated with the deterministic dynamical system}

\subsection{Proof of Proposition~\ref{prop:cv_sys_matching}}
\label{proof:2SpecMatching}

\paragraph{Proof of $i)$}
In this case, we show that the equilibrium $(\bar{n},0)$ is asymptotically stable for any positive initial condition. We prove this with the help of the following Lyapunov function
\[ V(n,h)=n-\frac{b-d}{c}\log(n)+\frac{p}{\rho}h\, .\] 
Its derivative along trajectories is
\[
\frac{d}{dt} V(n(t),h(t)) =-c\left(n(t)-\frac{b-d}{c}\right)^2 -h(t)^2\frac{p\gamma}{\rho} +h(t)p\left(\frac{b-d}{c}+\frac{\beta-\delta}{\rho}\right)\, .
\]
The last term is always negative because of condition~\eqref{eq:cond_equilibria_matching} being violated. This ensures that $t\mapsto V(n(t),h(t)) $ decreases with time and only vanishes for $(n,h)=(\frac{b-d}{c},0)$, which implies the convergence of trajectories to $(\bar{n},0)$.
\paragraph{Proof of $ii)$}

We now assume that condition~\eqref{eq:cond_equilibria_matching} holds, which implies that the equilibrium $(\hat{n},\hat{h})$ is positive.
To show that this equilibrium is asymptotically stable for any positive initial condition, we define the following Lyapunov function:
\begin{equation}
\label{eq:Lyap_2species}
V(n,h) := \frac{\hat{n}}{p}\left[\frac{n}{\hat{n}}-1- \log\left(\frac{n}{\hat{n}}\right)\right] +\frac{\hat{h}}{\rho}\left[\frac{h}{\hat{h}}-1- \log\left(\frac{h}{\hat{h}}\right)\right] \, .
\end{equation}
Derivation of $V(n,h)$ along a trajectory yields
\begin{align}
\label{eq:diff_lyap_2m}
\frac{d}{dt} V(n(t),h(t)) &= \frac{\hat{n}}{p}\left[\frac{d}{dt}\frac{n(t)}{\hat{n}} - \frac{\frac{d}{dt}n(t)}{n(t)}\right] +\frac{\hat{h}}{\rho}\left(\frac{d}{dt}\frac{h(t)}{\hat{h}} - \frac{\frac{d}{dt}h(t)}{h(t)}\right)\nonumber\\
&= \frac{1}{p}(n(t)-\hat{n}) (b-d-cn(t)-ph(t))  +\frac{1}{\rho}(h-\hat{h})(\beta-\delta-\gamma h(t)+ \rho n(t))\nonumber\\
&=-\frac{c}{p}(n(t)-\hat{n})^2-(n(t)-\hat{n})(h(t)-\hat{h}) -\frac{\gamma}{\rho}(h(t)-\hat{h})^2 +\frac{1}{\rho}\rho(h(t)-\hat{h})(n(t)-\hat{n})\nonumber\\
&= -\frac{c}{p}(n(t)-\hat{n})^2-\frac{\gamma}{\rho}(h(t)-\hat{h})^2\, ,
\end{align}
where we used in the last line that $(\hat{n}, \hat{h})$ is the positive equilibrium.
This ensures that $t\mapsto V(n(t),h(t)) $ decreases with time and only vanishes for $(n,h)=(\hat{n},\hat{h})$, which implies the convergence of to this equilibrium.

\paragraph{Proof of $iii)$}
The case where $\bar{h}\geq \frac{c}{p}\bar{n}$ follows along the same lines as $i)$ with $V(n,h) = h - \frac{\beta-\delta}{\gamma}\log(h) + n \frac{\rho}{p}$. 

\subsection{Proof of Proposition~\ref{prop:cv_syst_4}}
\label{proof:4SpeciesCoexist}
The proof is very similar the the case of two matching species above. We again define a Lyapunov function and show that its derivative along trajectories of the dynamical system is negative. We define the Lyapunov function as follows:
\begin{equation}
\label{eq:Lyap_4species}
V(n,h) = \sum_{i=0}^1 n_i-n^* \log(n_i) +\frac{p}{\rho} \left( h_i-h^* \log(h_i) )\, . \right)
\end{equation}
Its derivative along a trajectory is given by ($n = n_0+n_1, h=h_0+h_1$)
\begin{align*}
\frac{d}{dt} V(n(t),h(t)) 
&=\sum_{i=0}^1 (n_i-n^*)(b-d-c n(t)-ph_i(t)) + \frac{p}{\rho} (h_i(t)-h^*)(\beta-\delta-\gamma h(t)+\rho n_i(t))\\
&=\sum_{i=0}^1 (n_i-n^*)(2cn^*+ph^*-c n(t)-ph_i(t)) \\
&\qquad \qquad + \frac{p}{\rho} (h_i(t)-h^*)(2\gamma h^*-\rho n^*-\gamma h(t)+\rho n_i(t))\\	
&= -c\left(n_0 + n_1 - 2 n^\ast\right)^2-\frac{p \gamma}{\rho}\left(h_0 + h_1 - 2 h^\ast\right)^2 
\\
	&\le  0\, ,
\end{align*}
where we used that the equilibrium is positive, which follows from condition~\eqref{eq:cond_4_types}.
We then have that $t\mapsto V(n(t),h(t))$ decreases with time and only vanishes for $(n_0, n_1, h_0, h_1)=(n^*, n^*,h^*,h^*)$ which finishes the proof.

\subsection{Proof of Proposition~\ref{prop:cv_3types}}
\label{proof:3Species}

We study the dynamics of the proportion of type 0 prey denoted by $a(t) = n_0(t)/(n_0(t)+n_1(t))$. The dynamics are
\begin{equation}
    \frac{da(t)}{dt} = -p\ a(t)(1-a(t)) h_0(t)\, .
\end{equation}
As long as the type 0 predator density $h_0$ is positive, which is the case because $\beta-\delta >0$, the number of type 0 prey is therefore decreasing, which implies the result.

\subsection{Proof of Proposition~\ref{prop:syst-3-matching}}
\label{proof:3types-matching}
We study the dynamics of the proportion of type 1 predators, denoted by $\alpha(t) = h_1(t)/(h_0(t)+h_1(t))$. The dynamics are
\begin{equation}
    \frac{d\alpha(t)}{dt} = -\rho\ \alpha(t)(1-\alpha(t)) n_0(t)\, .
\end{equation}
This proportion is decreasing and converges to zero if $n_0(t)$ is positive for all times $t$. 
To show this, we note that because condition~\eqref{eq:cond_equilibria_matching} holds, the matching equilibrium exists. This implies that $n_0(t)>0$ for all $t\geq 0$ because the predator population $h_0$ is bounded from above by the following process:
\begin{equation}
    \frac{dh_0}{dt} = h_0 (\beta-\delta-\gamma(h_0+h_1) + \rho n_0) \le h_0 (\beta-\delta-\gamma h_0 + \rho n_0) = \frac{dh^+}{dt}\, ,
\end{equation}
where the right-hand side is the predator dynamics of the matching species system in Prop.~\ref{prop:cv_sys_matching}. 
Therefore $n_0(t)>0$ and thus the proportion of type 1 predators, $\alpha$, decreases monotonously for all $t\ge 0$, which concludes the proof.

\subsection{Proof of Proposition \ref{prop:syst-3}}
\label{proof:PredProportion}

\paragraph{Proof of $i)$}
The first result derives from the Jacobian matrix of the system \eqref{eq:syst-3} at the equilibrium $(0, \alpha \bar{h}, (1-\alpha)\bar{h})$, which equals
$$\begin{pmatrix}
b-d-\alpha_0 p \bar{h}&0&0\\
-\alpha_0 \bar{h}\beta&-(\beta-\delta)\alpha_0&-(\beta-\delta)\alpha_0\\
0&-(\beta-\delta)(1-\alpha_0)&-(\beta-\delta)(1-\alpha_0)
\end{pmatrix}$$
The eigenvalues are $b-d-\alpha_0 p \bar{h}$, $-(\beta-\delta)$ and $0$, which leads to the result. 
\paragraph{Proof of $ii)$}
Let us first prove the convergence of the solution as time goes to infinity.\\
Let us first consider the dynamics of the proportion of type $0$ predators which solves \eqref{eq:alpha_t}:
$$\frac{d\alpha(t)}{dt}=\rho n(t) \alpha(t)(1-\alpha(t))\ge 0.$$
As a consequence $\alpha(t$) increases and is bounded by $1$, and thus converges for $t$ to infinity towards $\alpha_\infty$. 
In the limit we have either $\alpha_\infty\in\{0,1\}$ or $\lim_{t\to\infty}n(t)=0$. 
However, because $\alpha(0)>0$, the convergence to $0$ is impossible.
Moreover, since $\hat{n}<0$, the equilibrium $(\hat{n},\hat{h},0)$ does not exist.
This entails that $\lim_{t\to\infty}n(t)=0$ and $\alpha_\infty\in]0,1]$.\\
As a consequence, the solution converges to $(0, \alpha_\infty\bar{h},(1-\alpha_\infty)\bar{h}) $ and we deduce from point $i)$ that necessarily, $\alpha_\infty \geq \alpha_c$. 

Moreover, since $t \mapsto \alpha(t)$ is increasing, we can write, with an abuse of notation $n_0(\alpha) = n_0(\alpha(t))$ and $h_0(\alpha) = h_0(\alpha(t))$, where
\begin{equation}\label{eq:phase_pred_prey}
    \left\lbrace
    \begin{aligned}
    & \frac{d n_0(\alpha)}{d \alpha} = \frac{p\bar{h}}{\rho} \ \frac{\alpha_c-\alpha}{\alpha(1-\alpha)} - \frac{c}{\rho}\ \frac{n_0(\alpha)}{\alpha(1-\alpha)} - \frac{p}{\rho}\ \frac{h_0(\alpha)-\bar{h}}{1-\alpha}\,, \\
    & \frac{d h_0(\alpha)}{d \alpha} = \frac{\gamma}{\rho} \ \frac{h_0(\alpha) (\bar{h} - h_0(\alpha))}{n(\alpha) \alpha(1-\alpha)} + \frac{ h_0(\alpha)}{1-\alpha}. 
    \end{aligned}
    \right.
\end{equation}

It  follows from \eqref{eq:phase_pred_prey} that the flow $(n_0(\alpha),h_0(\alpha))$ is continuously differentiable and that furthermore $\partial_\alpha n_0(\alpha_\infty)\neq 0$ if $\alpha_\infty \neq \alpha_c$ and that $\partial_\alpha h_0(\alpha_\infty)\neq 0$ for $\alpha_\infty\in (0,1)$. Setting $f(\alpha,n_0(0),\alpha_0)=(n_0(\alpha),h_0(\alpha)-\bar{h})$, the implicit function theorem implies the existence of a smooth function $\kappa_0(n_0(0),\alpha_0)=\alpha_\infty$. 

To apply this reasoning, it remains to show that $\alpha_\infty \neq \alpha_c$.  

Suppose that $\alpha_\infty = \alpha_c$, and let us obtain a contradiction.
In view of \eqref{eq:phase_pred_prey}, we see that if $\alpha_\infty=\alpha_c$ then $\left.\frac{d n}{d \alpha}\right|_{\alpha = \alpha_c} = 0$.
Since $n_0(t) \geq 0$,  we then should have $\left.\frac{d^2n}{d\alpha^2}\right|_{\alpha = \alpha_c} \geq 0$.
However, an elementary computation using \eqref{eq:phase_pred_prey} shows that
\begin{equation*}
    \left.\frac{d^2n}{d\alpha^2}\right|_{\alpha = \alpha_c} = - \frac{p \bar{h}}{\rho} \frac{1
    \alpha_c}{ (1-\alpha_c)} < 0\, ,
\end{equation*}
which is a contradiction.
Hence $\alpha_\infty > \alpha_c$.

\paragraph{Proof of $iii)$}
We first consider the total predator population  $h(t)=h_0(t)+h_1(t)$ which solves
$$\frac{d}{dt} h(t) = h(t)\left[\beta-\delta-\gamma h(t)+\rho \alpha(t) n_0(t) \right]$$
Therefore considering our initial condition $n_0(0)>0$, we have $\frac{d}{dt} h(0)>0$ and we notice that for any positive time $t_0$ such that $h(t_0)=\bar{h}$, the derivative will be positive as well. This entails that $h(t)>\bar{h}$ for all $t>0$.

The population dynamics of the invading prey species (type 0 prey) is given by
\begin{align*}
\frac{dn_0}{dt} &= n_0(t) \left(b-d - p \bar{h} \alpha(t) - c n_0(t) - p \alpha(t) (h(t) - \bar{h})\right) \\
	&< n_0(t) \left( b-d - p \bar{h}\alpha(t)\right) \\
		&= n_0(t) p \bar{h} \left( \alpha_c - \alpha(t)\right)\, ,
	\end{align*}	 
which holds because the total predator population satisfies $h(t)>\bar{h}$.
	
Using \eqref{eq:phase_pred_prey} we deduce
\begin{equation}\label{eq:phase_prey}
	\frac{dn}{d\alpha} = \frac{p\bar{h}}{\rho} \ \frac{\alpha_c-\alpha(t)}{\alpha(t)(1-\alpha(t))} - \frac{c}{\rho}\ \frac{n(\alpha(t))}{\alpha(t)(1-\alpha(t))} - \frac{p}{\rho}\ \frac{h(\alpha(t))-\bar{h}}{1-\alpha(t)} \leq \frac{p\bar{h}}{\rho}\ \frac{\alpha_c-\alpha(t)}{\alpha(t)(1-\alpha(t))}\, . 
	\end{equation}
These dynamics are illustrated by the red line in Fig.~\ref{fig:fcurve}.

Integrating this inequality from $t=0$ to $+\infty$ and using the fact that $\alpha(t)$ increases, we obtain
\begin{equation}
\label{eq:maj1} -\frac{\rho}{p\bar{h}}n_0(0) \leq \int_{\alpha_0}^{\alpha_\infty} \frac{\alpha_c-\alpha'}{\alpha'(1-\alpha')} d\alpha' \, ,
\end{equation}
where we suppressed the time-dependence of the frequency of type 0 predators. The integral explicitly solves as
\begin{equation}
\label{eq:maj2} \int_{\alpha_0}^{\alpha_\infty} \frac{\alpha_c-\alpha'}{\alpha'(1-\alpha')}d\alpha' = f(\alpha_0)-(\alpha_\infty)\, ,
\end{equation}
where $f(\alpha) = - \alpha_c \log(\alpha) - (1-\alpha_c)\log(1-\alpha)$. The function $f$ and the right-hand side of Eq.~\eqref{eq:maj1} are shown as the black and blue-dashed lines in Fig.~\ref{fig:fcurve}.

Combining \eqref{eq:maj1} and \eqref{eq:maj2} we get
	\[ f(\alpha_\infty) \leq f(\alpha_0) + \frac{\rho}{p\bar{h}}n_0(0)\,.\]
A simple computation gives that for all $\alpha<\alpha_c$, we have
\begin{align*}
f(1-\alpha)-f(\alpha) &= -(2\alpha_c - 1) (\log(1-\alpha)-\log(\alpha)) \\
&= (1-2\alpha_c) \log\left(\frac{1-\alpha}{\alpha}\right) \\
	&> (1-2\alpha_c) \log\left(\frac{1-\alpha_c}{\alpha_c}\right) \,,
\end{align*}
which is positive if and only if $\alpha_c < 1/2$, which is true by assumption.

We therefore find that
\[ f(\alpha_\infty)  < f(1-\alpha)-(1-2\alpha_c)\log\left(\frac{1-\alpha_c}{\alpha_c}\right)+ \frac{\rho}{p\bar{h}}n_0(0)\, .\]
Letting $n_0(0)$ tend to $0$ we obtain
\[ f(\kappa_0(\alpha_0))  \leq f(1-\alpha_0)-(1-2\alpha_c)\log\left(\frac{1-\alpha_c}{\alpha_c}\right)\, <f(1-\alpha_0) ,\]
which then implies that  $\kappa_0(\alpha_0)< 1-\alpha_0$ since $f$ is monotonously increasing on $[\alpha_c,1]$, and both $\kappa_0(\alpha_0)$ and $1-\alpha_0$ are larger than $\alpha_c$ .
This concludes the proof.

\begin{figure}[t]
	\centering
	\includegraphics[width=.75\textwidth]{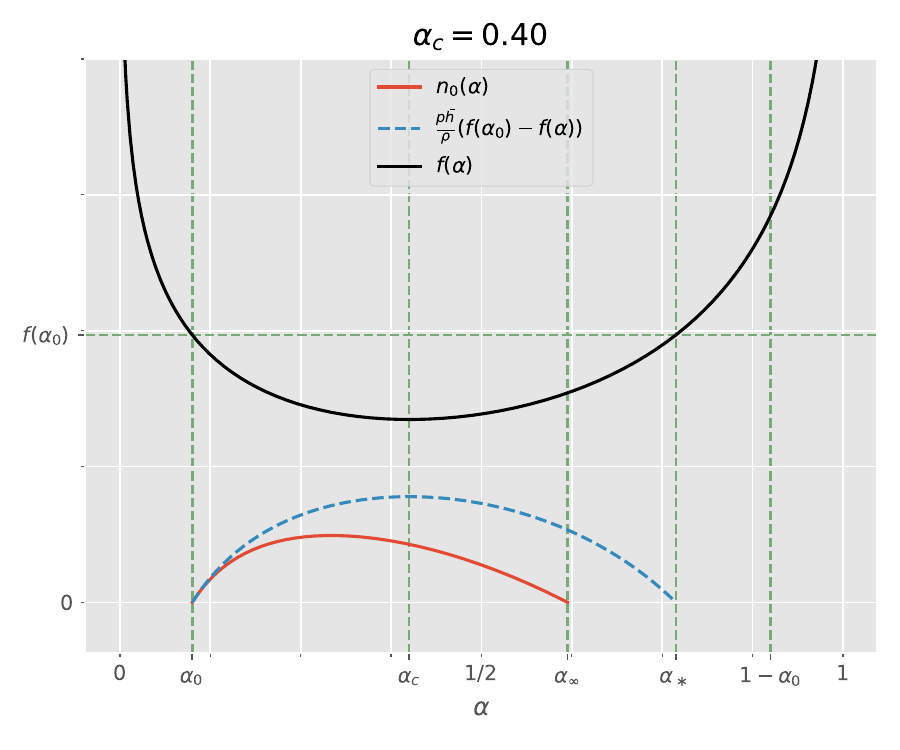}
	\caption{\textbf{Function $f$ as a function of predator proportion $\alpha$.}}
	\label{fig:fcurve}
\end{figure}

\section{Proofs associated with the stochastic behaviours}
\label{app:sto}
For sake of convenience, we shall use the following notations. When considering the non-rescaled process in $\N^4$, we write $\zbf = (N_0,N_1,H_0,H_1)$. We will note the $K$ dependency when considering the scaled version , and recall \eqref{eq:notation}
\begin{equation*}
\zbf^{K} = (n_0^K,n_1^K,h_0^K,h_1^K)  = \left(\frac{N_0}{K},\frac{N_1}{K},\frac{H_0}{K^m}, \frac{H_1}{K^m} \right)
\end{equation*}
Let us recalm the infinitesimal generator $\mathcal{L}_{sc}^K$ associated with $\mathbf{Z}^K_{sc}$ defined in \eqref{eq:gen_sc}, for all measurable functions $f:\R^4\to \R$
\begin{align*}
\mathcal{L}^K_{sc}f(\zbf^K)&=
\sum_{i=0}^1 bn_i^K(1-v_K) K(f(\zbf^K +\frac{1}{K} e_{n_i})-f(\zbf^K))  +  bn_i^K v_K K (f(\zbf +\frac{1}{K} e_{n_{1-i}})-f(\zbf^K))\nonumber\\
& + n_i^K K(d+c (n_i^K+n_{1-i}^K) + p h_i^K) (f(\zbf^K-\frac{1}{K}e_{n_i})-f(\zbf^K))\nonumber\\
&+   h_i^K(\beta+\rho n_i^K)(1-\vartheta_K) K^m(f(\zbf^K +\frac{1}{K^m} e_{h_i})-f(\zbf^K))\nonumber\\
&+  h_i^K (\beta+\rho n_i^K)\vartheta_K K^m (f(\zbf^K +\frac{1}{K^m} e_{h_{1-i}})-f(\zbf^K))\nonumber\\
& + h_i^K K^m(\delta+\gamma (h_i^K+h_{1-i}^K)) (f(\zbf^K-\frac{1}{K^m}e_{h_i})-f(\zbf^K))
\end{align*}
where $(e_{n_0}, e_{n_1}, e_{h_0}, e_{h_1})$ stands for the canonical basis of $\R^4$.
We will use at several places the following decomposition for $f$ a measurable and $\mathcal{C}^1$ function
\begin{equation}
    \label{eq:generator_approx}
\mathcal{L}^K_{sc}f(\zbf^K) = F(\zbf^K) \cdot \nabla f(\zbf^K) + \mathcal{R}^K_1f(\zbf^K)+\mathcal{R}^K_2(\zbf^K)
\end{equation}
where $F : \R_+^4 \to \R^4$ is the direction of the flow of the associated deterministic system defined for all $(n_0,n_1,h_0,h_1)\in(\R_+)^4$ by \eqref{eq:syst-4}:
\begin{align*}
F((n_0, n_1, h_0, h_1)) &= \begin{pmatrix}
n_0 (b-d - c(n_0+n_1) - ph_0) \\
n_1 (b-d - c(n_0+n_1) - ph_1) \\
h_0 (\beta-\delta - \gamma (h_0+h_1) + \rho n_0) \\
h_1 (\beta-\delta - \gamma (h_0+h_1) + \rho n_1)
\end{pmatrix},
\end{align*}
and $\mathcal{R}^K_1f(\zbf^K)$ and $\mathcal{R}^K_2(\zbf^K)$ are rest terms involved in neglecting small effects of mutations and the Taylor expansion of $f$. The detailed decomposition will be stated in the proofs below.

\subsection{Proof of Proposition \ref{prop:stability_eq}}
\label{app:proof_lyap}
In this proof we use the Lyapunov function associated with the two matching type population in order to prove that the stochastic birth and death process will remain close to the equilibrium as long as small populations remain small.\\
Recall that for $ \varepsilon > 0 $ and $ \eta > 0 $, we define
\begin{multline*}
	\theta^K_1 := \inf \Bigg\lbrace t \geq 0 : \prey_1(t) > \varepsilon \eta K \text{ or } \pred_1(t) > \varepsilon \eta K^m \\ \left. \text{ or } \left| \frac{1}{K}\prey_0(t) - \hat{n} \right| >  \varepsilon \text{ or } \left| \frac{1}{K^m} \pred_0(t) - \hat{h} \right| >  \varepsilon \right\rbrace.
\end{multline*}

We recall the Lyapunov function defined in \eqref{eq:Lyap_2species} associated with the deterministic system \eqref{eq:2-matching},  $\forall (n,h)\in(\R_+)^2$
\begin{align*}
	V(n,h) := \frac{\hat{n}}{p} \left( \frac{n}{\hat{n}} - 1 - \log\left( \frac{n}{\hat{n}} \right) \right) + \frac{\hat{h}}{\rho} \left( \frac{h}{\hat{h}} - 1 - \log \left( \frac{h}{\hat{h}} \right) \right).
\end{align*}
We note that $ V(n,h) \geq 0 $ and $ V(n,h) = 0 $ if and only if $ n = \hat{n} $ and $ h = \hat{h}$.		
Fix $ \nu \in (0, \hat{n} \wedge \hat{h}) $  and let $ \mathcal{V}_\nu $ be defined as
\begin{align*}
\mathcal{V}_\nu := \lbrace n > 0, h > 0 : | n - \hat{n} | \leq \nu \text{ and } | h - \hat{h} | \leq \nu \rbrace. 
\end{align*}
By convexity of $V$, for all $\nu$, there exists a constant $\newC{convex}^\nu$ such that
\begin{equation}
\label{eq:Vconvex}
(n,h)\notin \mathcal{V}_\nu \quad \Rightarrow \quad V(n,h)\ge \C{convex}^\nu((n-\hat{n})^2+(h-\hat{h})^2)\ge \C{convex}^\nu \nu^2.
\end{equation}
As a consequence
\begin{multline} \label{event_inclusion}
	\lbrace | \prey_0(t_K \wedge \theta^K_1) - n^* K | >  \varepsilon K \text{ or } | \pred_0(t_K \wedge \theta^K_1) - h^* K^m | >  \varepsilon K^m \rbrace \\
	\subset \left\lbrace V\left( \frac{\prey_0(t_K \wedge \theta^K_1)}{K}, \frac{\pred_0(t_K \wedge \theta^K_1)}{K^m} \right) > \C{convex}^\varepsilon \varepsilon^2 \right\rbrace.
\end{multline}
But, by the Markov inequality, for any $\lambda$
\begin{equation}
\label{eq:MarkovV}
\P\left(V\left( \frac{\prey_0(t_K \wedge \theta^K_1)}{K}, \frac{\pred_0(t_K \wedge \theta^K_1)}{K^m} \right) >\C{convex}^\varepsilon \varepsilon^2\right) \leq \frac{\E\left[e^{\lambda V\left( \frac{\prey_0(t_K \wedge \theta^K_1)}{K}, \frac{\pred_0(t_K \wedge \theta^K_1)}{K^m} \right)}-1\right]}{e^{ \C{convex}^\varepsilon \varepsilon^2 \lambda}-1}.
\end{equation}
For $ \lambda > 0 $, we define $ G_\lambda : \R^4 \to \R_+ $ by
\begin{align*}
	G_{\lambda}(\zbf) := \exp\left( \lambda V\left(n_0, h_0 \right) \right) \quad \text{ for } \zbf=(n_0,n_1,h_0,h_1)\,.
\end{align*}

The rest of the proof is devoted to obtain an exponential bound on $\E[G_\lambda(\mathbf{Z}^K(t_K\wedge \theta^K_1))]$.
Recall that $\mathcal{L}^K_{sc}$, the infinitesimal generator of $\mathbf{Z}^K_{sc}$, is given in \eqref{eq:gen_sc}. Similarly as in \eqref{eq:generator_approx}, let us write for $\zbf^K$ as in \eqref{eq:notation}
\begin{align}
\label{eq:dec_generator}
	\mathcal{L}^K_{sc} G_\lambda (\zbf^K) &=\mathcal{L}^{match}G_{\lambda}(\zbf^K) +  R_1^KG_{\lambda}(\zbf^K) + R_2^KG_{\lambda}(\zbf^K),
	\end{align}
where $\mathcal{L}^{match}G_{\lambda}(\zbf)$ is the generator associated with the two matching type deterministic system \eqref{eq:2-matching} defined for all $\zbf=(n_0,n_1,h_0,n_1)\in(\R_+)^4$ by
\begin{align*}
    \mathcal{L}^{match}G_{\lambda}(\zbf) &=
 n_0(b-d-cn_0-ph_0) \partial_{n_0} G_{\lambda}(\zbf) + \left( \beta + \rho n_0-\delta -\gamma h_0 \right) h_0 \partial_{h_0}G_{\lambda}(\zbf)\\
&=
 n_0(b-d-cn_0-ph_0) \partial_n V\left( n_0,h_0 \right)\lambda e^{\lambda V\left( n_0,h_0 \right)} \\
&\quad + \left( \beta + \rho n_0-\delta -\gamma h_0 \right) h_0 \partial_h V\left(n_0,h_0 \right) \lambda e^{\lambda V\left( n_0,h_0 \right)} 
\end{align*}
and the two rest terms: the first term $ R_1^KG_{\lambda}(\zbf^K)$ comes from the small order terms in the birth and death rates of the $ N_0 $ and $ H_0 $ populations (which either involve mutations or the competition pressure resulting from the - small - $ N_1 $ and $ H_1 $ populations)
\begin{align}
R_1^KG_{\lambda}(\zbf^K)&= b v_K (n_1^K - n_0^K) K\left( G_\lambda(\zbf^K+\frac{1}{K} e_{n_0})-G_\lambda(\zbf^K)\right)\nonumber\\
&	+ c n_1^K n_0^K K \left( G_\lambda(\zbf^K-\frac{1}{K} e_{n_0})-G_\lambda(\zbf^K)\right) \nonumber\\
&	+ \vartheta_K \left( \left( \beta + \rho n_1^K \right) h_1^K - \left( \beta + \rho n_0^K \right) h_0^K  \right) K^m \left( G_\lambda(\zbf^K+\frac{1}{K^m} e_{h_0})-G_\lambda(\zbf^K)\right)\nonumber\\
&	+ \gamma h_1^K h_0^K K^m \left( G_\lambda(\zbf^K-\frac{1}{K^m} e_{h_0})-G_\lambda(\zbf^K) \right)\,.
\label{eq:R1}
\end{align}

The second term $ R_2^KG_{\lambda}(\zbf^K) $ comes from the Taylor expansion of the function $G_\lambda$ used to approximate the jump terms in the generator.
It ready
\begin{align*}
	R_2^KG_{\lambda}(\zbf^K)&= b n_0^K K \Phi_\lambda(\lambda,\zbf^K +\frac{1}{K} e_{n_i}),\zbf^K) 
	+ \left( d + c n_0^K + p h_0^K \right)n_0^K K \Phi_\lambda\left( \zbf^K-\frac{1}{K}e_{n_0}, \zbf^K \right) \\
&	+ \left( \beta + \rho n_0^K\right) h_0^K K^m \Phi_\lambda\left(\zbf^K +\frac{1}{K^m} e_{h_{0}},\zbf^K  \right) 
	+ \left( \delta + \gamma h_0^K \right) h_0^K K^m\Phi_\lambda\left(\zbf^K -\frac{1}{K^m} e_{h_{0}},\zbf^K\right),
\end{align*}
where  $ \Phi_\lambda $ is the first order Taylor expansion of $G_\lambda$ :
\begin{align*}
	&\Phi_\lambda(\lambda, \zbf^K,\mathbf{w}^K ) = G_\lambda(\zbf^K)-G_\lambda(\mathbf{w}^K) -\langle \nabla G_\lambda(\mathbf{w}^K), \zbf^K-\mathbf{w}^K  \rangle.
\end{align*}

Let us first rearrange the first terms in $ \mathcal{L}^{match}G_\lambda $ as follows.
Replacing $ \partial_n V(n,h) $ and $ \partial_h V(n,h) $ by their expressions, we obtain with similar computations as in \eqref{eq:diff_lyap_2m}
\begin{align*} 
\mathcal{L}^{match} G_{\lambda} (\zbf^K) = - \left[ \frac{c}{p} \left( n_0^K - \hat{n} \right)^2 + \frac{\gamma}{\rho} \left( h_0^K - \hat{h} \right)^2 \right] \lambda\, G_\lambda(\zbf^K).
\end{align*}
Using the regularity of $V$ there exist constants 
$\newC{x2sup}^\nu > 0 $ (depending on $ \nu $) such that, for any $ (n,h) \in \mathcal{V}_\nu $,
\begin{align} \label{quad_approx_V}
	V(n,h) \leq \C{x2sup}^\nu \left( |n-\hat{n}|^2 + |h-\hat{h}|^2 \right).
\end{align}
Hence there exists a constant $ \newC{gronwall}^\nu > 0 $ such that, if $\zbf^K$ is chosen such that $ \left(n_0^K,h_0^K \right) \in \mathcal{V}_\nu $,
\begin{align}\label{generator_on_V}
\mathcal{L}^{match} G_{\lambda} (\zbf^K)  \leq - \C{gronwall}^\nu V\left( n_0^K,h_0^K\right) \lambda\, G_\lambda(\zbf^K).
\end{align}
Note that $\C{gronwall}^\nu=min(c/p,\gamma/\rho)/\C{x2sup}^\nu$.

We now control the rest terms. Using different Taylor expansions, we prove the following upper bound. The proof is postponed to the end of the section.
\begin{lemma}
\label{lem:rest_term}
    There exist constants $C$ and $C'$ independent of $K$ such that for any $K$ large enough and $t\in[0,\theta_1^K]$,
    \begin{align} \label{bound_rest}
	|R^K_1G_\lambda(\mathbf{Z}^K(t))|+|R^K_2G_\lambda(\mathbf{Z}^K(t))|\leq \, \left(\eta \lambda+\frac{\lambda^2}{K^{m\wedge1}}\right)\Psi\left( \frac{\lambda}{K^{m \wedge 1}} \right) G_{\lambda}(\mathbf{Z}^K(t)),
    \end{align}
    where $\Psi(x) := C\frac{e^{C'x}-1}{x}$.
\end{lemma}

Let us set
\begin{align*}
	M^K(t) := G_{\lambda}(\mathbf{Z}^K(t\wedge\theta^K_1)) -  G_{\lambda}(\mathbf{Z}^K(0)) - \int_{0}^{t \wedge \theta^K_1} \mathcal{L}^K_{sc}  G_{\lambda}(\mathbf{Z}^K(s))ds.
\end{align*}
Then $ (M^K(t), t \geq 0) $ is a martingale, since the function $G_{\lambda}$ is bounded on a bounded set and $t\mapsto \mathbf{Z}^{K}(t\wedge \theta_1^K)$ remains almost surely in a bounded set. 
As a result,
\begin{align*}
	\E\left[ G_{\lambda}(\mathbf{Z}^K(t\wedge\theta^K_1))\right] = \E\left[ G_{\lambda}(\mathbf{Z}^K(0))\right] + \E\left[{ \int_{0}^{t \wedge \theta^K_1} \mathcal{L}^K_{sc}  G_{\lambda}(\mathbf{Z}^K(s))}\right].
\end{align*}
Combining  \eqref{eq:dec_generator} with \eqref{generator_on_V}
we have 
\begin{align*}
	\E\left[ G_{\lambda}(\mathbf{Z}^K(t\wedge\theta^K_1))\right] &\le \E\left[ G_{\lambda}(\mathbf{Z}^K(0))\right] - \E\left[ \int_{0}^{t \wedge \theta^K_1}  \C{gronwall}^\varepsilon V\left( \frac{\prey_0(s)}{K}, \frac{\pred_0(s)}{K^m}\right) \lambda G_\lambda(\mathbf{Z}^K(s))\right]\\
	&\quad+  \E\left[ \int_{0}^{t \wedge \theta^K_1}	|R^K_1G_\lambda(\mathbf{Z}^K(s))|+|R^K_2G_\lambda(\mathbf{Z}^K(s))|\right].
\end{align*}
Using Lemma~\ref{lem:rest_term}, we see that
\begin{multline} \label{bound_expectation}
	\E\left[ G_\lambda(\mathbf{Z}^K(t\wedge \theta^K_1))-1\right] + \C{gronwall}^\varepsilon \lambda \E\left[ \int_{0}^{t \wedge \theta^K_1} V\left( \frac{\prey_0(s)}{K}, \frac{\pred_0(s)}{K^m} \right) G_\lambda(\mathbf{Z}^K(s)) ds \right] \\
	\leq \E\left[G_\lambda(\mathbf{Z}^K(0))-1\right] + \left( \eta \lambda + \frac{\lambda^2}{K^{m \wedge 1}} \right) \Psi\left( \frac{\lambda}{K^{m \wedge 1}} \right) \E\left[ \int_{0}^{t \wedge \theta^K_1} G_\lambda(\mathbf{Z}^K(s)) ds \right].
\end{multline}
Thus, if we set (temporarily fixing $ K $ and $ t $)
\begin{align*}
	A(\lambda) := \E\left[ \int_{0}^{t \wedge \theta^K_1} G_\lambda(\mathbf{Z}^K(s)) ds \right],
\end{align*}
we obtain from \eqref{bound_expectation} 
\begin{align*}
	\frac{d A(\lambda)}{d\lambda} \leq \frac{-1}{\C{gronwall}^\varepsilon}\E\left[ G_\lambda(\mathbf{Z}^K(t\wedge \theta^K_1))-1\right] + \frac{G_\lambda(\mathbf{Z}^K(0))-1}{\C{gronwall}^\varepsilon \lambda} + \frac{1}{\C{gronwall}^\varepsilon} \left( \eta + \frac{\lambda}{K^{m \wedge 1}} \right) \Psi\left( \frac{\lambda}{K^{m \wedge 1}} \right) A(\lambda).
\end{align*}
Then, using $ G_\lambda(\zbf) \geq 1 $ and \eqref{quad_approx_V} with $\mathbf{Z}^K(0)\in\mathcal{V}_\nu$ 
\begin{align*}
	\frac{d A(\lambda)}{d\lambda} \leq \frac{e^{2 \lambda \C{x2sup}^\nu \nu^2}-1}{\C{gronwall}^\varepsilon \lambda} + \frac{1}{\C{gronwall}^\varepsilon} \left( \eta + \frac{\lambda}{K^{m \wedge 1}} \right) \Psi\left( \frac{\lambda}{K^{m \wedge 1}} \right) A(\lambda).
\end{align*}
Integrating and using Gronwall's inequality, we obtain
\begin{align*}
	A(\lambda) \leq \left( A(0) + \frac{1}{\C{gronwall}^\varepsilon} \int_{0}^{\lambda} \frac{1}{u} (e^{2 u \C{x2sup}^\nu \nu^2}-1) du \right) \exp \left( \frac{1}{\C{gronwall}^\varepsilon} \int_{0}^{\lambda} \left( \eta + \frac{u}{K^{m \wedge 1}} \right) \Psi\left(\frac{u}{K^{m \wedge1}}\right) du \right) .
\end{align*}
Since $ u \mapsto \frac{1}{u}(e^{u C}-1) $ and $\Psi$ are  increasing,
and $ A(0) \le  t $, we obtain
\begin{align*}
	A(\lambda) \leq \left( t  + \frac{1}{\C{gronwall}}(e^{2 \lambda \C{x2sup}^\nu \nu^2}-1) \right) \exp\left( \frac{1}{\C{gronwall}^\varepsilon} \left( \eta \lambda + \frac{\lambda^2}{2 K^{m \wedge 1
	}} \right) \Psi\left( \frac{\lambda}{K^{m \wedge 1}} \right) \right).
\end{align*}
Plugging this into the right hand side of \eqref{bound_expectation} (and noting that the second term on the left hand side is non-negative), we obtain 
\begin{multline} \label{maxi_bound}
	\E\left[ G_\lambda(\mathbf{Z}^K(t\wedge \theta^K_1))-1\right] \le - \C{gronwall}^\varepsilon \lambda \E\left[ \int_{0}^{t \wedge \theta^K_1} V\left( \frac{\prey_0(s)}{K}, \frac{\pred_0(s)}{K^m} \right) G_\lambda(\mathbf{Z}^K(s)) ds \right] \\
\qquad\qquad\qquad\qquad\qquad + \E\left[G_\lambda(\mathbf{Z}^K(0))-1\right] + \left( \eta \lambda + \frac{\lambda^2}{K^{m \wedge 1}} \right) \Psi\left( \frac{\lambda}{K^{m \wedge 1}} \right) A(\lambda)\\
 \leq e^{2 \lambda \C{x2sup}^\nu \nu^2}-1	+ U(\lambda, K)\left( t + \frac{1}{\C{gronwall}^\varepsilon}(e^{2 \lambda \C{x2sup}^\nu \nu^2}-1) \right)\exp\left(\frac{U(\lambda,K)}{\C{gronwall}^\varepsilon} \right),
\end{multline}
where
\[ U(\lambda, K) = \left( \eta \lambda + \frac{\lambda^2}{ K^{m \wedge 1
	}} \right) \Psi\left( \frac{\lambda}{K^{m \wedge 1}} \right). \]

Setting $ \lambda_K = a K^{m \wedge 1} $, where $ a > 0 $ will be chosen later, we then have 
\[U( \lambda_K, K) = \lambda_K (\eta +a) \Psi( a). \] 
Coming back to \eqref{eq:MarkovV}, we have 
\begin{align*} 
\P&\left(V\left( \frac{\prey_0(t_K \wedge \theta^K_1)}{K}, \frac{\pred_0(t_K \wedge \theta^K_1)}{K^m} \right) >\C{convex}^\varepsilon \varepsilon^2\right) \\
&\le 	\frac{\E\left[G_{\lambda_K}(\mathbf{Z}^K(t\wedge \theta^K_1))-1\right]}{e^{\C{convex}^\varepsilon\varepsilon^2 \lambda_K}-1} \\
&\leq \frac{ e^{2 \lambda_K \C{x2sup}^\nu \nu^2}-1	}{e^{\C{convex}^\varepsilon\varepsilon^2 \lambda_K}-1} + \frac{\lambda_K (\eta +a) \Psi( a) }{{e^{ \C{convex}^\varepsilon \varepsilon^2 \lambda_K}-1}}\left( t + \frac{1}{\C{gronwall}^\varepsilon}(e^{2 \lambda_K \C{x2sup}^\nu \nu^2}-1) \right)\exp\left(\frac{\lambda_K (\eta +a)}{\C{gronwall}^\varepsilon} \Psi( a)  \right).
\\
\end{align*}

The first term of the r.h.s can be handled by choosing $ \nu $ small enough so that $ 2 \C{x2sup}^\nu \nu^2 <  \C{convex}^\varepsilon \varepsilon^2 $. With the definitions of both constants $\C{x2sup}^\nu$ and $\C{convex}^\varepsilon$ it is sufficient to choose $\nu \le \varepsilon/\sqrt{2}$.
For the second term note that for all $t \leq e^{\zeta K^{m \wedge 1}} $ it can be bounded by 
\[\newC{fin}^\varepsilon\lambda_K (\eta +a) \Psi( a)  
\exp\left(aK^{m\wedge 1} [\frac{\zeta}{a} \vee  2 \C{x2sup}^\nu \nu^2  + (\eta+a)\Psi(a) -\C{convex}^\varepsilon \varepsilon^2] \right)
\]
for $\C{fin}^\varepsilon$ a positive constant depending only on $\varepsilon$.
In order to conclude, it is therefore sufficient to choose $a$ and $\eta$ small enough such that the exponential term vanishes as $K\to\infty$. As a consequence, there exists a positive $r>0$ such that 
\begin{align*} 
\P&\left(V\left( \frac{\prey_0(t_K \wedge \theta^K_1)}{K}, \frac{\pred_0(t_K \wedge \theta^K_1)}{K^m} \right) >\C{convex}^\varepsilon \varepsilon^2\right) \le C^\varepsilon  K^{1\wedge m} e^{-r K^{1\wedge m}},
\end{align*}
which concludes the proof.

\paragraph{Proof of Lemma \ref{lem:rest_term}} 
Let us first handle $R^K_1G_\lambda(\zbf^K)$ defined in \eqref{eq:R1}. 
Remark that 
\begin{align*}
    G_\lambda(\zbf^K+\Delta)-G_\lambda(\zbf^K) &= e^{\lambda V(\zbf^K+\Delta) }- e^{\lambda V(\zbf^K)}\\
    &= G_\lambda(\zbf^K) \left(e^{\lambda (V(\zbf^K+\Delta)-V(\zbf^K))  }-1 \right)\\
    &=G_\lambda(\zbf^K) \lambda (V(\zbf^K+\Delta)-V(\zbf^K)) \Psi_1(\lambda (V(\zbf^K+\Delta)-V(\zbf^K)) )
\end{align*}
where $\Psi_1(x)=(e^x-1)/x$ is a continuous, increasing function on $\R$, and therefore bounded on any compact interval.
Furthermore, there exists $\newC{taylor1V}^\nu > 0$ such that, for any $ (n,h), (n',h') \in \mathcal{V}_\nu $,
\begin{align} \label{taylor_1_V}
	| V(n', h') - V(n,h) | \leq \C{taylor1V}^\nu \left( |n'-n| + |h'-h| \right),
\end{align}
As a consequence, using the fact that $\forall t\in[0,\theta^K_1]$
 $ N_1^K(t) \leq \eta \varepsilon K $, $ N_0^K(t) \leq (\bar{n}+\varepsilon)K $, $ H_1^K(t) \leq \eta \varepsilon K^m $ and $ H_0^K(t) \leq (\bar{h} + \varepsilon) K^m $, 
and therefore 
\begin{align*}
|R^K_1&G_\lambda(\mathbf{Z}^K(t))|\le    \big| b v_K (\frac{N_1^K}{K} - \frac{N_0^K}{K}) +c \frac{N_1^K}{K}\frac{ N_0^K}{K} \big| K\left(G_\lambda(\mathbf{Z}^K(t)) \lambda\frac{\C{taylor1V}^\varepsilon }{K}\Psi(\lambda \C{taylor1V}^\varepsilon/K)\right)\nonumber\\
&	+ \Big| \vartheta_K \left( \left( \beta + \rho \frac{N_1^K}{K} \right) \frac{H_1^K}{K^m} - \left( \beta + \rho \frac{N_0^K}{K} \right) \frac{H_0^K}{K^m}  \right) +\gamma \frac{H^K_0}{K^m}\frac{H_1^K}{K^m} \Big|\  K^m \left( G_\lambda(\mathbf{Z}^K(t)) \lambda\frac{\C{taylor1V}^\varepsilon }{K^m}\Psi(\lambda \C{taylor1V}^\varepsilon/K^m)\right)\nonumber\\
&\qquad\le  C \C{taylor1V}^\varepsilon (v_K + \eta \varepsilon)   \lambda  \Psi_1\left( \C{taylor1V}^\varepsilon \frac{\lambda}{K} \right) G_\lambda(\mathbf{Z}^K(t)) \\
&\qquad\qquad	+ C \C{taylor1V}^\varepsilon (\vartheta_K + \eta \varepsilon)\lambda \Psi_1\left( \C{taylor1V}^\varepsilon \frac{\lambda}{K^m} \right) G_\lambda(\mathbf{Z}^K(t)).
\end{align*}
where the constant $C$ depends neither on $K$ nor $\varepsilon$.\\
Since $ v_K $ and $ \vartheta_K $ tend to zero as $ K \to \infty $, there exists a constant $ C > 0 $ such that, for any $ K $ large enough and $ t \in [0, \theta^K_1] $, 
\begin{align} \label{bound_R1}
|R^K_1G_\lambda(\mathbf{Z}^K(t))|\leq C^\varepsilon\, \eta \lambda\, \Psi_1\left( \varepsilon C \frac{\lambda}{K^{m \wedge 1}} \right) G_{\lambda}(\mathbf{Z}^K(t)).
\end{align}
A similar analysis can be done for $R^K_2G_\lambda(\mathbf{Z}^K(t))$. Using a second order Taylor expansion for $V$ we obtain that for all  $ (n,h), (n',h') \in \mathcal{V}_\nu $,
\begin{align} \label{taylor_2_V}
	| V(n',h') - V(n,h) - (n'-n) \partial_n V(n,h) - (h'-h) \partial_h V(n,h) | \leq \newC{taylor2V}^\nu \left( |n'-n|^2 + |h'-h|^2 \right).
\end{align}
and that for a constant $ C > 0 $ such that, for $ K $ large enough and all $ t \in [0, \theta^K_1] $,
\begin{align} \label{bound_R2}
|R^K_2G_\lambda(\mathbf{Z}^K(t))| \leq C\varepsilon \frac{\lambda^2}{K^{m\wedge 1}} \Psi_2\left( \C{taylor2V} \frac{\lambda}{K^{m \wedge 1}} \right) G_\lambda(\mathbf{Z}^K(t)),
\end{align}
where $\Psi_2(x) := \frac{e^x-1-x}{x^2}$.

The bound in Lemma \ref{lem:rest_term} then follows from the combination of \eqref{bound_R1}, \eqref{bound_R2} and the fact that $\Psi_1(x)\ge \Psi_2(x)$.

\subsection{Proof of the stability of proportions}
\label{app:proof_proportion}
Let us fix $\varepsilon>0$. 
The proof of Proposition \ref{prop:proportions} relies on a first lemma which controls the change in proportions.
\begin{lemma} 
\label{lem:proportions}
Suppose that the assumptions of Proposition~\ref{prop:accumulation_B} hold. For any $\mathcal{A}>0$, there exists $\varepsilon_0$ such that for any $\xi\in\{1/2,1\}$ and $\varepsilon \leq \varepsilon_0$,
$$ \limsup_{K \to \infty} \P \left(  U_{\varepsilon^{1/6}} 
< R_{\mathcal{A}\varepsilon} \wedge T_{\varepsilon^\xi }  \wedge T_{0 }
\right)
\leq C(\mathcal{A},\xi) \varepsilon^{1/12}, 
$$
where $C(\mathcal{A},\xi)$ is a positive constant.
\end{lemma} 
We introduce 
\begin{equation} \label{def_tau_eps} \tau_{\varepsilon,K}:=U_{\varepsilon ^{1/8}} \wedge R_{\mathcal{A}\eps}  \wedge T_{\varepsilon^\xi } \wedge T_0 .
\end{equation}
\begin{proof}
The statement of Lemma \ref{lem:proportions} is a direct consequence of the following inequality:
\begin{equation}\label{eq:min1} 
\limsup_{K \to \infty} \P \left( \sup_{t \leq U_{\varepsilon^{1/8}} \wedge R_{\mathcal{A}\varepsilon }
\wedge T_{\varepsilon ^\xi }\wedge T_0 }\left| \frac{H^K_0(t)}{H^K(t)} -
\frac{H^K_0(0)}{H^K(0)} \right|> \varepsilon ^{1/6} \right) \leq C \varepsilon ^{1/12}.
\end{equation}  

To prove \eqref{eq:min1}, we decompose the process $\frac{H^K_0(t)}{H^K(t)} $ as the sum of a square integrable martingale $M_p$
and of a finite variation process $V_p$.
In the vein of Fournier and M\'el\'eard \cite{fournier_microscopic_2004} we represent the population
process in terms of Poisson measures.\\
Let $(Q^{(\varrho)}_{i}(ds,d\theta),i \in \{0,1\}, \varrho\in \{b,d\})$ be four 
independent Poisson random measures on $\R^2_+$ with intensity $dsd\theta$ representing respectively the birth and death events of predators of types $0$ and $1$.
Let us also denote by $\tilde{Q}^{(\varrho)}_{i}(ds,d\theta):=Q^{(\varrho)}_{i}(ds,d\theta)-dsd\theta$ the associated compensated measure, for any $ \varrho \in \{b,d\}, i\in\{0,1\}$. 
We can then write for $t\geq 0$
\[
\frac{H^K_0(t)}{H^K(t)} =  \frac{H^K_0(0)}{H^K(0)}+M_p(t)+V_p(t),
\]
with $M_p$ and $V_p$ such that:

 \begin{align}\label{Mp}
M_p(t) = & \int_{0}^{t} \int_{\R_+} \mathbf{1}_{\{ \theta \leq H_0^K(s-) (\beta+\frac{\rho N_0^K(s-)}{K}) \}}
\frac{H_1^K(s)}{H^K(s-)(H^K(s-)+1)}\tilde{Q}_{0}^{(b)}(ds,d\theta) 
 \\&- \int_{0}^{t} \int_{\R_+} \mathbf{1}_{\{\theta \leq H^K_0(s-) (\delta+\frac{\gamma H^K(s-)}{K^m})\}}\frac{H_1^K(s-)}{H^K(s-)(H^K(s-)-1)}\tilde{Q}_{0}^{(d)}(ds,d\theta) \nonumber
 \\&- \int_{0}^{t} \int_{\R_+} \mathbf{1}_{\{ \theta \leq H_1^K(s-) (\beta+\frac{\rho N_1^K(s-)}{K}) \}}
\frac{H_0^K(s)}{H^K(s-)(H^K(s-)+1)}\tilde{Q}_{1}^{(b)}(ds,d\theta) \nonumber
 \\&+ \int_{0}^{t} \int_{\R_+} \mathbf{1}_{\{\theta \leq H^K_1(s-) (\delta+\frac{\gamma H^K(s-)}{K^m})\}}\frac{H_0^K(s-)}{H^K(s-)(H^K(s-)-1)}\tilde{Q}_{1}^{(d)}(ds,d\theta),  \nonumber
\end{align}
and 
\begin{align*} 
V_p(t) =& \int_{0}^{t}\left( H_0^K(s) (\beta+\frac{\rho N_0^K(s)}{K}) \right)
\frac{H_1^K(s)}{H^K(s)(H^K(s)+1)}ds
 \nonumber\\
 &- \int_{0}^{t} \left(H^K_0(s) (\delta+\frac{\gamma H^K(s)}{K^m})\right)\frac{H_1^K(s)}{H^K(s)(H^K(s)-1)}ds \nonumber
 \\
 &- \int_{0}^{t} \left(H_1^K(s) (\beta+\frac{\rho N_1^K(s)}{K}) \right)
\frac{H_0^K(s)}{H^K(s)(H^K(s)+1)}ds \nonumber
 \\
 &+ \int_{0}^{t} \left( H^K_1(s) (\delta+\frac{\gamma H^K(s)}{K^m})\right) \frac{H_0^K(s)}{H^K(s)(H^K(s)-1)}ds \nonumber,
\\ \end{align*}
which equals after simplifications
 \begin{equation}
 \label{Ap}
V_p(t)= 
\int_{0}^{t}\rho  \frac{H^K_0(t) H^K_1(t) (N^K_0(s)-N^K_1(s))}{K H^K(s)(H^K(s)+1)}ds.
\end{equation}
Using such a decomposition, we find that for $\eps$ small enough, and $C_l$ a constant to be chosen afterwards
\begin{equation} \label{ineq_MA}
\begin{aligned}
 \P \bigg( \sup_{t \leq \tau_{\varepsilon,K}}& \left|\frac{H^K_0(t)}{H^K(t)} -
\frac{H^K_0(0)}{H^K(0)} \right|> \eps^{1/6} \bigg)\\
&\leq \P \left(  \sup_{t \leq \tau_{\varepsilon,K}} \left| M_p(t) \right|>\frac{\eps^{1/6}}{2} \right) +
 \P \left(  
 \sup_{t \leq \tau_{\varepsilon,K}} \left| V_p(t)\right|  > 
\frac{\eps^{1/6}}{2} \right)\\
&\leq \P \left( \sup_{t \leq C_l \log(K) \wedge \tau_{\varepsilon,K}} |M_p(t)| > \frac{\varepsilon^{1/6}}{2} \right) + \P \left( \tau_{\varepsilon,K} > C_l \log(K) \right) +  \P \left(  
 \sup_{t \leq \tau_{\varepsilon,K}} \left| V_p(t)\right|  > 
\frac{\eps^{1/6}}{2} \right)\\
&\leq \frac{2}{\eps^{1/6}}  \E \left[ \left|M_p\Big( C_l \log(K) \wedge \tau_{\varepsilon,K}\Big)\right| \right] +
 \frac{\sqrt{2}}{\eps^{1/12}} \E \left[ \sqrt{\sup_{t \leq \tau_{\varepsilon,K}} 
 \left| V_p(t)\right|}\right] + \P \left( \tau_{\varepsilon,K} > C_l \log(K) \right)\\
&\leq \frac{2}{\eps^{1/6} } \left( \sqrt{\E\left[ M_p^2( C_l \log(K) \wedge \tau_{\varepsilon,K}) \right]} +
\sqrt{\E \left[ \sup_{t \leq\tau_{\varepsilon,K}} \left| V_p(t)\right|\right] } \right)+ \P \left( \tau_{\varepsilon,K} > C_l \log(K) \right),
\end{aligned}
\end{equation}
where we applied Doob maximal, Markov, Cauchy-Schwarz and Jensen inequalities.

To handle the first term in \eqref{ineq_MA}, we use the quadratic variation of the martingale $M_p$ which equals
\begin{align*}
\langle M_p\rangle_{C_l \log(K) \wedge \tau_{\varepsilon,K}} =& \int_{0}^{C_l \log(K) \wedge \tau_{\varepsilon,K}}\left( H_0^K(s) (\beta+\frac{\rho N_0^K(s)}{K}) \right)
\left(\frac{H_1^K(s)}{H^K(s)(H^K(s)+1)}\right)^2ds
 \nonumber\\
 &+ \int_{0}^{C_l \log(K) \wedge \tau_{\varepsilon,K}} \left(H^K_0(s) (\delta+\frac{\gamma H^K(s)}{K^m})\right)\left(\frac{H_1^K(s)}{H^K(s)(H^Ks)-1)}\right)^2ds \nonumber
 \\
 &+ \int_{0}^{C_l \log(K) \wedge \tau_{\varepsilon,K}} \left(H_1^K(s) (\beta+\frac{\rho N_1^K(s)}{K}) \right)\left(
\frac{H_0^K(s)}{H^K(s)(H^K(s)+1)}\right)^2ds \nonumber
 \\
 &+ \int_{0}^{C_l \log(K) \wedge \tau_{\varepsilon,K}} \left( H^K_1(s) (\delta+\frac{\gamma H^K(s)}{K^m})\right)\left( \frac{H_0^K(s)}{H^K(s)(H^K(s)-1)}\right)^2ds\, . \nonumber
\end{align*}
We then have
\begin{equation}\label{eq:majMpbracket}
\begin{aligned}
\langle M_p\rangle_{C_l \log(K) \wedge\tau_{\varepsilon,K}} 
 \leq C\int_0^{C_l \log(K) \wedge\tau_{\varepsilon,K}}\frac{1}{H^K(s)} ds\le C\frac{C_l\log K}{K^m}. 
\end{aligned}
\end{equation}

Let us now consider the last term in \eqref{ineq_MA}. We will couple $\tau_{\varepsilon,K}$ with the invasion time $\widehat T_\varepsilon$ of a supercritical branching process $\widehat N^{K,-}_i$ such that $\tau_{\varepsilon,K} \le \widehat T_\varepsilon$. 
Recall that before $\tau_{\varepsilon,K}$ the proportions of predators with type $0$ (resp. type $1$) is close $\alpha$ (resp. $1-\alpha$). In order to fix ideas, we assume $\alpha>\alpha_c$ and $1-\alpha<1/2<\alpha_c$, which implies that the prey with type $1$ will invade. As such, we can lower bound their populations by a branching birth and death process $\widehat N^{K,-}_1$ starting from $\lfloor K^{(x^0_1 -C_0\nu) -1}\rfloor$ with birth rate $(1-\varepsilon)b$ and death rate $d+ (p(1-\alpha) \bar{h} + 2c\varepsilon+p g(\varepsilon,\nu)$ (a similar construction is used in the main proof on page~\pageref{p:BD_construction}). Note that from the choice of $\alpha>\alpha_c$, then $r^-=(1-\varepsilon)b- d- (p(1-\alpha) \bar{h} - 2c\varepsilon-p g(\varepsilon,\nu)>0$ for $\nu$ and $\varepsilon$ small. Then since $N^K_1(t) \ge \widehat N^{K,-}_1$, we can introduce 
$$\widehat T_\varepsilon=\inf\{t\ge 0 : \widehat N^{K,-}_1(t) = \lfloor\varepsilon K\rfloor\}\ge \tau_{\varepsilon,K}.$$
Choosing $C_l> r^{-}$, we obtain from Lemma A.2 in \cite{bovier_crossing_2019} that 
\begin{equation}
    \label{eq:maj_tau}
    0= \lim_{K\to\infty} \P \left( \widehat T_\varepsilon > C_l \log(K) \right) \ge \lim_{K\to\infty} \P \left( \tau_{\varepsilon,K} > C_l \log(K) \right).
\end{equation}

Hence, it remains to bound the last expectation of \eqref{ineq_MA}.
Very similarly to \cite{coron2021emergence}, we then obtain that
\begin{equation*}
 \sup_{t\leq \tau_{\varepsilon,K}} |V_p(t)|\leq C \int_0^{\tau_{\varepsilon,K}} \frac{N^K(s)}{K}ds.
 \end{equation*}

The main idea is then to find a linear function $f$ depending only on $N_0$ and $N_1$ such that, for all $t \in [0,\tau_{\varepsilon,K})$,
$$\mathcal{L}f(N_0^K(t), N_1^K(t), H_0^K(t), H_1^K(t)) \ge N^K(t),$$
where $\mathcal{L}$ is the infinitesimal generator of $(N_0,N_1,H_0,H_1)$.
Then we would have
\begin{align}
   \E \left[ \sup_{t \leq \tau_{\varepsilon,K}} \left| V_p(t)\right|\right]\le    \E \left[\frac{1}{K} \int_0^{\tau_{\varepsilon,K}} N^K(s)ds\right] &\le \E \left[\frac{1}{K} \int_0^{\tau_{\varepsilon,K}} \mathcal{L}f(N_0^K(s), N_1^K(s)) ds\right]\nonumber \\
    &\le \E\left[\frac{1}{K}(f(N^K_0({\tau_{\varepsilon,K}}), N^K_1(\tau_{\varepsilon,K}))-f(N^K_0(0), N^K_1(0))) \right] \nonumber\\
    &\le C \varepsilon  ,\label{eq:majAp} 
\end{align}
where the last inequality follows from the linearity of the function $f$.

We will now combine \eqref{eq:majMpbracket}, \eqref{eq:majAp} and \eqref{eq:maj_tau}, with \eqref{ineq_MA} to deduce that 
\begin{align*}
  \limsup_{K\to\infty}  \P \bigg( \sup_{t \leq \tau_{\varepsilon,K}} \left|\frac{H^K_0(t)}{H^K(t)} -
\frac{H^K_0(0)}{H^K(0)} \right|> \eps^{1/6} \bigg)
&\leq \frac{C}{\varepsilon^{1/6}}  \sqrt{\varepsilon}   
\end{align*}
\smallskip\\
The end of the proof is devoted to finding such an $f$. 
Recall that $N_0^K$ is invading while $N^K_1$ decreases and consider 
$$f( N_0,N_1,H_0,H_1) = A_0 N_0 -A_1N_1,$$
with $A_0, A_1>0$ to be chosen.
Let us now apply the infinitesimal generator of $(N_0^K,N_1^K,H_0^K,H_1^K)$ to the function $f$.
\begin{equation}\label{eq:generator}
\begin{aligned}
\mathcal{L}f(N_0,N_1,H_0,H_1)&=A_0N_0 \left[ b-d-\frac{c}{K}(N_0+N_1) - \frac{p}{K^m}H_0\right]\\
&\qquad- A_1N_1 \left[ b-d-\frac{c}{K}(N_0+N_1) - \frac{p}{K^m}H_1\right]\\
\end{aligned}
\end{equation}
Let us remark that for all $t < \tau_{\eps,K}$, we have 
$$\left\lvert b-d-\frac{c}{K}(N_0^K(t)+N_1^K(t)) - \frac{p}{K^m}H_0^K(t)-  (b-d-p\alpha\bar{h} )\right\lvert \le C\eps^{1/8}   $$
and
$$\left\lvert b-d-\frac{c}{K}(N_0^K(t)+N_1^K(t)) - \frac{p}{K^m}H_1^K(t) -  (b-d-p(1-\alpha)\bar{h})\right\lvert \le C\eps^{1/8}
$$
and that moreover, 
$b-d-p\alpha\bar{h} >0$ while $(b-d-p(1-\alpha)\bar{h})<0$.
Therefore, for $t \in [0,\tau_\eps)$
\begin{align*}
\mathcal{L}f(N_0^K(t),N_1^K(t),H_0^K(t),H_1^K(t)) &\ge A_0 N_0^K(t) \left[ b-d-p\bar{h}\alpha -C\varepsilon^{1/8}\right] \\
&\qquad+ A_1 N_1^K(t) \left[\left| b-d-(1-\alpha)p\bar{h}\right| -C\varepsilon^{1/8}\right]\\
&\ge N_0^K(t) +N_1^K(t),
\end{align*} 
for $A_0>(b-d-p\bar{h}\alpha)^{-1}>0 $ and $A_1>\left| b-d-(1-\alpha)p\bar{h}\right|^{-1}>0$ and $\varepsilon$ small enough.
\end{proof}

Starting from this result, we can then control the changes in the total predator population size. 
Similarly as in \cite{coron2021emergence} Lemma 3.4 one obtains that there exists $\mathcal{A}_0$ and $\varepsilon_0$ such that for any $\varepsilon\le \varepsilon_0$, 
\[ \limsup_{k\to\infty} \P (R_{\mathcal{A}_0\varepsilon} \le U_{\varepsilon^{1/6}} \wedge T_{\varepsilon }  \wedge T_0) = 0.\]
The main idea for this result is to show that the total predator population stays between two logistic birth and death processes which will stay close to their respective stable equilibrium size $\bar{h}_+$ and $\bar{h}_-$ for an exponential time when $K$ is large, with $\bar{h}_+$ and $\bar{h}_-$ both close to $\bar{h}$.
We leave the details to the (motivated) reader.

\section{Similar results for the cases $D$ and $E$}

In this section, we present lemmas that are needed to complete the proof in cases $D$ and $E$ where prey populations are of macroscopic order and predator populations microscopic. The rest of the probabilistic argument follow a similar path as in case B and C respectively.

\subsection{Results on the dynamics system}
Lastly, we study the following two prey-one predator dynamical system
\begin{equation}\label{eq:sys-3_prey}
\left\{\begin{aligned}
&\frac{dn_0(t)}{dt}= n_0(t)(b-d-c (n_0(t)+n_1(t)))-ph_0(t))\\
&\frac{dn_1(t)}{dt}=n_1(t)(b-d-c (n_0(t)+n_1(t)) )\\
&\frac{dh_0(t)}{dt}=h_0(t)(\beta-\delta-\gamma h_0(t)) +\rho n_0(t))
\end{aligned}\right.
\end{equation}
Analogously to Prop.~\ref{prop:syst-3}, in this case we will show that if predators are non-autonomous ($\beta<\delta$), the system admits a line of equilibria of the form $(a\bar{n},(1-a)\bar{n},0)$ for $a\in[0,1]$. We will also study the dynamics of the type 0 prey proportion $a(t) = n_0(t)/(n_0(t) + n_1(t))$, which has the following dynamics:
\begin{equation}
    \frac{da(t)}{dt} = -p h_0(t) a(t)(1-a(t))\, .
\end{equation}

\begin{prop}[]\label{prop:sys-3_prey}
Assume $\beta<\delta$.
\begin{enumerate}[label=\roman*)]
	\item The Jacobian matrix at the equilibrium $(a\bar{n},(1-a)\bar{n},0)$ admits two negative eigenvalues and a null eigenvalue if and only if $a<a_c$, where $a_c$ is the critical proportion of type $0$ prey above which the type 0 predator can invade:
	\begin{equation}\label{eq:a_c_first}
		a_c = \inf\{a:\beta-\delta + \rho a \bar{n}>0\} = -\frac{\gamma \bar{h}}{\rho \bar{n}}\, .
	\end{equation}
	
	\item Consider the solution of~\eqref{eq:sys-3_prey} with initial condition $(a_0 \bar{n},(1-a_0)\bar{n}, h_0(0))$ so that $h_0(0)>0$ and $a_0>a_c$. Then the solution converges as $t\to\infty$ to $(a_{\infty}\bar{n},(1-a_\infty)\bar{n},0)$ with $a_\infty<a_c$. Additionally, the limiting proportion $a_\infty$ can be written as a function $k$ of the initial conditions with the following properties: $\lim_{h_0\to 0} k(h_0,a_0) = k_0(a_0)$ exists and $k_0(a_0)<a_c$.
	
	\item Assume that $h^*\leq 0$ or equivalently $a_c>1/2$. Then for any $a_0 > a_c$ we have that $k_0(a_0)>1-a_0$.
	
\end{enumerate}
\end{prop}
\begin{proof}

\noindent\textbf{Proof of $i)$} 
This result follows from the Jacobian matrix of the system~\eqref{eq:sys-3_prey} at the equilibrium $(a\bar{n},(1-a)\bar{n},0)$, which is:
$$\begin{pmatrix}
-(b-d)a & -(b-d)a & -pa \bar{n} \\
-(b-d)(1-a) & -(b-d)(1-a) & 0 \\
0 & 0 & \beta-\delta + \rho a \bar{n}
\end{pmatrix}$$
The eigenvalues are $-(b-d)$, $\beta-\delta+\rho a \bar{n}$ and $0$. This yields the result.

\noindent\textbf{Proof of $ii)$}
The proof of $ii)$ follows the same arguments as outlined in Prop.~\ref{prop:syst-3} ii), with the roles of prey and predators interchanged.

\noindent\textbf{Proof of $iii)$}
The proof follows along the same lines as the one of Proposition~\ref{prop:syst-3} iii). Instead of the prevalent prey population starting at proportion $a=a_0$, we will track the proportion of type 1 prey, which increases during a type 0 predator invasion attempt. We denote the type 1 prey frequency by $\tilde{a}=1-a$ and set $\tilde{a}_c = 1-a_c$. The type 1 prey frequency dynamics then read
\[ \frac{d\tilde{a}}{dt} = \tilde{a}(1-\tilde{a})ph_0\, ,\]
and the type 0 predator dynamics are 
\[ \frac{dh_0}{dt} = h_0 \left(\beta-\delta - \gamma (h_0+h_1) + \rho (1-\tilde{a}_0)(n_0+n_1)\right)\, .\]
The total prey population size dynamics are
\[ \frac{dn(t)}{dt} = \frac{d(n_0(t)+n_1(t))}{dt} = n(t) \left(b-d - c n(t) - p (1-\tilde{a})n(t)\right)\, .\]
Then the phase space dynamics are given by
\[ \frac{dh_0}{d\tilde{a}} = \frac{1}{p\tilde{a}(1-\tilde{a})}\left(\rho \bar{n} (\tilde{a}_c-\tilde{a}) - \gamma h_0 + \rho (1-\tilde{a}) \underbrace{(n-\bar{n})}_{\leq 0}\right) \leq \frac{\rho\bar{n}}{p}\ \frac{\tilde{a}_c-\tilde{a}}{\tilde{a}(1-\tilde{a})}\, .\]
Integrating yields
\[ -h_0(0) \leq -\frac{\rho \bar{n}}{p}(g(\tilde{a}_\infty) - g(\tilde{a}))\ ,\]
where $g(a) = -\tilde{a}_c \log(a) - (1-\tilde{a}_c)\log(1-a)$. Setting $f(a) = g(1-a)$ and replacing $\tilde{a}_c$ by $1-a_c$, the result then follows by the same arguments as in the proof of Prop.~\ref{prop:syst-3} iii); note that we need to show $f(a_\infty)<f(1-a)$ as before.

\end{proof}

\end{document}